\documentclass[11pt,reqno]{amsart}
\newcommand{\mysection}[1]{\section{#1}
      \setcounter{equation}{0}}

\newcommand\cbrk{\text{$]$\kern-.15em$]$}} 
\newcommand\opar{\text{\raise.2ex\hbox{${\scriptstyle | }$}\kern-.34em$($} }

\usepackage{color}
\usepackage{amsmath,amsthm,amssymb,amsfonts,enumerate,color,enumerate}
\usepackage[pdftex]{graphicx}

\usepackage{verbatim}
\usepackage[utf8]{inputenc}
\usepackage[spanish,USenglish]{babel} % espanol, ingles
\usepackage{color}
\usepackage{tikz}
\usepackage[hidelinks]{hyperref}
\usepackage{mathtools}
\usepackage{bbm}
\usepackage{mathabx}
\allowdisplaybreaks

\DeclareMathOperator*{\esssup}{ess\,sup}

\newtheorem{theorem}{Theorem}[section]
\newtheorem{lemma}[theorem]{Lemma}
\newtheorem{proposition}[theorem]{Proposition}
\newtheorem{corollary}[theorem]{Corollary}

\theoremstyle{definition}
\newtheorem{assumption}{Assumption}[section]
\newtheorem{definition}{Definition}[section]

\theoremstyle{remark}
\newtheorem{remark}{Remark}[section]

\newcommand{\F}{\mathcal{{F}}}
\newcommand{\R}{\mathbb{{R}}}

\newcommand\bB{\mathbb{B}}

\newcommand\bI{\mathbb{I}}

\newcommand\bL{\mathbb{L}}
\newcommand\bR{\mathbb{R}}

\newcommand\bN{\mathbb{N}}
\newcommand\bM{\mathbb{M}}

\newcommand\bS{\mathbb{S}}
\newcommand\bT{\mathbb{T}}

\newcommand\bW{\mathbb{W}}

\newcommand\frc{\mathfrak{c}}

\newcommand\frM{\mathfrak{M}}
\newcommand\frZ{\mathfrak{Z}}
\newcommand\frz{\mathfrak{z}}

\newcommand\cA{\mathcal{A}}
\newcommand\cB{\mathcal{B}}

\newcommand\cF{\mathcal{F}}
\newcommand\cG{\mathcal{G}}

\newcommand\cL{\mathcal{L}}
\newcommand\cM{\mathcal{M}}

\newcommand\cO{\mathcal{O}}
\newcommand\cQ{\mathcal{Q}}
\newcommand\cR{\mathcal{R}}
\newcommand\cS{\mathcal{S}}

\newcommand\cZ{\mathcal{Z}}

\newcommand{\ep}{{(\varepsilon)}}

\newcommand{\ke}{k_{\varepsilon}}
\newcommand{\kea}{k_{\varepsilon,\alpha}}

\newcommand{\E}{\mathbb{E}}

\newcommand{\Nte}{\tilde{N}_{1}}

\newcommand{\vp}{\varphi}

\newcommand{\bRY}{{\mathbb{R}^{d'}}}
\newcommand{\bRYn}{{\mathbb{R}^{d'}\setminus\{0\}}}

\makeatletter
 \newcommand{\sumstar}
 {\operatornamewithlimits{\sum@\kern-.2em\raise1ex\hbox{*}}}
 \makeatother

\begin{document}

\author[F. Germ]{Fabian Germ}
\address{School of Mathematics,
University of Edinburgh,
King's  Buildings,
Edinburgh, EH9 3JZ, United Kingdom}
\email{fgerm@ed.ac.uk}

\author[I. Gy\"ongy]{Istv\'an Gy\"ongy}
\address{School of Mathematics and Maxwell Institute, 
University of Edinburgh, Scotland, United Kingdom.}
\email{i.gyongy@ed.ac.uk}


%\subjclass[2010]{Primary  	60H05, 60H15; Secondary 35R60}

\begin{abstract} 
The filtering equations associated to a partially observed jump diffusion 
model $(Z_t)_{t\in [0,T]}=(X_t,Y_t)_{t\in [0,T]}$, driven by Wiener processes 
and Poisson martingale measures are considered. 
Building on results from two preceding articles on the filtering equations, 
the regularity of the conditional density of the signal $X_t$, 
given observations $(Y_s)_{s\in [0,t]}$, is investigated, when the conditional density of 
$X_0$ given $Y_0$ exists and belongs to a Sobolev space, and the coefficients 
satisfy appropriate smoothness and growth conditions.
\end{abstract}

\title[Regularity of the filtering density]{On partially observed jump diffusions III.
Regularity of the filtering density}

\maketitle
\tableofcontents
\mysection{Introduction}

Let $(\Omega,\F,\{\F_t\}_{t\geq 0},P)$ be a complete filtered probability space, 
carrying a $d_1+d'$-dimensional $\cF_t$-Wiener process $(W_t,V_t)_{t\geq0}$ and independent 
$\F_t$-Poisson martingale measures 
$\tilde N_i(d\frz,dt) = N_i(d\frz,dt)-\nu_i(d\frz)dt$ on $\bR_{+}\times \frZ_i$, for $i=0,1$,
with $\sigma$-finite characteristic measures $\nu_0$ and $\nu_1$ 
on a separable measurable space 
$(\frZ_0,\cZ_0)$ and on $(\frZ_1,\cZ_1)=(\bRYn, \cB(\bRYn)$, respectively, 
where $\cB(V)$ denotes the Borel $\sigma$-algebra on $V$ for topological spaces $V$.

We consider the signal and observation model
\begin{equation}                                                                          
\begin{split}
    dX_t    &= b(t,X_t,Y_t)\,dt + \sigma(t,X_t,Y_t)\,dW_t + \rho(t,X_t,Y_t)\,dV_t\\
            &+\int_{\frZ_0}\eta(t, X_{t-},Y_{t-},\frz)\,\tilde N_0(d\frz,dt) 
            + \int_{\frZ_1}\xi(t,X_{t-},Y_{t-},\frz)\,\tilde N_{1}(d\frz,dt)\\
    dY_t    &=B(t,X_t,Y_t)\,dt + dV_t + \int_{\frZ_1} \frz\,\tilde N_{1}(d\frz, dt),
    \end{split}                                                                                             \label{system_1}
\end{equation} 
where $b=(b^i)$, $B=(B^i)$, $\sigma=(\sigma^{ij})$ and $\rho=(\rho^{il})$ 
are Borel functions 
on $\bR_+\times\bR^{d+d'}$, with values in $\bR^d$, $\bR^{d'}$, $\bR^{d\times d_1}$ 
and $\bR^{d\times d'}$, respectively, $\eta=(\eta^i)$ and $\xi=(\xi^i)$ are 
$\bR^d$-valued $\cB(\bR_+\times\bR^{d+d'})\otimes\cZ_0$-measurable and 
$\bR^d$-valued $\cB(\bR_+\times\bR^{d+d'})\otimes\cZ_1$-measurable functions on 
$\bR_+\times\bR^{d+d'}\times \frZ_0$ and $\bR_+\times\bR^{d+d'}\times \frZ_1$, 
respectively.  

This paper is a continuation of \cite{GG1} and \cite{GG2}. In \cite{GG1} 
we derive the filtering equations, 
describing the time evolution of the conditional distribution 
$P_t(dx) = P(X_t\in dx| \cF^Y_t)$ and the unnormalised 
conditional distribution $\mu_t(dx) = P_t(dx)\lambda_t$  
of the unobserved component $X_t$ given $\cF^Y_t$, the $\sigma$-algebra 
generated by the observations $(Y_s)_{s\in [0,t]}$, where $(\lambda_t)_{t\in [0,T]}$ 
is a normalising positive process. The equation for $(\mu_t(dx))_{t\in [0,T]}$, 
referred to as Zakai equation, is given in Theorem \ref{theorem Z1} 
below and has the advantage of being linear in $\mu$, 
making it easier to analyse in some situations. 
For more details on the filtering equations for partially observed 
(jump) diffusions as well as for a historical account we refer to \cite{GG1} 
and the references therein. In \cite{GG2} it is shown that the conditional 
density $\pi_t=dP_t/dx$ exists for $t>0$ and belongs to $L_q$ for 
$q\in[1,p]$, if $\pi_0=dP_0/dx\in L_p$ for a $p\geq2$,  
the coefficients of \eqref{system_1} satisfy appropriate Lipschitz and growth conditions, 
and the derivatives of $\xi$ and $\eta$ in $x$ are equicontinuous in $x$, uniformly 
in their other variables. 
The aim of the present paper is to show that if in addition to the 
Lipschitz and growth conditions in \cite{GG1}, we assume that 
the coefficients have continuous and bounded derivatives in $x$ up to order $m+1$, 
for some integer $m\geq0$, and 
$\pi_0\in W^m_p$ for some $p\geq2$, then $(\pi_t)_{t\geq0}$ is a $W^m_p$-valued 
weakly cadlag process.

For partially observed diffusion processes, i.e., 
when $\xi=\eta=0$ and the observation process $Y$ does not have jumps, 
the existence and the regularity properties of the conditional density $\pi_t$ 
have been extensively studied in the literature. 
In \cite{K1978}, an early work on the regularity of the 
filtering density for continuous diffusions, it was shown 
that if the coefficients are bounded, $\sigma, \rho$ 
admit  bounded derivatives in $x\in\bR^d$ up to order $m+1$, 
$b,B$ admit bounded derivatives in $x$ 
up to order $m$, $\sigma\sigma^{\ast}$ is uniformly non-degenerate
and $\pi_0\in W^m_p\cap W^m_2$, then the 
filtering density $(\pi_t)_{t\in [0,T]}$ is weakly continuous 
as an $W^m_p$-valued process, where $p\geq2$ and $m\geq 0$. 
Later, under the uniform non-degeneracy condition 
on $\sigma\sigma^{\ast}$, stronger results are obtained 
in \cite{K1999} and \cite{K2011} by the help of the $L_p$-theory of SPDEs
developed there. Generalisations of the linear filtering theory 
are presented in \cite{K2010} by the help of the theory of SPDEs 
with VMO leading coefficients and growing lower order coefficients,  
see \cite{K2009}, \cite{K2010a}.

In \cite{R1980} it was proven that the nondegeneracy 
condition can be dropped if one imposes $m+2$ bounded 
derivatives on $\sigma,\rho$ in $x$, as well as 
$m+1$ derivatives on $b,B$ in $x$, to get that $\pi$ 
is a $W^m_p$-valued weakly continuous process 
if $\pi_0\in W^m_p\cap W^m_2$
for a $p\geq2$ and $m\geq1$. 
In \cite{KX1999} it was shown that the conditional density 
$dP_t/dx$ exists for any $t>0$ and it is in $L_2$ if $\pi_0=dP_0/dx$ exists,  
it belongs to $\in L_2$, and the coefficients are bounded and Lipschitz continuous. To achieve this, a nice calculation is presented in \cite{KX1999} 
to show that the $L_2(\Omega\times\bR^d,\bR)$-norm of $P_t^{(\varepsilon)}$, 
the conditional distribution $P_t$ mollified 
by Gaussian kernels, can be estimated 
by the $L_2(\Omega\times\bR^d,\bR)$-norm 
of $\pi_0$, independently of $\varepsilon>0$. 

More recently filtering densities associated 
to signal and observation models $(X_t,Y_t)_{t\geq0}$ 
with jumps have been investigated 
in a growing number of publications.
The above mentioned method from \cite{KX1999} 
was used in \cite{B2014}, \cite{CF2022},  \cite{MPX2019} 
and \cite{Q2021}, 
under various conditions on the filtering models,  
to prove that the conditional density 
$\pi_t=P(X_t\in dx|(Y_s)_{s\in[0,t]})/dx$ 
exists for $t>0$ and it is in $L_2$, when the initial density exists and belongs to 
$L_2$.
In \cite{B2014} only the observation process has jumps.  
In \cite{CF2022}  only the signal process has jumps due to an additive noise 
component which is a cadlag process of bounded variation, adapted to the 
observation process. In \cite{MPX2019} a fairly general jump diffusion model 
is considered, but, as in \cite{B2014} and \cite{CF2022},
the driving noises in the signal are independent of those in the observation. 
In \cite{Q2021} a general jump diffusion model with correlated signal and observation noises 
is considered, but as in the articles \cite{B2014}, \cite{CF2022} and \cite{MPX2019},
the coefficients in the signal do not depend on the observation process. In the above 
publications,  \cite{B2014}, \cite{CF2022}, \cite{MPX2019} 
and \cite{Q2021}, the coefficients of the SDE describing the signal and 
observation models are bounded and satisfy appropriate Lipschitz 
conditions. An approach from \cite{KO} is adapted to study the uniqueness 
of measure-valued solutions to the filtering equations for a  model with 
jumps in \cite{QD2015}, and  the existence of the filtering 
density in $L_2$ is obtained when the initial filtering density is in $L_2$ 
and the jump component in the signal is a symmetric $\alpha$-stable L\'evy process.  
In this publication the signal and observation noises 
are independent of each other, the drift and diffusion coefficients in the signal 
process depend only on $x$, the variable for the signal, they are bounded and 
their derivatives up to first order and and up second order, respectively, 
exist and are bounded functions.

As the present paper is a direct continuation of \cite{GG2}, 
it builds on the results of the latter. 
In \cite{GG2} the method from \cite{KX1999}, combined with methods 
from the theory of SPDEs, is applied to 
partially observed jump diffusions to show that the filtering density $\pi_t$ exists 
for $t>0$ and belongs to $L_q$ for a $p\in[1,p]$, provided $\pi_0\in L_p$  
for a $p\geq2$, the coefficients in the SDE 
satisfy appropriate Lipschitz conditions, the drift coefficient 
in the observation process is bounded, and the other coefficients 
satisfy a linear growth condition. In the present paper we investigate 
the regularity of the filtering density for the same filtering model as in \cite{GG2}.
In addition to these assumptions from \cite{GG2}, in the present paper 
we assume that for an integer $m\geq0$ the initial conditional density 
$\pi_0$ is in $W^m_p$ for some $p\geq2$ and the coefficients of the SDE admit 
bounded derivatives in $x$ up to order $m+1$.  Under these conditions 
we prove that $(\pi_t)_{t\geq0}$ is a $W^m_p$-valued weakly cadlag process.   
Moreover, we show that if the coefficients are also bounded, 
then $(\pi_t)_{t\geq0}$ is a $W^s_p$-valued strongly cadlag process for any $s<m$. \newline

This article is structured as follows. Section \ref{sec main results} 
contains the main results along with the required assumptions. 
In section \ref{sec preliminaries} we state some important results 
from \cite{GG1} and \cite{GG2} which we build on. 
Section \ref{sec estimates} contains Sobolev estimates 
necessary to obtain a priori estimates for the smoothed filtering measures. 
In section \ref{sec solvability} we investigate some solvability properties 
of the Zakai equation. Section \ref{sec proof} finally contains 
the proof of our main theorem, as well as some auxiliary results.  \newline

In conclusion we present important notions and notations used in this paper.
For an integer $n\geq0$ the notation $C^n_b(\bR^d)$ means 
the space of real-valued bounded continuous functions on $\bR^d$, 
which have bounded and continuous derivatives up to order $n$. 
(If $n=0$, then $C^0_b(\bR^d)=C_b(\bR^d)$ denotes the space of 
real-valued bounded continuous functions on $\bR^d$). 
We denote by  $\bM=\bM(\bR^d)$ the set of finite Borel measures 
on $\bR^d$ and by $\frM=\frM(\bR^d)$ the set of finite signed 
Borel measures on $\bR^d$. For $\mu\in\frM$ we use the notation 
$$
\mu(\varphi)=\int_{\bR^d}\varphi(x)\,\mu(dx) 
$$
for Borel functions $\varphi$ on $\bR^d$. 
We say that a function $\nu:\Omega\to\bM$ is $\cG$-measurable 
for a $\sigma$-algebra $\cG\subset\cF$, if $\nu(\varphi)$ is a  
$\cG$-measurable random variable for every bounded Borel function 
$\varphi$ on $\bR^d$. 
An $\bM$-valued stochastic process $\nu=(\nu_t)_{t\in[0,T]}$ 
is said to be weakly cadlag if almost surely 
$\nu_t(\varphi)$ is a cadlag function of $t$ for all $\varphi\in C_b(\bR^d)$. 
An $\frM$-valued process $(\nu_t)_{t\in [0,T]}$ is weakly cadlag, 
if it is the difference of two $\bM$-valued weakly cadlag processes.
For processes $U=(U_t)_{t\in [0,T]}$ we use the notation
$
\cF_t^{U}
$
for the $P$-completion of the $\sigma$-algebra generated by 
$\{U_s: s\leq t\}$. By an abuse of notation, we often write $\cF_t^U$ 
when referring to the filtration $(\cF^U_t)_{t\in [0,T]}$, 
whenever this is clear from the context.  For a  measure space 
$(\frZ,\cZ,\nu)$ and $p\geq1$ we use the notation 
$L_p(\frZ)$ for the $L_p$-space of $\bR^d$-valued 
$\cZ$-measurable processes defined on $\frZ$.  However, 
if not otherwise specified, the function spaces are considered to be over $\bR^d$.
We always use without mention the summation convention, 
by which repeated integer valued indices imply a summation. 
For a multi-index $\alpha=(\alpha_1,\dots,\alpha_d)$ 
of nonnegative integers $\alpha_i, i=1,\dots,d$, 
a function $\vp$ of $x=(x_1,\dots,x_d)\in\bR^d$ 
and a nonnegative  integer $k$ we use the notation
$$
D^\alpha\vp(x)=D_1^{\alpha_1}D_2^{\alpha_2}\dots D_d^{\alpha_d}\vp(x),
\quad\text{as well as}
\quad 
|D^k\vp|^2=\sum_{|\gamma|=k}|D^\gamma \vp|^2,
$$ 
where $D_i=\tfrac{\partial}{\partial {x^i}}$ and $|\cdot|$ denotes an appropriate norm. 
We also use the notation $D_{ij}=D_iD_j$.
If we want to stress that the derivative is taken in a variable $x$, 
we write $D^\alpha_x$. If the norm $|\cdot|$ is not clear from the context, 
we sometimes use appropriate subscripts, 
as in $|\vp|_{L_p}$ for the $L_p(\bR^d)$-norm of $\vp$. 
For $p\geq 1$ and integers $m\geq 0$ we use the notation $W^m_p$ 
for Borel functions $f=f(x)$ on $\bR^d$ such that
$$
|f|_{W^m_p}^p:=\sum_{k=0}^m \int_{\bR^d}|D^kf(x)|^p\,dx<\infty.
$$
Throughout the paper we work on the finite time interval $[0,T]$, 
where $T>0$ is fixed but arbitrary, as well as 
on a given complete probability space $(\Omega,\cF,P)$ equipped 
with a filtration $(\cF_t)_{t\geq0}$ such that $\cF_0$ 
contains all the $P$-null sets. 
For $p,q\geq 1$ and integers $m\geq 1$ 
we denote by $\bW^m_p= L_p((\Omega,\cF_0,P),W^m_p(\bR^d))$ and 
$\bW^m_{p,q}\subset  L_p(\Omega,L_q([0,T],W^m_p(\bR^d)))$ 
the set of $\cF_0\otimes  \cB(\bR^d)$-measurable 
real-valued functions   $f=f(\omega,x)$  
and $\cF_t$-optional $W^m_p$-valued processes  $g=g_t(\omega,x)$ such that
$$
|f|_{\bW^m_p}^p:=\E |f|_{W^m_p}^p<\infty
\quad\text{and}
\quad
|g|_{\bL_{p,q}}^p:=\E\Big(\int_0^T |g_t|_{W^m_p}^q dt\Big)^{p/q}<\infty
$$
respectively. If $m=0$ we set $\bL_p=\bW^0_p$ 
and $\bL_{p,q}=\bW^0_{p,q}$. In case a different 
$\sigma$-algebra $\cG$ than $\cF_0$ is considered above, 
we denote this explicitly by $\bL_p(\cG)$ and $\bW^m_p(\cG)$. 
If $m\geq 0$ is not an integer and $p> 1$ then $W^m_p$ 
denotes the space of real-valued generalised functions $h$ on $\bR^d$ 
such that 
$$
|h|_{W^m_p}:=| (1-\Delta)^{m/2} h|_{L_p}<\infty.
$$
Finally, for real-valued functions $f$ and $g$ on $\bR^d$, 
we often denote by $(f,g)$ the integral of $f\cdot g$ over $\bR^d$.

\mysection{Formulation of the main results}
\label{sec main results}

We fix nonnegative constants $K_0$, $K_1$, $L$, $K$ and functions 
$\bar\xi\in L_2(\frZ_1)=L_2(\frZ_1,\cZ_1,\nu_1)$, $\bar\eta\in 
L_2(\frZ_0)=L_2(\frZ_0,\cZ_0,\nu_0)$, used throughout 
the paper, and make the following assumptions.

\begin{assumption}                                                    \label{assumption SDE}
\begin{enumerate}
\item[(i)]For $z_j=(x_j,y_j)\in\mathbb{R}^{d+d'}$ ($j=1,2$), 
$t\geq0$ and $\frz_i\in\frZ_i$ ($i=0,1$) ,
$$
|b(t, z_1)-b(t,z_2)| + |B(t,z_1)-B(t,z_2)|
    +|\sigma(t,z_1)-\sigma(t, z_2)| 
    + |\rho(t,z_1)-\rho(t,z_2)|\leq L|z_1-z_2|,
$$
$$
|\eta(t,z_1,\frz_0)-\eta(t,z_2,\frz_0)|\leq \bar{\eta}(\frz_0)|z_1-z_2|,
$$
$$
|\xi(t,z_1,\frz_1)-\xi(t,z_2,\frz_1)|\leq \bar{\xi}(\frz_1)|z_1-z_2|.
$$
\item[(ii)] 
For all $z=(x,y)\in\mathbb{R}^{d+d'}$, $t\geq0$ 
and $\frz_i\in \frZ_i$ for $i=0,1$ we have 
$$
|b(t,z)|
 +|\sigma(t,z)|+ |\rho(t,z)|\leq K_0+K_1|z|,
 \quad |B(t,z)|\leq K,\quad 
 \int_{\frZ_1}|\frz|^2\,\nu_1(d\frz)\leq K_0^2,
$$
$$
|\eta(t,z,\frz_0)|\leq \bar{\eta}(\frz_0)(K_0+K_1|z|),
\quad   
|\xi(t,z,\frz_1)|
\leq \bar{\xi}(\frz_1)( K_0+K_1|z|).
$$
\item[(iii)] The initial condition $Z_0=(X_0,Y_0)$ is 
an $\cF_0$-measurable random variable 
with values in $\bR^{d+d'}$.
\end{enumerate}
\end{assumption}

\begin{assumption}                                                           \label{assumption p}
The functions $\bar\eta\in L_2(\frZ_0,\cZ_0,\nu_0)$ 
and $\bar\xi\in L_2(\frZ_1,\cZ_1,\nu_1)$ are 
are such that for constants $K_\eta$ and $K_\xi$ we have 
$\bar{\eta}(\frz_0)\leq K_\eta$ and $\bar{\xi}(\frz_1)\leq K_\xi$ 
for all $\frz_0\in\frZ_0$, $\frz_1\in\frZ_1$.
\end{assumption}

\begin{assumption}                                                                           \label{assumption nu}
For some $r>2$ let $\E|X_0|^r<\infty$ and the measure $\nu_1$ satisfy
$$
K_r:=\int_{\frZ_1} |\frz|^{r}\,\nu_1(d\frz)<\infty.
$$
\end{assumption}

By a well-known theorem of It\^o one knows that 
Assumption \ref{assumption SDE} ensures the existence and 
uniqueness of a solution $(X_t,Y_t)_{t\geq0}$ to  \eqref{system_1} 
for any given $\cF_0$-measurable initial 
value $Z_0=(X_0,Y_0)$, and for every $T>0$,
\begin{equation}                                                              \label{bound_Z}
\E\sup_{t\leq T}(|X_t|^q+|Y_t|^q)\leq N(1+\E|X_0|^q+\E|Y_0|^q)
\end{equation}
holds for $q=2$ with a constant $N$ depending only on 
$T$, $K_0$, $K_1$, $K_2$, $|\bar{\xi}|_{L_2}$, 
$|\bar{\eta}|_{L_2}$ and $d+d'$.
If in addition to Assumption \ref{assumption SDE} 
we assume Assumptions  \ref{assumption p} 
and \ref{assumption nu}, then it is known, 
see e.g. \cite{DKS}, that the moment estimate 
\eqref{bound_Z} holds with $q:=r$ 
for every $T>0$, 
where now the constant $N$ depends also 
on $r$, $K_r$ $K_{\xi}$ and $K_{\eta}$.

\begin{assumption}                                                                              \label{assumption estimates}

(i) For a constant $\lambda>0$ we have 
$$
\lambda|x-\bar x |\leq |x-\bar{x}+\theta(f_i(t,x,y,\frz_i) - f_i(t,\bar{x},y,\frz_i))|
$$
for all $\theta\in [0,1]$, $t\in [0,T]$, $y\in\bR^{d'}$, $x,\bar x\in\bR^d$, 
$\frz_i\in\frZ_i$, $i=0,$ and $f_0(t,x,y,\frz_0)=\eta(t,x,y,\frz_0)$, 
$f_1(t,x,y,\frz_1)=\xi(t,x,y,\frz_1)$.
\newline
(ii) For all $(t,y)\in\bR_+\times\bRY$ and all $x_1,x_2\in\bR^d$,
$$
|(\rho B)(t,x_1,y)-(\rho B)(t,x_2,y)|\leq L|x_1-x_2|.
$$
\newline
(iii) The functions $f_0(t,x,y,\frz):=\xi(t, x,y,\frz)$
and $f_1(t,x,y,\frz):= \eta(t,x,y,\frz)$
are continuously differentiable in $x\in\bR^d$ 
for each $(t,y,\frz)\in \bR_+\times\bRY\times \frZ_i$, for $i=0$ and $i=1$, 
respectively, such that 
$$
\lim_{\varepsilon\downarrow0}
\sup_{t\in[0,T]}\sup_{\frz\in\frZ_i}\sup_{|y|\leq R}
\sup_{|x|\leq R, |\bar x|\leq R, |x-x'|\leq\varepsilon}
|D_xf_i(t,x,y,\frz)-D_xf_i(t,\bar x,y,\frz)|=0 
$$
for every $R>0$. 
\end{assumption}

\begin{assumption}                                                                            \label{assumption SDE2}
Let $m\geq 0$ be an integer.
\begin{enumerate} 
\item[(i)]
The partial derivatives in $x\in\bR^d$ of the 
coefficients $b$, $B$, $\sigma$, $\rho$, $(\rho B)$, $\eta$ 
and $\xi$ up to order $m+1$ 
are functions such that
$$
\sum_{k=1}^{m+1}|D^{k}_x(b,B,\sigma,\rho,(\rho B))|\leq L
\quad\text{for all $t\in [0,T]$, $x\in\bR^d$, $y\in\bR^{d'}$}.
$$
\item[(ii)] Moreover,
$$
\sum_{k=1}^{m+1}|D^{k}_x\eta|\leq L\bar\eta, 
\quad 
\sum_{k=1}^{m+1}|D^{k}_x\xi|\leq L\bar\xi, 
$$
for all $t\in[0,T]$, $x\in\bR^d$, $y\in\bR^{d'}$ 
and $\frz_i\in \frZ_i, i=0,1$.
\end{enumerate}
\end{assumption}

\begin{remark}
\label{remark diffeom}
Note that Assumption \ref{assumption estimates}(i), together 
with Assumptions \ref{assumption p} and \ref{assumption SDE}(i), 
implies that for a constant $c = c(\lambda,K_\xi,K_\eta)$ 
we have for all $\theta\in[0,1]$, $y\in\bR^{d'}$, $t\in [0,T]$ and 
$\frz_i\in\frZ_i$, $i=0,1$, 
\begin{equation*}
c^{-1}|x-\bar{x}|\leq |x-\bar{x}+\theta(f_i(t,x,y,\frz_i)-f(t,\bar{x},y,\frz_i))|
\leq c|x-\bar{x}|
\quad
\text{for $x,\bar{x}\in\bR^d$,} 
\end{equation*}
with $f_0(t,x,y,\frz_0):=\eta(t,x,y,\frz_0)$ 
and $f_1(t,x,y,\frz_1):=\xi(t, x,y,\frz_1)$.
This, together with Assumption \ref{assumption estimates}(iii) 
in particular implies that for all $\theta\in[0,1]$, $y\in\bR^{d'}$, $t\in [0,T]$ and 
$\frz_i\in\frZ_i$, $i=0,1$ the mappings
$$
\tau^\eta(x) = x+\theta\eta(t,x,y,\frz_0)
\quad
\text{and}\quad\tau^\xi(x) = x+\theta\xi(t,x,y,\frz_1)
$$
are $C^1$-diffeomorphisms.
\end{remark}

Let $\cF_t^Y$ denote the completion of the $\sigma$-algebra generated by $(Y_s)_{s\leq t}$. 
\begin{theorem}
\label{theorem regularity}
Let Assumptions \ref{assumption SDE}, \ref{assumption p},  \ref{assumption estimates} 
and  \ref{assumption SDE2} hold. If $K_1\neq 0$ in Assumption \ref{assumption SDE}, 
then let additionally Assumption \ref{assumption nu} hold.
Assume the conditional density $\pi_0=P(X_0\in dx|\cF^Y_0)/dx$ exists almost surely
and for some $p\geq 2$ and integer $m\geq 0$ we have $\E|\pi_0|_{W_p^m}^p<\infty$. 
Then almost surely $P(X_t\in dx|\cF^Y_t)/dx$ exits and belongs 
to $W^m_p$ for every $t\in[0,T]$.\newline
Moreover, there is an $W^m_p$-valued weakly cadlag process $\pi=(\pi_t)_{t\in [0,T]}$ 
such that for each $t$ almost surely $\pi_t= P(X_t\in dx|\cF^Y_t)/dx$. If $K_1=0$ and $m\geq 1$, 
then $\pi$ is strongly cadlag as $W^s_p$-valued process for $s\in [0,m)$. 
\end{theorem}

%%%%%%%%%%%%%%%% SECTION THREE %%%%%%%%%%%%%%%%%%%%%

\mysection{Preliminaries}
\label{sec preliminaries}
Recall the notions and notations concerning measure-valued processes,  
given in the final part of the Introduction, and note that if $(\nu_t)_{t\geq0}$ 
is an $\bM$-valued weakly cadlag process then
there is a set 
$\Omega'\subset\Omega$ of full probability and there is uniquely defined 
(up to indistinguishability)   $\bM$-valued processes $(\nu_{t-})_{t\geq0}$ 
such that for every $\omega\in\Omega'$ 
$$
\nu_{t-}(\varphi)=\lim_{s\uparrow t}\nu_{s}(\varphi)
\quad\text{for all $\varphi\in C_b(\bR^d)$
and $t>0$,}
$$
and for each $\omega\in\Omega'$ we have $\nu_{t-}=\nu_t$, 
for all but at most countably many $t\in(0,\infty)$. 
The following result was proven in \cite{GG1}.
In order to formulate it, we define 
\begin{equation*}                                                                                   \label{gamma}
\gamma_t=\exp\left(-\int_0^tB(s,X_s,Y_s)\,dV_s-\tfrac{1}{2}
\int_0^t|B(s,X_s,Y_s)|^2\,ds\right), 
\quad t\in[0,T].
\end{equation*}
and note that since $B$ is bounded in magnitude, 
we have $\E\gamma_T=1$ and it is an $\cF_t$-martingale 
under $P$. Thus we can define the equivalent 
probability measure $Q$ by $Q:=\gamma_TP$.

\begin{theorem}                                                                                                         \label{theorem Z1}
Let Assumption \ref{assumption SDE} hold. If $K_1\neq 0$, 
then assume also $\E|X_0|^2<\infty$.
Then there exist measure-valued $\cF^Y_t$-adapted weakly cadlag processes 
$(P_t)_{t\in[0,T]}$ and $(\mu_t)_{t\in[0,T]}$ 
such that 
$$
P_t(\varphi)=\mu_t(\varphi)/\mu_t({\bf 1}),
\quad 
\text{for $\omega\in\Omega,\,\, t\in[0,T]$}, 
$$
$$
P_t(\varphi)=\E(\varphi(X_t)|\cF^Y_t),
\quad \mu_t(\varphi)=\E_{Q}(\gamma_t^{-1}\varphi(X_t)|\cF^Y_t) 
\quad
\text{(a.s.) for each $t\in[0,T]$}.
$$
\end{theorem}
We refer to $\mu_t$ (resp. $P_t$) as the unnormalised (resp. normalised) conditional distribution of $X_t$ given $\cF^Y_t$, $t\in [0,T]$.

We introduce 
the random differential operators 
\begin{equation}                                                                           \label{1.28.9.22}
\tilde \cL_t=a^{ij}_t(x)D_{ij}+b^i_t(x)D_i + \beta_t^k\cM_t^k, 
\quad 
\cM^k_t=\rho_t^{ik}(x)D_i+B^k_t(x), \quad k=1,2,...,d', 
\end{equation}
where $\beta_t:=B_t(X_t)$ and
$$
a^{ij}_t(x):=\tfrac{1}{2}\sum_{k=1}^{d_1}(\sigma^{ik}_t\sigma^{jk}_t)(x)
+\tfrac{1}{2}\sum_{l=1}^{d'}(\rho^{il}\rho_t^{jl})(x), 
\quad\sigma_t^{ik}(x):=\sigma^{ik}(t,x,Y_t),\quad
\rho_t^{il}(x):=\rho^{il}(t,x,Y_t), 
$$
$$
b^i_t(x):=b^i(t,x,Y_t),
\quad 
B^k_t(x):=
B^k(t,x,Y_t)
$$
for $\omega\in\Omega$, $t\geq0$, $x=(x^1,...,x^d)\in\bR^d$, 
and $D_i=\partial/\partial x^i$, 
$D_{ij}=\partial^2/(\partial x^i\partial x^j)$ for $i,j=1,2...,d$. 
Moreover for every $t\geq0$ and $\frz \in \frZ_1$ 
we introduce the random operators $I_t^{\xi}$ and $J_t^{\xi}$ defined by 
\begin{equation*}
T_t^{\xi}\varphi(x,\frz) = \varphi(x+\xi_t(x,\frz), \frz)
\end{equation*}
$$
I_t^{\xi}\varphi(x,\frz)=T_t^{\xi}\varphi(x,\frz)-\varphi(x,\frz), 
\quad
J_t^{\xi}\phi(x, \frz)=I_t^{\xi}\phi(x, \frz)-\sum_{i=1}^d\xi_t^i(x,\frz)D_i\phi(x,\frz)
$$
for functions $\varphi=\varphi(x,\frz)$ and $\phi=\phi(x,\frz)$ of 
$x\in\bR^d$ and $\frz\in\frZ_1$, 
and furthermore the random operators 
$I_t^{\eta}$ and $J_t^{\eta}$, defined as $I_t^{\xi}$ and $J_t^{\xi}$, respectively, with 
$\eta_t(x,\frz)$ in place of $\xi_t(x,\frz)$, where 
$$
\xi_t(x,\frz_{1}):=\xi(t,x,Y_{t-},\frz_{1}),
\quad
\eta_t(x,\frz_{0}):=\eta(t,x,Y_{t-},\frz_{0})
$$
for $\omega\in\Omega$, $t\geq0$, $x\in\bR^d$ and $\frz_i\in\frZ_i$ for $i=0,1$. 

From \cite{GG2} we know that if the unnormalised conditional distribution 
$\mu_t$ has a density such that $u_t=d\mu_t/dx$ (a.s.) for each $t\in[0,T]$ 
for an $L_p$-valued  weakly cadlag process $(u_t)_{t\in[0,T]}$ for some $p\geq 2$, 
then it satisfies for each $\vp\in C^\infty_0$ almost surely 
\begin{equation}
\begin{split}
(u_t,\vp)=&(\psi,\vp) +  \int_0^t(u_{s},\tilde\cL_s\varphi)\,ds
+ \int_0^t(u_{s},\cM_s^k\varphi)\,dV^k_s
+ \int_0^t\int_{\frZ_0}(u_{s},J_s^{\eta}\varphi)\,\nu_0(d\frz)ds\\ 
&+ \int_0^t\int_{\frZ_1}(u_{s},J_s^{\xi}\varphi)\,\nu_1(d\frz)ds
+\int_0^t\int_{\frZ_1}(u_{s-},I_s^{\xi}\varphi)\,\tilde N_1(d\frz,ds),\quad t\in [0,T].
\end{split}
\label{equZ}
\end{equation}
for all $t\in[0,T]$. 
Formally we may write \eqref{equZ} as the Cauchy problem
\begin{align}
\nonumber
du_t =& \tilde{\cL}_t^\ast u_t\,dt + \cM^{\ast k}_tu_t\,dV_t^k 
+ \int_{\frZ_0}J_t^{\eta\ast}u_t\,\nu_0(d\frz)dt
\\
\label{equdZ}
 &+\int_{\frZ_1}J_t^{\xi\ast}u_t\,\nu_1(d\frz)dt 
 + \int_{\frZ_1}I_t^{\xi\ast}u_{t-}\,\Nte(d\frz,dt),\\
 \nonumber
u_0 =& \psi.
\end{align}
for a given $\psi$.

\begin{definition}                                                                                                 \label{def Lp solution}
Let integers $m\geq 0$ and $p\geq 2$.  
Let $\psi$ be an $W^m_p$-valued $\cF_0$-measurable random variable. 
Then we say that a $W^m_p$-valued $\cF_t$-adapted 
weakly cadlag process $(u_t)_{t\in[0,T]}$ 
is a $W^m_p$-solution of \eqref{equdZ} with initial condition $\psi$, 
if for each $\varphi\in C_0^{\infty}$ almost surely \eqref{equZ} holds 
for every $t\in[0,T]$.\newline
If $m = 0$, then we call $u$ an $L_p$-solution instead of a $W^0_p$-solution.
\end{definition}

As in \cite{GG2} we are interested in solutions that satisfy
\begin{equation}
\label{condition u}
\esssup_{t\in [0,T]}|u_t|_{L_1}<\infty\quad\text{and}
\quad
\sup_{t\in [0,T]}\int_{\bR^d}|y|^2|u_t(y)|\,dy<\infty\quad\text{(a.s.).}
\end{equation}

To formulate the following results from \cite{GG2}, Lemma 5.7 and Theorem 2.1 therein, 
we recall that there exists a cadlag $\cF^Y_t$-adapted process 
$( ^o\!\gamma_t)_{t\in [0,T]}$, called the optional projection 
of $(\gamma_t)_{t\in [0,T]}$ under $P$ with respect to $(\cF^Y_t)_{t\in [0,T]}$, 
such that for every $\F^Y_t$-stopping time $\tau\leq T$ we have 
\begin{equation}
\label{o gamma}
\E(\gamma_\tau|\cF^Y_\tau) = {^o\!\gamma}_\tau\quad\text{almost surely.}
\end{equation}
Since for each $t$, by known properties of conditional expectations, almost surely
$$
\mu_t({\bf1}) = \E_Q(\gamma_t^{-1}|\cF^Y_t) = 1/\E(\gamma_t|\cF^Y_t) = 1/{^o\!\gamma}_t
$$
and $P,\mu$ are weakly cadlag in the sense described above, 
we also have that 
almost surely $P_t(\vp) = \mu_t(\vp){^o\!\gamma}_t$ 
for each $t\in [0,T]$ and $\vp\in C_b$.

\begin{theorem}
\label{theorem Lp}
Let Assumptions \ref{assumption SDE}, \ref{assumption p} and  \ref{assumption estimates} hold. 
If $K_1\neq 0$, then let additionally Assumption \ref{assumption nu} hold
for some $r>2$. Assume the conditional density $\pi_0=P(X_0\in dx|Y_0)/dx$ 
exists almost surely and $\E|\pi_0|_{L_p}^p<\infty$ for some $p\geq2$. \newline
(i) The unnormalized conditional density $(u_t)_{t\in [0,T]}$ exists almost surely 
and is an $L_p$-valued weakly cadlag process such that for each $t\in [0,T]$ 
almost surely $u_t = d\mu_t/dx$ and
$$
\E\sup_{t\in [0,T]}|u_t|_{L_p}^p\leq N\E|\pi_0|_{L_p}^p.
$$
for a constant $N=N(d,d',p,K,K_\xi,K_\eta, L,T,\lambda,|\bar{\xi}|_{L_2},|\bar{\eta}|_{L_2})$.
Moreover, $u$ is the unique $L_2$-solution to \eqref{equdZ} satisfying the conditions in \eqref{condition u}.\newline
(ii)
Almost surely the conditional density 
$P(X_t\in dx|\cF^Y_t)/dx$ 
exists and belongs to $L_p$ for all $t\in[0,T]$. Moreover, there is an $L_p$-valued  weakly cadlag process  $(\pi_t)_{t\in[0,T]}$, such that for each $t\in [0,T]$ almost surely $\pi_t=P(X_t\in dx|\cF^Y_t)/dx$, as well as almost surely $\pi_t = u_t {^o\!\gamma_t}$ for all $t\in [0,T]$.
\end{theorem}
\begin{proof}
See Lemma 5.5 and Theorem 2.1 in \cite{GG2}. 
\end{proof}

\begin{lemma}
\label{lemma weakly cadlag}
Let $1<p<\infty$ and let $(v_t)_{t\in [0,T]}$ be a weakly cadlag 
$L_p$-valued process. Assume moreover that for 
an $m\geq 0$ almost surely $\esssup_{t\in [0,T]}|v_t|_{W^m_p}<\infty$ 
and $v_T\in W^m_p$. 
 Then $v$ is weakly cadlag as a $W^m_p$-valued process.
\end{lemma}
\begin{proof}
Let $\Omega'$ be the set of those $\omega\in\Omega$ 
such that $(v_t(\omega))_{t\in [0,T]}$ is weakly cadlag 
as an $L_p$-valued function, $v_T(\omega)\in W^m_p$ 
and $\esssup_{t\in [0,T]}|v_t(\omega)|_{W^m_p}<\infty$.  
Then $P(\Omega')=1$, and for each $\omega\in\Omega'$ 
there exists a dense subset $\bT_\omega$ in $[0,T]$ 
such that $\sup_{ t\in \bT_\omega}|v_t(\omega)|_{W^m_p}<\infty$. 
If $\omega\in\Omega'$ and $t\notin\bT_\omega$, $t\neq T$, 
then there exists a sequence $(t_n)_{n=1}^\infty\subset\bT_\omega$ 
such that $t_n\downarrow t$. Since 
$\sup_{ t\in \bT_\omega}|v_t(\omega)|_{W^m_p}<\infty$ 
there exists a subsequence, also denoted by $(t_n)_{n=1}^\infty$, 
such that $v_{t_n}(\omega)$ converges weakly in $W^m_p$
to some element $\tilde{v}\in W^m_p$. 
However, as $v$ is weakly cadlag as an $L_p$-valued process, 
we know that $v_{t_n}\to v_t$ weakly in $L_p$ as $n\to\infty$ 
and hence $\tilde{v} = v_t\in W^m_p$. 
Thus clearly also 
$\sup_{t\in [0,T]}|v_t(\omega)|_{W^m_p}<\infty$ if $\omega\in\Omega'$. 
To see that $v$ is weakly cadlag as a $W^m_p$-valued process, 
note first that since $W^m_p$ is a reflexive space, which is embedded 
continuously and densely into $L_p$, we have that 
the dual $(L_p)^*=L_q$, $q=p/(p-1)$, is embedded continuously and densely 
into $(W^m_p)^*$. Therefore, 
for each $\varepsilon>0$ and $\phi\in (W^m_p)^*$ 
there is an $\phi_\varepsilon\in L_q$ 
such that $|\phi-\phi_{\varepsilon}|_{(W^m_p)^*}<\varepsilon$. 
Fix a $t\in [0,T)$ and a sequence $t_n\downarrow t$. Then
$$
|(v_{t_n},\phi)-(v_t,\phi)|\leq |(v_{t_n},\phi-\phi_\varepsilon)| 
+ |(v_{t_n},\phi_\varepsilon)-(v_t,\phi_\varepsilon)| 
+ |(v_t,\phi_{\varepsilon}-\phi)| 
$$
$$
\leq   2\varepsilon\sup_{t\in [0,T]}|v_t|_{W^m_p} 
+ |(v_{t_n},\phi_\varepsilon)-(v_t,\phi_\varepsilon)|.
$$
Recalling that $v$ is weakly cadlag as an $L_p$-valued process finishes the proof. 
\end{proof}

The corollary of the following lemma  
will play an essential role in the proof of the statement 
on the strong cadlagness 
of $L_p$-solutions to the filtering equations, 
see Proposition \ref{proposition strongly cadlag}.
\begin{lemma}                                                                     \label{lemma 1.14.3.22}
Let $\zeta$ be an $\bR^d$-valued function on $\bR^d$ such that 
for an integer $m\geq1$ it is continuously differentiable up to order $m$, 
and 
\begin{equation}                                                              \label{diffeomorphism}
\inf_{\theta\in[0,1]}\inf_{x\in\bR^d}|\det(\bI+
\theta D\zeta(x))|=:\lambda>0, 
\quad 
\max_{0\leq k\leq m}\sup_{x\in\bR^d}|D^k\zeta(x)|=:M_m<\infty. 
\end{equation}
Then the following statements hold.
\begin{enumerate}
\item[(i)] The function $\tau=x+\theta\zeta(x)$, $x\in\bR^d$, 
is a $C^m$-diffeomorphism for each $\theta$, 
such that for all $x\in\bR^d$, $\theta\in[0,1]$,
\begin{equation}                                                                  \label{Dinverse}
\lambda'\leq |\det D\tau^{-1}(x)|\leq\lambda'',
\quad\text{and}\quad 
\max_{1\leq k\leq m}\sup_{x\in\bR^d}|D^k\tau^{-1}|\leq M'_{m}<\infty, 
\end{equation}
with constants $\lambda'=\lambda'(d, M_1)>0$, 
$\lambda''=\lambda''(d,M_1)$  
and $M'_m=M'_m(d,\lambda,M_m)$. 
\item[(ii)] The function $\zeta^{\ast}(x)=-x+\tau^{-1}(x)$, $x\in\bR^d$,  
is continuously differentiable up to order $m$, such that 
\begin{align}
\sup_{\bR^d}|\zeta^{\ast}|&=\sup_{\bR^d}|\zeta|,                                   \label{1.12.3.22} \\
\sup_{\bR^d}|D^k\zeta^{\ast}|
&\leq M^{\ast}_m\max_{1\leq j\leq k}\sup_{\bR^d}|D^j\zeta|, 
\quad\text{for $k=1,2,...,m$},                                                                  \label{2.12.3.22}\\                                                               
\inf_{\theta\in[0,1]}\inf_{\bR^d}|\det(\bI+\theta D\zeta^{\ast})|
&\geq \lambda'\inf_{\theta\in[0,1]}\inf_{\bR^d}|\det(\bI+\theta D\zeta)|,              \label{3.12.3.22}                                         
\end{align}
with a constant $M_m^{\ast}=M_m^{\ast}(d,\lambda,M_m)$ and with 
$\lambda'$ from  \eqref{Dinverse}.
\item[(iii)] For the function 
$\frc=\det(\bI+D\zeta^{\ast})-1$
we have 
\begin{equation}                                                                                        \label{4.12.3.22}
\sup_{x\in\bR^d}|D^k\frc(x)|
\leq N\max_{1\leq j\leq k+1}\sup_{\bR^d}|D^j\zeta|,        
\end{equation}                                                                                            
for $0\leq k\leq m-1$ with a constant $N=N(d,\lambda, m, M_m)$.
\end{enumerate}
\end{lemma}
\begin{proof}
Claims (i) and (ii) are Lemma 6.1 in \cite{GG2}. To prove (iii)
notice that for the function $F(A)=\det A$, considered as the function 
of the entries $A^{ij}$ of $d\times d$ real matrices $A$, 
by Taylor's formula 
$$
\frc=\det(\bI+D\zeta^{\ast})-\det \bI
=\int_0^1
\tfrac{\partial}{\partial A^{ij}}F(\bI+\theta D\zeta^{\ast})\,d\theta D_i\zeta^{\ast j}. 
$$
\end{proof}

\begin{corollary}                                                       \label{corollary 1.2.10.22}
For $\bR^d$-valued functions $\zeta$ on $\bR^d$ we define the operators 
$T^{\zeta}$, $I^{\zeta}$ and $J^{\zeta}$ by 
$$
T^{\zeta}\varphi(x)=\phi(x+\zeta(x))\quad\text{for $x\in\bR^d$, and} 
\quad
I^{\zeta}\varphi=T^{\zeta}\varphi-\varphi, 
\quad
J^{\zeta}\varphi=I^{\zeta}-\zeta^iD_i\varphi
$$
for differentiable functions $\varphi$ on $\bR^d$. 
Assume that $\zeta$ 
satisfies the conditions of Lemma \ref{lemma 1.14.3.22} with $m=2$. 
Then $\tau_{\theta\zeta}(x)=x+\theta\zeta(x)$, $x\in\bR^d$ 
are $C^2$-diffeomorphisms for each $\theta\in[0,1]$, and 
for every $v,\varphi\in C_0^{\infty}$ we have 
\begin{equation}                                                             \label{2.2.10.22}
(v,I^{\zeta}\varphi)=(I^{\zeta\ast}v,\varphi), 
\quad 
(v,J^{\zeta}\varphi)=(K^{\zeta}_i v, D_i\varphi), 
\end{equation}
with 
$$ 
I^{\zeta\ast}v(x)=-\int_0^1
D_i\Big(v(\tau^{-1}_{\theta\zeta}(x))\zeta^{i}(\tau^{-1}_{\theta\zeta}(x))
|{\rm det}D\tau^{-1}_{\theta\zeta}(x)|\Big)\,d\theta, 
$$
$$
K^{\zeta}_iv(x)=\int_0^1(\theta-1)
D_j\Big(v(\tau_{\theta\zeta}^{-1}(x))
\zeta^i(\tau^{-1}_{\theta\zeta}(x))\zeta^j(\tau^{-1}_{\theta\zeta}(x))
|{\rm{det}}D\tau^{-1}_{\theta\zeta}(x)|\Big)\,d\theta, 
\quad x\in\bR^d,  
$$
for $i=1,2,...,d$. Moreover, for every $x\in\bR^d$ we have 
\begin{equation}                                                         \label{3.2.10.22}
|I^{\eta\ast}v(x)|
\leq N\sup_{x\in\bR^d}
(|\zeta(x)|+|D\zeta(x)|)
\int_0^1|v(\tau^{-1}_{\theta\zeta}(x))|+|(Dv)(\tau^{-1}_{\theta\zeta}(x))|\,d\theta
\end{equation}
\begin{equation}                                                          \label{4.2.10.22}
|K_i^{\zeta}v(x)|
\leq N\sup_{x\in\bR^d}
(|\zeta(x)|^2+|D\zeta(x)|^2)
\int_0^1|v(\tau^{-1}_{\theta\zeta}(x))|+|(Dv)(\tau^{-1}_{\theta\zeta}(x))|\,d\theta
\end{equation}
with a constant $N=N(d,\lambda, M_1,M_2)$. 
\end{corollary}
\begin{proof} 
Using Taylor's formula we have 
$$
(v, I^{\zeta}\varphi)=\int_{\bR^d}\int_0^1
v(x)(D_{i}\varphi)(\tau_{\theta\xi}(x))\xi^i(x)\,d\theta\,dx,
$$
$$
(v,J^{\zeta}\varphi)=\int_{\bR^d}\int_0^1(1-\theta)
v(x)(D_{ij}\varphi)(\tau_{\theta\xi}(x))\xi^i(x)\xi^j(x)\,d\theta\,dx, 
$$
and by a change of variables in the calculation of the integrals over $\bR^d$ 
and then integrating by parts we get the equations 
in \eqref{2.2.10.22}. Hence we get the estimates \eqref{3.2.10.22} 
and \eqref{4.2.10.22} by applying standard derivative rules 
and using the estimates in \eqref{Dinverse}. 
\end{proof}

%%%%%%%%%%%%%%%%%%%%%% SECTION FOUR %%%%%%%%%%%%%%%%%%%%%%%%%%%%%%%%%%%%%%%%%%%%

\mysection{Sobolev estimates}
\label{sec estimates}

In this section we present some estimates which are needed in the subsequent sections. 
In the following lemmas note that by lower indices $i$ we mean the derivative with respect to $x^i$, i.e. $u_i=\frac{\partial}{\partial x^i}u(x)$.
For $\varepsilon>0$ we use the notation $k_{\varepsilon}$ for the Gaussian density function on $\bR^d$ 
with mean 0 and 
variance $\varepsilon$. For linear functionals $\Phi$, acting on a real vector space $V$
containing $\cS=\cS(\bR^d)$, the rapidly decreasing functions on $\bR^d$,  the mollification 
 $\Phi^{(\varepsilon)}$ is defined by 
 $$
 \Phi^{(\varepsilon)}(x)=\Phi(\ke(x-\cdot)), \quad x\in\bR^d.
 $$
In particular, when $\Phi=\mu$ is a (signed) measure from $\cS^{\ast}$, the dual of $\cS$, 
or $\Phi=f$ is a function from $\cS^{\ast}$, 
then 
$$
\mu^{(\varepsilon)}(x)=\int_{\bR^d}k_{\varepsilon}(x-y)\,\mu(dy),
\quad
f^{(\varepsilon)}(x)=\int_{\bR^d}k_{\varepsilon}(x-y)f(y)\,dy,\quad x\in\bR^d,   
$$
and, using the formal adjoint $L^*$, we write
$$
(L^*\mu)^{(\varepsilon)}(x):=\int_{\bR^d}L_yk_{\varepsilon}(x-y)\mu(dy),\quad x\in\R^d
$$
when $L$ is a linear operator on $V$ such that the integral is well-defined for every $x\in\bR^d$. 
Here the subscript $y$ in $L_y$ indicates that the operator $L$ acts in the $y$-variable of  
the function $\bar k_{\varepsilon}(x,y):= k_{\varepsilon}(x-y)$. 
For example, if $L$ is a differential operator of the form $a^{ij}D_{ij}+b^iD_i+c$,  
where $a^{ij}$, $b^i$ and $c$ are functions defined on $\bR^d$, then 
$$
(L^*\mu)^{(\varepsilon)}(x)
=\int_{\bR^d}(a^{ij}(y)\tfrac{\partial^2}{\partial y^i\partial y^j}
+b^i(y)\tfrac{\partial}{\partial y^i}+c(y))k_{\varepsilon}(x-y)\mu(dy). 
$$
We will often use the following well-known properties of mollifications with $k_{\varepsilon}$: 
\begin{enumerate}
\item[(i)] $|\varphi^{(\varepsilon)}|_{L_p}\leq |\varphi|_{L_p}$ 
for $\varphi\in L_p(\bR^d)$, $p\in[1,\infty)$;
\item[(ii)] $\mu^{(\varepsilon)}(\varphi):=\int_{\bR^d}\mu^{(\varepsilon)}(x)\varphi(x)\,dx
=\int_{\bR^d}\varphi^{(\varepsilon)}(x)\mu(dx)=:\mu(\varphi^{(\varepsilon)})$ for 
finite (signed) Borel measures $\mu$ on $\cB(\bR^d)$ and $\varphi\in L_p(\bR^d)$, $p\geq1$; 
\item[(iii)] $|\mu^{(\delta)}|_{L_p}\leq |\mu^{(\varepsilon)}|_{L_p}$ for $0\leq\varepsilon\leq\delta$,  
finite Borel measures $\mu$ on $\bR^d$ and $p\geq1$. 
This property follows immediately from (i) 
and the ``semigroup property" of the Gaussian kernel,  
\begin{equation}                                                                                               \label{sg}
k_{r+s}(y-z)=\int_{\bR^d}k_{r}(y-x)k_{s}(x-z)\,dx,
\quad \text{$y,z\in\bR^d$ and $r,s\in(0,\infty)$}. 
\end{equation}
\end{enumerate}
The following generalization of (iii) is also useful: for integers $p\geq 2$ we have 
\begin{equation}                                                                                       \label{rho}
\rho_{\varepsilon}(y):=\int_{\bR^d}\Pi_{r=1}^pk_{\varepsilon}(x-y_r)\,dx
=c_{p,\varepsilon}e^{-\sum_{1\leq r<s\leq p}|y_r-y_s|^2/(2\varepsilon p)}, 
\quad y=(y_1,....,y_p)\in\bR^{pd},
\end{equation} 
for $\varepsilon>0$, with a constant 
$c_{p,\varepsilon}=c_{p,\varepsilon}(d)=p^{-d/2}(2\pi\varepsilon)^{(1-p)d/2}$. 
This calculation can be found in \cite[Sec. 4]{GG2}.
Clearly, for every $r=1,2,...,p$ and $i=1,2,...,d$ 
\begin{equation}                                                                                               \label{partialrho}
\partial_{y^i_r}\rho_{\varepsilon}(y)=\tfrac{1}{\varepsilon p}\sum_{s=1}^p(y^i_s-y^i_r)\rho_{\varepsilon}(y), 
\quad y=(y_1,...,y_p)\in\bR^d, \quad y_r=(y^1_r...,y^d_r)\in\bR^d. 
\end{equation}
It is easy to see that 
$$
\sum_{r=1}^p\partial_{y_r^j}\rho_{\varepsilon}(y)=0\quad\text{for $y\in\bR^{pd}$, $j=1,2,...,d$},
$$
which we will often use in the form
\begin{equation}                                                                                              \label{Drho}
\partial_{y^j_r}\rho_{\varepsilon}(y)=-\sum_{s\neq r}^p\partial_{y_s^j}\rho_{\varepsilon}(y)
\quad\text{for $r=1,...,p$ and $j=1,2,...,d$}.
\end{equation} 
Moreover, we will use that for $q=1,2$, with a constant $N=N(d,p,q)$,
\begin{equation}                                                                                                \label{qrho}
\varepsilon^{-q}\sum_{s\neq r}|y_s-y_r|^{2q}\rho_{\varepsilon}
(y)\leq N\rho_{2\varepsilon}(y), \quad y\in\bR^{pd}.
\end{equation}

The case of $\alpha = 0$ in the following Lemmas in this section 
is proven in \cite{GG2} and hence this case will be omitted in the proofs.

The following estimates for $\mu\in\frM$ with density $u=d\mu/dx\in W^m_p$, for $m\geq 0$ and $p\geq 2$ even, will be useful in later sections.
In order for the left-hand side of these estimates to be well-defined, we require that
\begin{equation}
\label{second moment mu}
K_1\int_{\bR^d}|x|^2\,|u(x)|\,dx<\infty,
\end{equation}
where we use the formal convention that $0\cdot\infty = 0$, i.e. if $K_1=0$, 
then the second moment of $|\mu(dx)|=|u(x)|dx$ is not required to be finite.

\begin{lemma}                                                                                  \label{lemma dpe1}
Consider integers $m\geq 0$ and $p\geq2$ even.
Let $\sigma=(\sigma^{ik})$ be a Borel function on $\bR^d$ 
with values in $\bR^{d\times k}$, such that for some nonnegative constants 
$K_0$ and $L$ 
\begin{equation}
\label{dpc1}                                                                                         
|\sigma(x)|\leq K_0,
\quad \sum_{k=1}^{m+1}|D^k\sigma(x)|\leq L,
\end{equation} 
for all $x,y\in\bR^d$. Set $a^{ij}=\sigma^{ik}\sigma^{jk}/2$ for $i,j=1,2,...,d$. 
Let $\mu\in\frM$ such that it admits a density $u=d\mu/dx\in W^m_p$ 
which satisfies \eqref{second moment mu}.
Then for $\varepsilon>0$ 
 we have  
\begin{align}                                                                
A^{\alpha}
:=&p((D^{\alpha}\mu^{(\varepsilon)})^{p-1}, 
D^{\alpha}((a^{ij}D_{ij})^*\mu)^{(\varepsilon)})                                                          \nonumber\\
 &+\tfrac{p(p-1)}{2}
 ((D^{\alpha}\mu^{(\varepsilon)})^{p-2}D^{\alpha}((\sigma^{ik}D_i)^*\mu)^{(\varepsilon)}, 
 D^{\alpha}((\sigma^{jk}D_j)^*\mu)^{(\varepsilon)})
\leq NL^2|u|^p_{W^m_p}
\label{dpe1}                                                                                                                                                                                        
\end{align}
for multi-indices $\alpha=(\alpha_1,...,\alpha_d)$ such that 
$0\leq|\alpha|\leq m$, where 
$N$ is a constant depending only on $d$, $m$ and $p$. 
\end{lemma}

\begin{proof}
Note first that using 
\begin{equation}
\label{wellposedness lefthand side1}
\sup_{x\in\bR^d}\sum_{k=0}^{m+2}|D^k\ke(x)|<\infty,\quad
\sup_{x\in\bR^d}\sum_{k=0}^{m+2} |D^k\rho_\varepsilon(x)|<\infty,\quad\text{for all }\varepsilon>0
\end{equation}
and
\begin{equation}
\label{wellposedness lefthand side2}
\int_{\bR^d}(1+|x|+|x|^2)\,|u(x)|\,dx<\infty,
\end{equation}
as well as the conditions on $\sigma$, it is easy to verify that the left-hand side of \eqref{dpe1} is well-defined.
Changing the order of taking derivatives and integrals, then writing integer powers of integrals 
as iterated integrals and using 
$$
D^{\alpha}_xk_{\varepsilon}(x-y)=(-1)^{|\alpha|}D^{\alpha}_yk_{\varepsilon}(x-y),
$$
we have 
\begin{align}
((D^{\alpha}\mu^{(\varepsilon)}(x))^{p-1}=&\int_{\bR^{(p-1)d}}
\Pi_{r=1}^{p-1}D^{\alpha}_xk_{\varepsilon}(x-y_r)\,\mu(dy_1)...\mu(dy_{p-1})                  \nonumber \\  
=& \int_{\bR^{(p-1)d}}(-1)^{(p-1)|\alpha|}
D^{\alpha}_{y_1}...D^{\alpha}_{y_{p-1}}\Pi_{r=1}^{p-1}
k_{\varepsilon}(x-y_r)\,\mu(dy_1)...\mu(dy_{p-1}),                                                            \nonumber\\
D^{\alpha}((a^{ij}D_{ij})^*\mu)^{(\varepsilon)}(x)=&\int_{\bR^{d}}
a^{ij}(y_p)\partial_{y^i_p}\partial_{y^j_p}D^{\alpha}_x
k_{\varepsilon}(x-y_p)\,\mu(dy_{p})                                                                                     \nonumber \\  
=& \int_{\bR^{d}}
(-1)^{|\alpha|}a^{ij}(y_p)
\partial_{y^i_p}\partial_{y^j_p}D^{\alpha}_{y_p}k_{\varepsilon}(x-y_p)\,\mu(dy_{p}),                            \nonumber 
\end{align}
and hence for their product we get 
\begin{equation}                                                                                                             \label{ps}
(D^{\alpha}\mu^{(\varepsilon)})^{p-1}D^{\alpha}((a^{ij}D_{ij})^*\mu)^{(\varepsilon)}(x)
=\int_{\bR^{pd}}a^{ij}(y_p)\partial_{y^i_p}\partial_{y^j_p}
D^{p{\alpha}}_y\Pi_{r=1}^pk_{\varepsilon}(x-y_r)\mu_p(dy), 
\end{equation}
where $D^{p\alpha}_y:=D^{\alpha}_{y^1}...D^{\alpha}_{y_{p}}$ 
and $\mu(dy):=\mu(dy_1)...\mu(dy_{p})$. 
Similarly,
$$
(D^{\alpha}\mu^{(\varepsilon)})^{p-2}D^{\alpha}((\sigma^{ik}D_i)^*\mu)^{(\varepsilon)} 
 D^{\alpha}((\sigma^{jk}D_j)^*\mu)^{(\varepsilon)}(x)
 $$
$$
 =\int_{\bR^{pd}}\sigma^{ik}(y_{p-1})\sigma^{jk}(y_p)\partial_{y^i_{p-1}}\partial_{y^j_p}
D^{p{\alpha}}_y\Pi_{r=1}^pk_{\varepsilon}(x-y_r)\mu_p(dy). 
$$
Adding this to \eqref{ps}, then integrating against $dx$ over $\bR^d$ and using \eqref{rho} we obtain
$$
A=\int_{\bR^{pd}}
\left(pa^{ij}(y_p)\partial_{y^i_p}\partial_{y^j_p}
+\tfrac{p(p-1)}{2}
\sigma^{ik}(y_{p-1})\sigma^{jk}(y_{p})\partial_{y^i_{p-1}}\partial_{y^j_p}\right)
D^{p{\alpha}}_y\rho_{\varepsilon}(y)\,\mu_p(dy). 
$$
Using here the symmetry of $D^{p{\alpha}}_y\rho_{\varepsilon}(y)$ and $\mu_p(dy)$ 
in $y\in\bR^{dp}$ and then interchanging differential operators 
we get 
$$
A=\int_{\bR^{pd}}
\Big(\sum_{r=1}^pa^{ij}(y_r)D^{p{\alpha}}_y\partial_{y^i_r}\partial_{y^j_r}
+\sum_{1\leq r<s\leq p}
\sigma^{ik}(y_{r})\sigma^{jk}(y_{s})D^{p{\alpha}}_y\partial_{y^i_{r}}\partial_{y^j_s}\Big)
\rho_{\varepsilon}(y)\,\mu_p(dy)
$$
Using 
$$
\partial_{y^j_r}\rho_{\varepsilon}(y)
=-\sum_{s\neq r}\partial_{y^j_s}\rho_{\varepsilon}(y), 
$$
see \eqref{Drho}, we have 
$$
\sum_{r=1}^pa^{ij}(y_r)D^{p{\alpha}}_y\partial_{y^i_r}\partial_{y^j_r}\rho_{\varepsilon}(y)
=-\sum_{1\leq r<s\leq p}(a^{ij}(y_r)+a^{ij}(y_s))
D^{p{\alpha}}_y\partial_{y^i_r}\partial_{y^j_s}\rho_{\varepsilon}(y),   
$$
and due to $a^{ij}=\sigma^{ik}\sigma^{jk}/2$ we have 
$$
-2a^{ij}(y_r,y_s):=-2(a^{ij}(y_r)+a^{ij}(y_s))
+\sigma^{ik}(y_{r})\sigma^{jk}(y_{s})+\sigma^{ik}(y_{s})\sigma^{jk}(y_{r})
$$
$$
=-(\sigma^{ik}(y_r)-\sigma^{ik}(y_s))(\sigma^{jk}(y_r)-\sigma^{jk}(y_s)). 
$$
Hence 
\begin{equation}                 
A=-\tfrac{1}{2}\sum_{r\neq s}                                                                                         \label{dpA}
\int_{\bR^{pd}}a^{ij}(y_r,y_s)
D^{p{\alpha}}_y\partial_{y^i_r}\partial_{y^j_s}\rho_{\varepsilon}(y)\,\mu_p(dy), 
\end{equation}
that by integration by parts gives 
\begin{equation}                                                                                      \label{dpip}
=-\tfrac{1}{2}\sum_{\beta\leq\alpha}\sum_{\gamma\leq\alpha}
c^{\alpha}_{\beta}c^{\alpha}_{\gamma}
\int_{\bR^{pd}}\sum_{r\neq s}a^{ij}_{\beta\gamma}(y_r,y_s)\partial_{y^i_r}\partial_{y^j_s}
\rho_{\varepsilon}(y)u_{\bar\beta}(y_r)u_{\bar\gamma}(y_s)\Pi_{q\neq r, q\neq s}u_{\alpha}(y_q)
\,dy, 
\end{equation}
where 
$a^{ij}_{\beta\gamma}(x,r)
:=\partial_{x}^{\beta}\partial^{\gamma}_{r}a^{ij}(x,r)$ and 
$u_{\delta}(x):=\partial^{\delta} u(x)$ for $x,r\in\bR^d$, 
for multi-indices $\beta$, $\gamma$ and $\delta$, $\bar\delta:=\alpha-\delta$ for multi-indices 
$\delta\leq \alpha$ (i.e. $\delta_i\leq \alpha_i$ for $i=1,2,...,d$), 
$
c^{\alpha}_{\delta}=\Pi_{i=1}^dc^{\alpha_i}_{\delta_i} 
$ 
with binomial coefficients $c^n_k$ for integers $0\leq k\leq n$, 
$$
u(y):=u(y_1)....u(y_p)\quad \text{for $y=(y_1,...,y_p)\in\bR^{dp}$}, 
$$
and $dy=dy_1...dy_p$ is the Lebesgue measure on 
$\bR^{pd}$. 
For each $\beta\leq \alpha$ and $\gamma\leq\alpha$ we are going to estimate the 
integrand 
$$
f^{\beta\gamma}(y):=
\sum_{r\neq s}a^{ij}_{\beta\gamma}(y_r,y_s)
\partial_{y^i_r}\partial_{y^j_s}\rho_{\varepsilon}(y)
u_{\bar\beta}(y_r)u_{\bar\gamma}(y_s)\Pi_{q\neq r, q\neq s}u_{\alpha}(y_q),\quad y\in\bR^{dp}, 
\quad \beta\leq \alpha,\,\,\gamma\leq \alpha
$$
in the integral in \eqref{dpip}. Because of the symmetry in 
$\beta$ and $\gamma$, we need only consider the following cases: 
(i) $|\beta|\geq 1$ and $|\gamma|\geq 1$, (ii) $|\beta|\geq1$ and $\gamma=0$ 
and (iii) $\beta=\gamma=0$. To proceed with the calculations in each of these cases, 
for functions $h=h(y)$ and $g=(y)$ of $y\in\bR^{pd}$ 
we will use the notations $h\sim g$  if the integral of $g-h$ 
against $dy$ over $\bR^{pd}$ is zero. In case (i) by integration by parts we have 
$$
f^{\beta\gamma}\sim \sum_{j=1}^4f_j^{\beta\gamma}
$$
with 
\begin{align}
f_1^{\beta\gamma}:=&
\sum_{r\neq s}\partial_{y^i_r}\partial_{y^j_s}a^{ij}_{\beta\gamma}(y_r,y_s)
\rho_{\varepsilon}(y)
u_{\bar\beta}(y_r)u_{\bar\gamma}(y_s)\Pi_{q\neq r, q\neq s}u_{\alpha}(y_q),        \nonumber\\
f_2^{\beta\gamma}:=&
\sum_{r\neq s}\partial_{y^j_s}a^{ij}_{\beta\gamma}(y_r,y_s)
\rho_{\varepsilon}(y)
\partial_{y^i_r}u_{\bar\beta}(y_r)u_{\bar\gamma}(y_s)\Pi_{q\neq r, q\neq s}u_{\alpha}(y_q),       \nonumber\\
f_3^{\beta\gamma}:=&
\sum_{r\neq s}\partial_{y^i_r}a^{ij}_{\beta\gamma}(y_r,y_s)
\rho_{\varepsilon}(y)
u_{\bar\beta}(y_r)\partial_{y^j_s}u_{\bar\gamma}(y_s)
\Pi_{q\neq r, q\neq s}u_{\alpha}(y_q),       \nonumber\\
f_4^{\beta\gamma}:=&
\sum_{r\neq s}a^{ij}_{\beta\gamma}(y_r,y_s)
\rho_{\varepsilon}(y)
\partial_{y^i_r}u_{\bar\beta}(y_r)\partial_{y^j_s}u_{\bar\gamma}(y_s)
\Pi_{q\neq r, q\neq s}u_{\alpha}(y_q).       \nonumber\\
\end{align}
It is easy to see that for $j=1,2,3,4$
$$
|f_j^{\beta\gamma}(y)|\leq NL^2\rho_{\varepsilon}(y)
\sum_{|\delta|\leq m}|u_{\delta}(y_1)|...\sum_{|\delta|\leq m}|u_{\delta}(y_p)|, 
\quad (y_1,y_2,...,y_p)\in\bR^{pd}
$$
with a constant $N=N(d,m,p)$.  Hence in the case (i) we get 
\begin{align}
\int_{\bR^{pd}}f^{\beta\gamma}(y)\,dy
\leq &
NL^2\int_{\bR^{pd}}
\sum_{|\delta|\leq m}|u_{\delta}(y_1)|...\sum_{|\delta|\leq m}|u_{\delta}(y_p)|
\rho_{\varepsilon}(y)\,dy                                                                       \nonumber\\
=&NL^2\int_{\bR^{pd}}\int_{\bR^d}
\sum_{|\delta|\leq m}|u_{\delta}(y_1)|...\sum_{|\delta|\leq m}|u_{\delta}(y_p)| 
\Pi_{r=1}^pk_{\varepsilon}(x-y_r)\,dx\,dy                                                    \nonumber\\                                       
\leq& N'L^2\sum_{|\delta|\leq m}||D^{\delta}u|^{(\varepsilon)}|^p_{L_p}.    \nonumber
\end{align}
with constants $N$ and $N'$ depending only on $d$, $m$ and $p$. 
Integrating by parts in the case (ii) we have 
$$
f^{\beta0}\sim-f_1^{\beta0}-f_2^{\beta0}
$$
with 
$$
f_1^{\beta0}=\sum_{r\neq s}\partial_{y^i_r}a^{ij}_{\beta0}(y_r,y_s)\partial_{y^j_s}
\rho_{\varepsilon}(y)u_{\bar\beta}(y_r)\Pi_{q\neq r}u_{\alpha}(y_q)
$$
$$
f_2^{\beta0}=\sum_{r\neq s}a^{ij}_{\beta0}(y_r,y_s)\partial_{y^j_s}
\rho_{\varepsilon}(y)
\partial_{y^i_r}u_{\bar\beta}(y_r)\Pi_{q\neq r}u_{\alpha}(y_q). 
$$
Clearly, for $r\neq s$ we have 
$$
\partial_{y^i_r}a^{ij}_{\beta0}(y_r,y_s)=g^{\beta,j}(y_r,y_s)+h^{\beta,j}(y_r), 
$$
with 
\begin{align}
g^{\beta,j}(y_r,y_s)=&\partial_{y_r^i}\partial_{y_r}^{\beta}\sigma^{ik}(y_r)
(\sigma^{jk}(y_r)-\sigma^{jk}(y_s))
+\partial_{y_r^i}\partial_{y_r}^{\beta}\sigma^{jk}(y_r)
(\sigma^{ik}(y_r)-\sigma^{ik}(y_s)),                                                                                     \nonumber\\
h^{\beta,j}(y_r)=&\sum_{1\leq|\delta|,\delta<\beta(i)}c^{\beta(i)}_{\delta}
\partial^{\delta}_{y_r}\sigma^{ik}(y_r)\partial^{\beta(i)-\delta}_{y_r}\sigma^{jk}(y_r), \nonumber
\end{align}
where the multi-index $\beta(i)$ is defined by 
$\partial^{\beta(i)}=\partial_{y^i_r}\partial^{\beta}_{y_r}$. Thus 
$$
f_1^{\beta0}=f_{11}^{\beta0}+f_{12}^{\beta0}
$$
with
\begin{align}                                                                          
f_{11}^{\beta0}=&\sum_{r=1}^p\sum_{s\neq r}g^{\beta,j}(y_r,y_s)
\partial_{y^j_s}\rho_{\varepsilon}(y)
u_{\bar\beta}(y_r)\Pi_{q\neq r}u_{\alpha}(y_q),                                                                       \nonumber\\
f_{12}^{\beta0}=&\sum_{r=1}^p\sum_{s\neq r}h^{\beta,j}(y_r)\partial_{y^j_s}
\rho_{\varepsilon}(y)u_{\bar\beta}(y_r)\Pi_{q\neq r}u_{\alpha}(y_q).    \nonumber\\        
\end{align}                                                                             
Since 
$$
|g^{\beta,j}(y_r,y_s)|\leq NL^2|y_r-y_s|\quad j=1,2,...,p,
$$
for some $N = N(d,m,p)$, taking into account \eqref{partialrho} 
we have 
$$
|g^{\beta,j}(y_r,y_s)\partial_{y^j_s}\rho_{\varepsilon}(y)|\leq 
\tfrac{N}{p\varepsilon}\sum_{1\leq k<l\leq p}^p|y_k-y_l|^2\rho_{\varepsilon}(y)
\leq N'\rho_{2\varepsilon}(y)  
$$
 and hence 
$$
|f_{11}^{\beta0}|\leq N'\rho_{2\varepsilon}(y)
\sum_{|\delta|\leq m}|u_{\delta}(y_1)|...\sum_{|\delta|\leq m}|u_{\delta}(y_p)| 
$$
with a constant $N'=N'(d,m,p)$. 
Remembering  \eqref{Drho} 
by integration by parts we obtain 
$$
f^{\beta0}_{12}=-\sum_{r=1}^ph^{\beta,j}(y_r)
\partial_{y^j_r}\rho_{\varepsilon}(y)u_{\bar\beta}(y_r)\Pi_{q\neq r}u_{\alpha}(y_q)
\sim f^{\beta0}_{121}+f^{\beta0}_{122}
$$
with 
$$
f^{\beta0}_{121}=
\sum_{r=1}^ph^{\beta,j}(y_r)
\rho_{\varepsilon}(y)\partial_{y^j_r}u_{\bar\beta}(y_r)\Pi_{q\neq r}u_{\alpha}(y_q), 
$$
$$
f^{\beta0}_{122}=\sum_{r=1}^p\partial_{y^j_r}h^{\beta,j}(y_r)
\rho_{\varepsilon}(y)u_{\bar\beta}(y_r)\Pi_{q\neq r}u_{\alpha}(y_q).
$$
Hence noting that 
$$
|h^{\beta,j}(y_r)|+|\partial_{y^j_r}h^{\beta,j}(y_r)|\leq NL^2
$$
with a constant $N=N(d,m,p)$, we get 
$$
|f^{\beta0}_{121}+f^{\beta0}_{122}|
\leq 
NL^2\rho_{\varepsilon}(y)
\sum_{|\delta|\leq m}|u_{\delta}(y_1)|...\sum_{|\delta|\leq m}|u_{\delta}(y_p)| 
$$
Consequently, for a constant $N'=N'(d,m,p)$,
\begin{align}
\int_{\bR^{pd}}f_1^{\beta0}(y)\,dy
\leq &
NL^2\int_{\bR^{pd}}
\sum_{|\delta|\leq m}|u_{\delta}(y_1)|...\sum_{|\delta|\leq m}|u_{\delta}(y_p)|
\rho_{2\varepsilon}(y)\,dy                                                                       \nonumber\\
\leq& N'L^2\sum_{|\delta|\leq m}||D^{\delta}u|^{(2\varepsilon)}|^p_{L_p}
\leq N'L^2\sum_{|\delta|\leq m}||D^{\delta}u|^{(\varepsilon)}|^p_{L_p}.     \label{f1}
\end{align}
Now we are going to estimate the integral of $f^{\beta0}_{2}$. If $|\beta|=1$, then 
$$
|a^{ij}_{\beta0}(y_r,y_s)|\leq NL^2|y_r-y_s|, 
$$
and taking into account  \eqref{partialrho}, we get 
$$
|f_{2}^{\beta0}|\leq NL^2\rho_{2\varepsilon}(y)
\sum_{|\delta|\leq m}|u_{\delta}(y_1)|...\sum_{|\delta|\leq m}|u_{\delta}(y_p)| 
$$
with $N=N(d,p,m)$ in the same way as $|f_{11}|$ is estimated. Hence, 
as above,  
\begin{equation}
\int_{\bR^{pd}}f_2^{\beta0}(y)\,dy                                                                    
\leq NL^2\sum_{|\delta|\leq m}||D^{\delta}u|^{(2\varepsilon)}|^p_{L_p}
\leq NL^2\sum_{|\delta|\leq m}||D^{\delta}u|^{(\varepsilon)}|^p_{L_p}.      \label{f2}
\end{equation}
for $|\beta|=1$. 
If $|\beta|\geq2$, then 
$$
a^{ij}_{\beta0}(y_r,y_s)=g^{\beta,ij}(y_r,y_s)+h^{\beta,ij}(y_r)
$$
with 
\begin{align}
g^{\beta,ij}(y_r,y_s)=
&\partial^{\beta}_{y_r}\sigma^{ik}(y_r)
(\sigma^{jk}(y_r)-\sigma^{jk}(y_s))
+ \partial^{\beta}_{y_r}\sigma^{jk}(y_r)
(\sigma^{ik}(y_r)-\sigma^{ik}(y_s))                                                                        \nonumber\\
h^{\beta,ij}(y_r)=
&\sum_{1\leq|\delta|,\delta<\beta}c^{\beta}_{\delta}\partial^{\delta}_{y_r}\sigma^{ik}(y_r)
\partial^{\beta-\delta}_{y_r}\sigma^{jk}(y_r).                                                              \nonumber
\end{align}
Noticing that for a constant $N=N(d,m,p)$, 
$$
\sum_{i,j}|g^{\beta,ij}(y_r,y_s)|\leq NL^2|y_r-y_s|
$$
and 
$$
\sum_{ij}|h^{\beta,ij}(y_r)|+\sum_{ij}|\partial_{y^j_r}h^{\beta,ij}(y_r)|\leq NL^2, 
$$
we obtain \eqref{f2} for $|\beta|\geq 2$ in the same way as the integral of $f_1^{\beta0}$ 
is estimated. It remains to consider the case (iii), i.e., to estimate the integral 
of $f^{00}$. Since 
$$
|a^{ij}_{00}(y_r,y_s)|\leq NL^2|y_r-y_s|^2
$$
with a constant $N=N(d,m,p)$ and 
$$
\partial_{y^i_r}\partial_{y^{j}_s}\rho_{\varepsilon}(y)
=\tfrac{1}{p^2\varepsilon^2}\sum_{k=1}^p\sum_{l=1}^p(y_k^i-y^i_r)(y_l^j-y^j_s)
+\frac{1}{p\varepsilon}\delta_{ij}, 
$$
we have for a constant $N'=N'(d,m,p)$,
\begin{align}
|a^{ij}_{00}(y_r,y_s)\partial_{y^i_r}\partial_{y^{j}_s}\rho_{\varepsilon}(y)|
\leq&
 \tfrac{N}{\varepsilon^2}L^2\sum_{1\leq k<l\leq p}|y_k-y_l|^4\rho_{\varepsilon}(y)
+\tfrac{N}{\varepsilon}L^2\sum_{1\leq k<l\leq p}|y_k-y_l|^2\rho_{\varepsilon}(y)              \nonumber\\
\leq &
N'L^2\rho_{2\varepsilon}(y) \quad\text{for  $y=(y_1,...,y_p)\in\bR^{pd}$}.                                                                       \nonumber                                                                                                                                                                              
\end{align}
Hence 
$$
|f^{00}(y)|\leq NL^2\rho_{2\varepsilon}(y)\Pi_{r=1}^p|u_{\alpha}(y_r)|, 
$$
that gives
$$ 
\int_{\bR^{pd}}f^{00}(y)\,dy\leq NL^2||D^{\alpha}u|^{(2\varepsilon)}|^p_{L_p}
\leq NL^2||D^{\alpha}u|^{(\varepsilon)}|^p_{L_p} 
$$
with a constant $N=N(d,m,p)$, and we finish the proof of \eqref{dpe1} 
by using $|v^{(\varepsilon)}|_{L_p}\leq |v|_{L_p}$ for $v\in L_p(\bR^d)$. 
\end{proof}

\begin{corollary}
\label{p3 cor lemma dpe1}
Let the conditions of Lemma \ref{lemma dpe1} hold for integers $m\geq 0$ and $p\geq2$ even.
Then for $\varepsilon>0$ 
 we have  
$$
((D^{\alpha}\mu^{(\varepsilon)})^{p-1}, 
D^{\alpha}((a^{ij}D_{ij})^*\mu)^{(\varepsilon)})
\leq  NL^2|u|^p_{W^m_p} 
$$
for multi-indices $\alpha=(\alpha_1,...,\alpha_d)$ such that 
$0\leq|\alpha|\leq m$, where 
$N$ is a constant depending only on $d$, $m$ and $p$. 
\end{corollary}
\begin{proof}
It suffices to note that 
$$
((D^{\alpha}\mu^{(\varepsilon)})^{p-2}D^{\alpha}((\sigma^{ik}D_i)^*\mu)^{(\varepsilon)}, 
 D^{\alpha}((\sigma^{jk}D_j)^*\mu)^{(\varepsilon)})
 $$
 $$
 =\int_{\bR^d} (D^{\alpha}\mu^{(\varepsilon)})^{p-2}(x) \sum_{k=1}^d\big| D^{\alpha}((\sigma^{ik}D_i)^*\mu)^{(\varepsilon)}(x)\big|^2\,dx  \geq 0
$$
\end{proof}

\begin{lemma}                                                                                   \label{lemma pe5}
Let $p\geq 2$ and $m\geq0$ be integers, 
and let $\sigma = (\sigma^{i})$ and $b$ be 
Borel functions on $\bR^d$ with values in $\bR^d$ and $\bR$ respectively. 
Assume the partial derivatives of $\sigma$ 
and $b\sigma$ up to order $m$ are functions such that
there exist constants $K\geq L\geq1$ such that
$$
\sum_{k=0}^{m+1}|D^kb(x)|\leq K,
\quad |\sigma(x)|\leq K_0,
$$
$$
\sum_{k=1}^{m+1}|D^k\sigma(x)| +\sum_{k=1}^{m+1}|D^k(b\sigma)(x)|
\leq L
$$
for all $x,y\in\bR^d$. Then for finite signed Borel measures $\mu$ on $\bR$ with density 
$u:=d\mu/dx\in W^m_p$, satisfying \eqref{second moment mu}, we have 
\begin{equation}                                                                                                 \label{Dpe4_1}
\big((D^{\alpha}\mu^{\ep})^{p-2}D^{\alpha}(b\mu)^\ep,D^{\alpha}(b\mu)^\ep \big)
\leq NK^2|u|_{W^m_p}^p,
\end{equation}
\begin{equation}
 \label{Dpe4_2}
\big((D^{\alpha}\mu^\ep)^{p-2},D^{\alpha}((\sigma^iD_i)^*\mu)^\ep D^{\alpha}(b\mu)^\ep \big)
\leq NKL|u|_{W^m_p}^p 
\end{equation}
for $\varepsilon>0$ and multi-indices $\alpha$ 
such that $|\alpha|\leq m$,  where $N$ is a constant 
depending only on $d$, $p$, $m$.  
\end{lemma}

\begin{proof} First note that by \eqref{wellposedness lefthand side1} 
and \eqref{wellposedness lefthand side2}, as well as the conditions 
on $\sigma$ and $b$, the left-hand sides of \eqref{Dpe4_1} and \eqref{Dpe4_2} are well-defined.
Interchanging the order of integration and the differential operator  
$D^{\alpha}$, rewriting the product of 
integrals as multiple integral, using Fubini's theorem and 
the identity 
$$
D^{\alpha}_x k_{\varepsilon}(x-z)=(-1)^{|\alpha|}D^{\alpha}_zk_{\varepsilon}(x-z), \quad x,y\in\bR^d, 
$$
as well as \eqref{rho}, 
for the left-hand side $F_{\alpha}$ of 
\eqref{Dpe4_1} we compute 
$$
F^{\alpha} = \int_{\bR^d}\int_{\bR^{pd}}b(y_{r})b(y_s)\Pi_{j=1}^p D^{\alpha }_x\ke (x-y_j)\,\mu_p(dy)\,dx
$$
$$
=(-1)^{p|\alpha|}\int_{\bR^d}\int_{\bR^{pd}}b(y_{r})b(y_s)\Pi_{j=1}^p D^{\alpha }_{y_j}\ke (x-y_j)\,\mu_p(dy)\,dx
$$
$$
=(-1)^{p|\alpha|}\int_{\bR^{pd}}b(y_{r})b(y_s) D^{p\alpha }_{y}\int_{\bR^d}\Pi_{j=1}^p\ke (x-y_j)\,dx\,\mu_p(dy)
$$
$$
=(-1)^{p|\alpha|}\int_{\bR^{pd}}b(y_{r})b(y_s) D^{p\alpha }_{y}\rho_{\varepsilon}(y)\Pi_{j=1}^pu(y_j)\,dy 
$$
for any $r,s\in\{1,2,...,p\}$ such that $r\neq s$, where recall that $dy=dy_1...dy_p$ and 
$D^{p\alpha}_y=\Pi_{j=1}^pD^{\alpha}_{y_j}$.  
Hence by integration by parts 
we obtain
$$
F^{\alpha}=\sum_{\beta\leq \alpha}\sum_{\gamma\leq\alpha}
c^{\alpha}_{\beta}c^{\alpha}_{\gamma}\int_{\bR^{pd}}b_{\beta}(y_{r})
b_{\gamma}(y_s) \rho_{\varepsilon}(y)u_{\bar\beta}(y_r)
u_{\bar\gamma}(y_s)\Pi_{j\neq s,r}^pu_{\alpha}(y_j)\,dy,  
$$
where $v_{\delta}:=D^{\delta}v$ and $\bar\delta:=\alpha-\delta$ for functions $v$ on $\bR^d$ 
and multi-indices $\delta\leq \alpha$. 
Using here \eqref{rho} and the boundedness condition on $|b|$ and $|D^{\delta}b|$ we have   
$$
F^{\alpha}\leq NK^2\sum_{\beta\leq\alpha}\sum_{\gamma\leq\alpha}
\int_{\bR^{pd}}\int_{\bR^d}\Pi_{j=1}^p\ke (x-y_j)|u_{\bar\beta}(y_r)|
|u_{\bar\gamma}(y_s)|\Pi_{j\neq s,r}^p|u_{\alpha}(y_j)|
\,dx\,dy 
$$
$$
=NK^2\sum_{\beta\leq\alpha}\int_{\bR^d}
|u_{\bar\beta}|^{(\varepsilon)}
|u_{\bar\gamma}|^{(\varepsilon)}||u_{\alpha}|^{(\varepsilon)}|^{p-2}
\,dx\leq N'K^2|u|^{p}_{W^m_p}
$$
with constants $N$ and $N'$ depending only on $p$, $d$ and $m$, where the last inequality follows by 
Young's inequality and the boundedness of the mollification operator 
in $L_p$.  Now we are going to prove \eqref{Dpe4_2}. By the same way 
as we have rewritten $F^{\alpha}$ we can rewrite the left-hand side $R^{\alpha}$ of the inequality 
\eqref{Dpe4_2}
as 
\begin{equation}                                                                                  \label{DRa}
R^{\alpha}=\int_{\bR^{dp}}
f_{krr}(y)\,dy,
\end{equation}
for any $r, k\in\{1,2,..,p\}$ such that $r\neq k$, where 
$$
f_{krs}(y):=(-1)^{p|\alpha|}b(y_{k})\sigma^i(y_{r})\partial_{y^i_{s}}D^{p\alpha}_y\rho_\varepsilon(y)\,
\Pi_{j=1}^pu(y_j),\quad y=(y_1,...,y_p)\in\bR^{pd}
$$
for $k,r,s\in\{1,2,...,p\}$.  As in the proof of 
Lemma \ref{lemma dpe1}, for real functions $f$ and $g$ we write $f\sim g$ if 
they have the same (finite) Lebesgue integral against $dy=dy_1...dy_p$ 
over $\bR^{pd}$. 
We write $f\preceq g$ if the integrals of $f$ and $g$ against 
$dy$ over $\bR^d$ are finite, and the integral of $f-g$ 
can be estimated by $NKL|u|_{W^m_p}$ for all $u\in W^m_p$ with a constant 
$N=N(d,m,p)$, independent of $u$. By integration by parts 
we have 
$$
f_{krr}\sim \sum_{\gamma\leq \alpha}\sum_{\beta\leq \alpha}
f_{krr}^{\gamma\beta}
$$
with 
$$
f_{krs}^{\gamma\beta}(y):=c^{\alpha}_{\gamma}c^{\alpha}_{\beta}b_{\gamma}(y_{k})\sigma^i_{\beta}(y_{r})\partial_{y_s^i}\rho_\varepsilon(y)\,
u_{\bar\gamma}(y_k)u_{\bar\beta}(y_r)
\Pi_{j\neq k,r}u_\alpha(y_j) 
$$
If $\beta\neq0$ then by integration by parts 
(dropping $\partial_{y^i_r}$ from $\rho_{\varepsilon}$ to the other terms), 
and using the boundedness of $b$, its derivatives up to order $m+1$, and 
the boundedness of the derivatives of $\sigma$ up to order $m+1$, 
we see that $f_{krr}^{\gamma\beta}\preceq0$ for any $k=1,2,..p$, $r\neq k$ and 
$\gamma\leq\alpha$. If $\beta = 0$ and $\gamma = 0$, then $f^{00}_{krr}$ can be estimated by an exact repetition of the proof of Lemma 4.2 in \cite{GG2}, by replacing $\mu$ therein with $u_{\alpha}dy$, to yield $f^{00}_{krr}\preceq 0$. 
Consequently, 
$$
f_{krr}\preceq \sum_{0\neq\gamma\leq \alpha}
f_{krr}^{\gamma0}\quad\text{for every $k=1,...,p$ and $r\in\{1,2,...,p\}\setminus\{k\}$}. 
$$
Writing $f_{krr}^{\gamma 0}(y) = g_{krr}^\gamma(y)h_k^{\bar{\gamma}}(y)$, with
$$
g_{krs}^\gamma(y):=c^{\alpha}_{\gamma}b_{\gamma}(y_{k})\sigma^i(y_{r})\partial_{y_s^i}\rho_\varepsilon(y),\quad 
h_k^{\bar{\gamma}}(y):= u_{\bar\gamma}(y_k)
\Pi_{j\neq k}u_\alpha(y_j), 
$$
we get 
\begin{equation}                                                                                     \label{DR1}
p(p-1)(p-2)R^\alpha \leq \sum_{0\neq \gamma\leq \alpha}
\sum_{s=1}^p\sum_{r\neq s}\sum_{k\neq s,r}
\int_{\bR^{dp}}g_{kss}^{\gamma}(y)h_k^{\bar\gamma}(y)\,dy+NKL|u|^p_{W^m_p},
\end{equation}
and by \eqref{Drho},  
$$
p(p-1)R^{\alpha}\leq-\sum_{0\neq\gamma\leq \alpha}
\sum_{k=1}^p\sum_{r\neq k}\sum_{s\neq r}
\int_{\bR^{pd}}g^{\gamma}_{krs}(y)h_k^{\bar\gamma}(y)\,dy+NKL|u|^p_{W^m_p} 
$$
$$
=-\sum_{0\neq\gamma\leq \alpha}
\sum_{s=1}^p\sum_{r\neq s}\sum_{k\neq s,r}
\int_{\bR^{pd}}g^{\gamma}_{krs}(y)h_k^{\bar\gamma}(y)\,dy 
$$
\begin{equation}
\label{DR2}
-\sum_{0\neq\gamma\leq \alpha}
\sum_{s=1}^p\sum_{r\neq s}\int_{\bR^{pd}}g^{\gamma}_{srs}(y)h_s^{\bar\gamma}(y)\,dy
+NKL|u|^p_{W^m_p}
\end{equation}
with a constant $N=N(d,m,p)$. Summing up \eqref{DR1} and \eqref{DR2} 
we obtain 
$$
c_pR^{\alpha}\leq 
\sum_{0\neq\gamma\leq \alpha}
\sum_{s=1}^p\sum_{r\neq s}\sum_{k\neq r,s}
\int_{\bR^{pd}}(g^{\gamma}_{kss}(y)-g^{\gamma}_{krs}(y))h_k^{\bar{\gamma}}(y)\,dy
+NKL|u|^p_{W^m_p}
$$
\begin{equation}                                                                     \label{DR}                                                                                                                
 -\sum_{0\neq\gamma\leq \alpha}
\sum_{s=1}^p\sum_{r\neq s}\int_{\bR^{pd}}g^{\gamma}_{srs}(y)h_s^{\bar\gamma}(y)\,dy
+NKL|u|^p_{W^m_p}
\end{equation}
where $c_p=p(p-1)^2$, and 
$$
(g^{\gamma}_{kss}(y)-g^{\gamma}_{krs}(y))h_k^{\bar\gamma}(y)=
c^{\alpha}_{\gamma}
b_{\gamma}(y_k)(\sigma^{i}(y_s)-\sigma^{i}(y_r))\partial_{y^i_s}
\rho_{\varepsilon}(y)u_{\bar\gamma}(y_k)\Pi_{j\neq k}u_\alpha(y).  
$$
By  the boundedness of $|b_{\gamma}|$ and the Lipschitz condition on 
$\sigma$, using \eqref{qrho} we get 
$$
(g^{\gamma}_{kss}-g^{\gamma}_{krs})h_k^{\bar{\gamma}}\preceq 0,\quad\text{for all $0\neq \gamma\leq\alpha$, $s=1,2,...,p$ and $r\neq s$, $k\neq r,s$.} 
$$
By integration by parts we have for the last term in \eqref{DR}, 
$$
g^{\gamma}_{srs}h_s^{\bar\gamma}\preceq 0,
\quad
\text{for $0\neq\gamma\leq\alpha$ and $s=1,\dots,p$, $r\neq s$,}
$$
which finishes the proof
of \eqref{Dpe4_2}. 
\end{proof}

For vectors $\xi=\xi(x)\in\bR^d$, depending on $x\in\bR^d$ we consider the linear operators 
$I^{\xi}$ and $J^{\xi}$ defined by 
\begin{equation}
\label{equ operators TIJ}
T^\xi\vp(x) = \vp(x+\xi(x))
\end{equation}
$$
I^{\xi}\varphi(x):=T^\xi\varphi(x)-\varphi(x), 
\quad 
J^{\xi}\psi(x):=I^{\xi}\psi(x)-\xi(x)D_i\psi(x),
$$
$x\in\bR^d$,  acting on functions $\varphi$ and differentiable functions 
$\psi$ on $\bR^d$.

\begin{lemma}                                                                                                  \label{lemma dpe3}
Let $\xi=\xi(x,\frz)$ be an $\bR^d$-valued function of 
$x\in\bR^d$ for every $\frz\in \frZ$ for 
a set $\frZ$. Assume that for an integer $m\geq1$ 
the partial derivatives of $\xi$ in $x\in\bR^d$ 
up to order $m$ are functions on $\bR^d$  for each $\frz\in \frZ$, such that 
for a constant $\lambda>0$, a function $\bar\xi$ on $\frZ$ 
and a constant $K_\xi\geq 0$ we have 
$$
|\xi(x,\frz)|\leq \bar\xi(\frz)\leq K_\xi,\quad
$$
\begin{equation}                                                                                                  \label{dpc2}
\sum_{k=1}^{m+1}|D^k_x\xi(x,\frz)|\leq \bar\xi(\frz), 
\quad |{\rm{det}}(\mathbb I+\theta D_x\xi(x,\frz))|\geq \lambda^{-1}
\end{equation}
for all $x,y\in\bR^d$, $\frz\in \frZ$ and $\theta\in[0,1]$. 
Let $p\geq2$ be an even integer. Then 
for every finite signed Borel  measure $\mu$ 
with density $u=d\mu/dx\in W^m_p$, satisfying \eqref{second moment mu}, 
we have 
$$
C:=\int_{\bR^d}p(D_x^{\alpha}\mu^{(\varepsilon)})^{p-1}
D_x^{\alpha}(J^{\xi*}\mu)^{(\varepsilon)}
\,dx
$$
$$
+\int_{\bR^d}(D_x^{\alpha}\mu^{(\varepsilon)}
+D_x^{\alpha}(I^{\xi*}\mu)^{(\varepsilon)})^p
-(D_x^{\alpha}\mu^{(\varepsilon)})^p
-p(D_x^{\alpha}\mu^{(\varepsilon)})^{p-1}
D_x^{\alpha}(I^{\xi*}\mu)^{(\varepsilon)}\,dx
$$
\begin{equation}                                                                                                \label{dpe3}
\leq N\bar\xi^2(\frz)|u|^p_{W^m_p} \quad \text{for $\frz\in \frZ$,\,\,$\varepsilon>0$} 
\end{equation}
for multi-indices $\alpha$, $0\leq|\alpha|\leq m$ with a constant 
$N=N(d,p,m,\lambda,K_\xi)$. 
\end{lemma}
\begin{proof} 
Again we note that by \eqref{wellposedness lefthand side1} 
\& \eqref{wellposedness lefthand side2}, 
together with the conditions on $\xi$,
it is easy to verify that $C$ is well-defined.
Notice that 
$$
D^{\alpha}_x\mu^{(\varepsilon)}
+D^{\alpha}_x(I^{\xi*}\mu)^{(\varepsilon)}=D^{\alpha}_x(T^{\xi*}\mu)^{(\varepsilon)}
$$
and 
$$
p(D^{\alpha}_x\mu^{(\varepsilon)})^{p-1}D^{\alpha}_x(J^{\xi*}\mu)^{(\varepsilon)}
-p(D^{\alpha}_x\mu^{(\varepsilon)})^{p-1}D^{\alpha}_x(I^{\xi*}\mu)^{(\varepsilon)}
$$
$$
=-p(D^{\alpha}_x\mu^{(\varepsilon)})^{p-1}D^{\alpha}_x((\xi^iD_i)^*\mu)^{(\varepsilon)},  
$$
Hence                                                                                                      
\begin{equation}                                                                                  \label{dpC}
C=\int_{\bR^d}(D^{\alpha}_x(T^{\xi*}\mu)^{(\varepsilon)})^p
-(D^{\alpha}_x\mu^{(\varepsilon)})^p
-p(D^{\alpha}_x\mu^{(\varepsilon)})^{p-1}D^{\alpha}_x((\xi^iD_i)^*\mu)^{(\varepsilon)}\,dx. 
\end{equation}
First we change the order of $D^{\alpha}_x$ and the integrals and operators 
$T^{\xi}_y$ and $I^{\xi}_y$ acting in the variable $y\in\bR^d$, then we use  
$$
D^{\alpha}_xk_{\varepsilon}(x-y)=(-1)^{|\alpha|}D^{\alpha}_yk_{\varepsilon}(x-y)
$$
to get  
$$
D^{\alpha}_x(T^{\xi*}\mu)^{(\varepsilon)}
=(-1)^{|\alpha|}\int_{\bR^d}T^{\xi}_yD^{\alpha}_yk_{\varepsilon}(x-y)\,\mu(dy), 
$$
$$
D^{\alpha}_x\mu^{(\varepsilon)}
=(-1)^{|\alpha|}\int_{\bR^d}D^{\alpha}_yk_{\varepsilon}(x-y)\,\mu(dy), 
$$
\quad
$$
D^{\alpha}_x((\xi^iD_i)^*\mu)^{(\varepsilon)}
=(-1)^{|\alpha|}\int_{\bR^d}\xi^i(y)\partial_{y^i}D^{\alpha}_yk_{\varepsilon}(x-y)\,\mu(dy). 
$$
Thus rewriting the product of integrals as multiple 
integrals, and using the product measure 
$\mu_p(dy):=\mu(dy_1)...\mu(dy_p)$ on $\bR^{dp}$ by Fubini's theorem we get 
\begin{align}
(D^{\alpha}_x(T^{\xi*}\mu)^{(\varepsilon)})^p(x)
=&\int_{\bR^{pd}}
\Pi_{r=1}^p(T^{\xi}_{y_r}D^{\alpha}_{y_r}k_{\varepsilon}(x-y_r))\,\mu_p(dy)               \nonumber\\ 
=&\int_{\bR^{pd}}
\Pi_{r=1}^pT^{\xi}_{y_r}D^{p\alpha}_{y}\Pi_{r=1}^pk_{\varepsilon}(x-y_r)\,\mu_p(dy), \nonumber\\ 
(D^{\alpha}_x\mu^{(\varepsilon)})^p
=&\int_{\bR^{pd}}\Pi_{r=1}^p
(D^{\alpha}_{y_r}k_{\varepsilon}(x-y_r))\,\mu_p(dy)                                                         \nonumber\\
=&\int_{\bR^{pd}}
D^{p\alpha}_{y}\Pi_{r=1}^pk_{\varepsilon}(x-y_r)\,\mu_p(dy) 
\end{align}
and 
\begin{align}
p(D^{\alpha}_x\mu^{(\varepsilon)})^{p-1}D^{\alpha}_x((\xi^iD_i)^*\mu)^{(\varepsilon)}
=&p\int_{\bR^{pd}}\Pi_{r=1}^{p-1}(D^{\alpha}_{y_r}k_{\varepsilon}(x-y_r))
\xi^i(y_p)\partial_{y^i_p}D^{\alpha}_{y_p}k_{\varepsilon}(x-y_p)
\mu_p(dy)                                                                                                                        \nonumber\\ 
=&p\int_{\bR^{pd}}\xi^i(y_p)\partial_{y^i_p}
D^{p\alpha}_{y}\Pi_{r=1}^pk_{\varepsilon}(x-y_r)\mu_p(dy)                                             \nonumber\\ 
=&\int_{\bR^{pd}}\sum_{r=1}^p\xi^i(y_r)\partial_{y^i_r}
D^{p\alpha}_{y}\Pi_{r=1}^{p}k_{\varepsilon}(x-y_r)\,\mu_p(dy),        
\end{align}
where again
$$
D^{p\alpha}_{y}:=\Pi_{r=1}^{p}D^{\alpha}_{y_r}\quad \text{for $y=(y_1,...,y_p)\in\bR^{dp}$}, 
$$
and the last equation is due to the symmetry of the function 
$\Pi_{r=1}^{p}D^{\alpha}_{y_r}k_{\varepsilon}(x-y_r)$ and the measure $\mu_p(dy)$ in 
$y=(y_1,...,y_p)\in\bR^{pd}$. 
Thus from \eqref{dpC} we get 
$$
C=\int_{\bR^d}\int_{\bR^{pd}}
L^{\xi}_yD^{p\alpha}_y\Pi_{r=1}^pk_{\varepsilon}(x-y_r)\,\mu_p(dy)\,dx
$$
with the operator 
$$
L^{\xi}_y=\Pi_{r=1}^pT^{\xi}_{y_r}-\mathbb I-\sum_{r=1}^p\xi^i(y_r)\partial_{y^i_r},   
$$
defined by 
$$
L^{\xi}_y\varphi(y)=\varphi(y_1+\xi(y_1),...,y_p+\xi(y_p))-\varphi(y)-
\sum_{r=1}^p\xi^i(y_r)\partial_{y^i_r}\varphi(y),  
\quad
y=(y_1,....,y_p)\in\bR^{pd}
$$
for differentiable functions $\varphi$ of $y=(y_1,....,y_p)\in\bR^{pd}$.
Using here Fubini's theorem then changing the order of the operator 
$L^{\xi}_yD^{p\alpha}_y$ and 
the integration against $dx$, by virtue of \eqref{rho} we have  
\begin{equation}
C=\int_{\bR^{pd}}L^{\xi}_yD^{p\alpha}_y\int_{\bR^d}\Pi_{r=1}^pk_{\varepsilon}(x-y_r)
\,dx\,\mu_p(dy)
=\int_{\bR^{pd}}L^{\xi}_yD^{p\alpha}_y\rho_{\varepsilon}(y)\,\mu_p(dy),                       \label{dpC2}
\end{equation}
By Taylor's formula
$$
L^{\xi}_yD^{p\alpha}_y\rho_{\varepsilon}(y)
=\int_0^1(1-\vartheta)
\xi^i(y_k)\xi^j(y_l)(\partial_{y^i_k}\partial_{y^j_l}
D^{p\alpha}\rho_{\varepsilon})(y+\vartheta\bar\xi(y))\,d\vartheta, 
$$
where $y=(y_1,...,y_p)\in\bR^{pd}$, $y_k\in\bR^d$ for $k=1,2,...,p$, 
and $\bar\xi(y):=(\xi(y_1),...,\xi(y_p))$ 
for $y=(y_1,...,y_p)\in\bR^{dp}$. 
Thus by changing the order of integrals and then changing 
the variables $y_k$ with $y_k+\vartheta\xi(y_k)$ for $k=1,2,...,p$, 
from \eqref{dpC2} we obtain 
\begin{equation}                                                                                                       \label{dpC3}
C=\int_0^1(1-\vartheta)C(\vartheta)\,d\vartheta                                                 
\end{equation}
with
$$
C({\vartheta})=\int_{\bR^{pd}}
\sum_{k=1}^p\sum_{l=1}^p
\hat\xi^i(y_k)\hat\xi^j(y_l)
\partial_{y^i_k}\partial_{y^j_l}D^{p\alpha}_y\rho_{\varepsilon}(y)
\Pi_{r=1}^p\hat u(y_r)
\, dy, 
$$
where,  with $\tau_\vartheta(x):=x + \vartheta\xi(x)$,
\begin{equation}
\label{19.2.22.1}
\hat\xi^i(x):=\xi^i(\tau^{-1}_{\vartheta}(x)), 
\quad \hat u(x)=u(\tau^{-1}_\vartheta(x))|\det D\tau_\vartheta^{-1}(x)|, \quad x\in\bR^d,
\quad i=1,2,...,d, 
\end{equation}
and  $dy:=dy_1dy_2...dy_p$ denotes the Lebesgue measure on $\bR^{pd}$. 
Clearly, 
$$
C({\vartheta})=C_1({\vartheta})+C_2({\vartheta})
$$
with
\begin{align}
C_1({\vartheta})=&\int_{\bR^{pd}}                                                                                   
\sum_{k=1}^p
\hat\xi^i(y_k)\hat\xi^j(y_k)
\partial_{y^i_k}\partial_{y^j_k}D^{p\alpha}_y\rho_{\varepsilon}(y)
\Pi_{r=1}^p\hat u(y_r)\,dy,                                                                        \nonumber\\
C_2({\vartheta})=&\int_{\bR^{pd}}                                                                                
\sum_{k=1}^p\sum_{l\neq k}
\hat\xi^i(y_k)\hat\xi^j(y_l)
\partial_{y^i_k}\partial_{y^j_l}D^{p\alpha}_y\rho_{\varepsilon}(y)
\Pi_{r=1}^p\hat u(y_r)\,dy.                                                                              \nonumber
\end{align}
Using \eqref{Drho} and the symmetry in $y_k$ and $y_l$, we have 
\begin{align}                                                                                                    \nonumber\\
C_1(\vartheta)=&-\tfrac{1}{2}
\int_{\bR^{pd}}                                                                                   
\sum_{k=1}^p\sum_{l\neq k}
(\hat\xi^i(y_k)\hat\xi^j(y_k)+\hat\xi^i(y_l)\hat\xi^j(y_l))
\partial_{y^i_k}\partial_{y^j_l}D^{p\alpha}_y\rho_{\varepsilon}(y)
\Pi_{r=1}^p\hat u(y_r)\,dy,                                                                             \nonumber\\
C_2({\vartheta})=&\tfrac{1}{2}\int_{\bR^{pd}}                                                                                
\sum_{k=1}^p\sum_{l\neq k}
(\hat\xi^i(y_k)\hat\xi^j(y_l)+\hat\xi^i(y_l)\hat\xi^j(y_k))
\partial_{y^i_k}\partial_{y^j_l}D^{p\alpha}_y\rho_{\varepsilon}(y)
\Pi_{r=1}^p\hat u(y_r)\,dy.                                                                              \nonumber\\                       
\end{align}
Hence 
\begin{equation}                                                                                          \label{dpCtheta}
C(\vartheta)=
\int_{\bR^{pd}}                                                                                   
\sum_{r=1}^p\sum_{s\neq r}
\hat a^{ij}(y_r,y_s)
\partial_{y^i_r}\partial_{y^j_s}D^{p\alpha}_y\rho_{\varepsilon}(y)
\Pi_{r=1}^p\hat u(y_r)\,dy                                               
\end{equation}
with 
$$
\hat a^{ij}(y_r,y_s)=-\tfrac{1}{2}(\hat \xi^i(y_r)-\hat \xi^i(y_s))(\hat \xi^j(y_r)-\hat \xi(y^j_s))
$$
Notice that the right-hand side of equation \eqref{dpCtheta} is the same as 
the right-hand side of \eqref{dpA} with $\hat\xi^{i}$ in place of $\sigma^{i\cdot}$ 
for each $i=1,2,...,d$ and with $\hat u$ in place of $u$. It is easy to verify, 
see Lemma 3.3 in \cite{DGW}, that for a constant $N=N(d,\lambda,m,K_\xi)$ we have 
$$
\sum_{k=1}^{m+1} |D_x^k(\tau_\vartheta^{-1}(x))|\leq N,
\quad
\text{for each $\vartheta\in[0,1], \frz\in\frZ, x\in\bR^d$.}
$$
Thus also  for each $\vartheta\in[0,1]$,
\begin{equation}
\label{19.2.22.2}
\sum_{k=1}^{m+1}|D^k\hat\xi(x,\frz)|\leq N\bar\xi(\frz)\quad 
\text{for $x\in\bR^d$, $\frz\in\frZ$},
\end{equation}
with a constant $N=N(d,m,\lambda,K_\xi)$, 
i.e., for each $\vartheta\in[0,1]$  
and $\frz\in \frZ$ the function $\hat\xi$ of $x\in\bR^d$ satisfies 
the condition \eqref{dpc1} on $\sigma$ in Lemma \ref{lemma dpe1}, 
with $N\bar\xi(\frz)$ in place of $L$. 
Consequently, copying the calculations 
which lead from equation \eqref{dpA} to the estimate \eqref{dpe1} 
in the proof of Lemma \ref{lemma dpe1}, we obtain  
$$
C(\vartheta)\leq N \bar\xi^2(\frz)|\hat u|^p_{W^m_p}
\quad
\text{for each $\vartheta\in[0,1]$, $\frz\in \frZ$}
$$
with a constant $N=N(d,m,p,\lambda,K_\xi)$. Note that 
due to the condition \eqref{dpc2} there is a constant 
$N=N(d,p,m,\lambda,K_\xi)$ such that 
\begin{equation}
\label{19.2.22.3}
|\hat u|_{W^m_p}\leq N|u|_{W^m_p} \quad \text{for all $\vartheta\in[0,1]$}. 
\end{equation}
Hence by virtue of \eqref{dpC3} the estimate \eqref{dpe3} follows.
\end{proof}

\begin{corollary}
\label{corollary dJ}
Let the conditions of Lemma \ref{lemma dpe3} hold.
Then 
for every finite signed Borel  measure $\mu$ 
with density $u=d\mu/dx\in W^m_p$, satisfying \eqref{second moment mu}, 
we have 
\begin{equation}
\label{dpe3-2}
\int_{\bR^d}(D_x^{\alpha}\mu^{(\varepsilon)})^{p-1}
D_x^{\alpha}(J^{\xi*}\mu)^{(\varepsilon)}
\,dx
\leq N \bar\xi^2(\frz)|u|^p_{W^m_p} 
\quad 
\text{for $\frz\in \frZ$,\,\,$\varepsilon>0$} 
\end{equation}
for multi-indices $\alpha$, $0\leq|\alpha|\leq m$ with a constant 
$N=N(d,p,m,\lambda,K_\xi)$. 
\end{corollary}

\begin{proof}
By the convexity of the function $f(a) = a^p$ 
for even $p\geq 2$ we know that $(a+b)^p - a^p - pa^{p-1}b\geq 0$ 
for all $a,b\in \bR^d$. Applying this with $a = D^\alpha u^\ep$ 
and $b = D^\alpha (I^\xi u)^\ep$ shows that \eqref{dpe3} implies  \eqref{dpe3-2}.
\end{proof}

\begin{lemma}
\label{lemma dpe4}
Let the conditions of Lemma \ref{lemma dpe3} hold. 
Then for every finite signed Borel measure $\mu$ 
with density $u = d\mu/dx\in W_p^m$, 
satisfying \eqref{second moment mu}, we have
\begin{equation}
\label{dpe4}
\Big|\int_{\bR^d}
(D^\alpha u^{(\varepsilon)}
+D^\alpha(I^{\xi*}\mu)^{(\varepsilon)})^p-(D^\alpha u^{(\varepsilon)})^p
\,dx\Big|
\leq N \bar{\xi}(\frz)|u|^p_{W^m_p},
\end{equation}
for a constant $N=N(d,p,m,\lambda,K_\xi)$ for $\frz\in\frZ$, 
where the argument $x\in\bR^d$ 
is suppressed in the integrand.
\end{lemma}
\begin{proof}
Define
$$
F:= \int_{\bR^d}
(D^\alpha u^{(\varepsilon)}
+D^\alpha(I^{\xi*}u)^{(\varepsilon)})^p-(D^\alpha u^{(\varepsilon)})^p
\,dx
$$
$$
= \int_{\bR^d}
(D^\alpha(T^{\xi*}u)^{(\varepsilon)})^p-(D^\alpha u^{(\varepsilon)})^p\,dx,
$$
where we use the operator $T$ defined in \eqref{equ operators TIJ}. 
As in the proof of Lemma 4.5 in \cite{GG2} we define the operator
$$
M_y^\xi = \Pi_{i=1}^p T_{y_i}^\xi - \bI
$$
where $\bI$ is the identity operator. 
Observe that using Fubini's theorem 
and the notation $D^{p\alpha}_y = \Pi_{r=1}^p D^\alpha_{y_r}$, 
$dy = dy_1\cdots dy_p$, $y= (y_1,\dots,y_p)\in \bR^{pd}$,
$$
F = \int_{\bR^d}\int_{\bR^{dp}}\Big(D_x^{p\alpha}\Pi_{i=1}^p T_{y_i}^\xi \Pi_{j=1}^p \ke(x-y_j) \Pi_{k=1}^p u(y_k)
- D_x^{p\alpha}\Pi_{j=1}^p \ke (x-y_j)\Pi_{k=1}^p u(y_k)\Big)\,dy\,dx
$$
$$
= \int_{\bR^d}\int_{\bR^{dp}}
\Big(M_y^\xi D_y^{p\alpha} \Pi_{j=1}^p \ke(x-y_j)\Big) \Pi_{k=1}^p u(y_k)\,dy\,dx
= \int_{\bR^{dp}}
\Big(M_y^\xi D_y^{p\alpha} \rho_\varepsilon(y)\Big) \Pi_{k=1}^p u(y_k)\,dy.
$$
Next, note that by Taylor's formula with $\bar{\xi}(y) = (\xi(y_1),\dots,\xi(y_p))\in \bR^{dp}$,  
$$
M_y^\xi D_y^{p\alpha} \rho_\varepsilon(y) 
= \sum_{k=1}^p\int_0^1(\partial_{y_k^i}D^{p\alpha}_y\rho_\varepsilon)(y 
+ \vartheta\bar\xi (y))\,d\vartheta\,\xi^i(y_k).
$$
Thus, by a change of variables, Fubini's theorem and the functions defined in \eqref{19.2.22.1},
$$
F = \sum_{k=1}^p\int_0^1\int_{\bR^{dp}}
\partial_{y_k^i}D_y^{p\alpha}\rho_\varepsilon(y)\hat{\xi}^i(y_k)\Pi_{j=1}^p\hat{u}(y_j)
\,dy\,d\vartheta,
$$
which by integration by parts gives, with multi-indices $\beta\leq\alpha$, 
$\bar{\beta}:=\alpha - \beta$ and constants $c_\beta^\alpha$,
$$
F = \sum_{\beta\leq\alpha}c_\beta^\alpha\sum_{k=1}^p
\int_0^1\int_{\bR^{dp}}
\partial_{y_k^i}\rho_\varepsilon(y)\hat{\xi}_{\beta}^i(y_k)\hat{u}_{\bar\beta}(y_k)
\Pi_{j\neq k}^p\hat{u}_{\alpha}(y_j)
\,dy\,d\vartheta
$$
$$
=\sum_{\beta\leq\alpha}c^\alpha_\beta\sum_{k=1}^p \int_0^1 f_k^\beta(\vartheta)\,d\vartheta,
$$
where  for $k=1,\dots,p$, $\vartheta\in [0,1]$ and $\beta\leq\alpha$,
$$
f_k^\beta(\vartheta) 
:= \int_{\bR^{dp}}\partial_{y_k^i}\rho_\varepsilon(y)\hat{\xi}_{\beta}^i(y_k)
\hat{u}_{\bar\beta}(y_k)\Pi_{j\neq k}^p\hat{u}_{\alpha}(y_j)\,dy
$$
and where $\hat{u}_{\gamma}(y_k)=D_{y_k}^\gamma \hat{u}(y_k)$ 
for $\gamma = \alpha,\bar\beta$. We consider two cases. In the first case, 
let $\bar{\beta}<\alpha$ and hence $|\beta|\geq 1$. Then by 
integration by parts, 
for all $k=1,\dots,p$ and a constant $N = N(d,p,m,\lambda)$, 
$$
f_k^\beta(\vartheta) = -\int_{\bR^{dp}}
\rho_\varepsilon(y)
\big(
(\partial_{y_k^i}\hat{\xi}_{\beta}^i(y_k))\hat{u}_{\bar\beta}(y_k)
+\hat{\xi}_{\beta}^i(y_k)(\partial_{y_k^i}\hat{u}_{\bar\beta}(y_k))
\big)
\Pi_{j\neq k}^p\hat{u}_{\alpha}(y_j)\,dy\,d\vartheta
\leq N\bar{\xi}(\frz)|u|_{W_p^m}^p,
$$
where we used \eqref{19.2.22.2} and \eqref{19.2.22.3}. 
In the second case $\bar{\beta} = \alpha$ so that $\beta = 0$ 
and we have 
$$
\sum_{k=1}^p f_k^0 = \sum_{k=1}^p 
\int_{\bR^{pd}}
\partial_{y_k^i}\rho_\varepsilon(y)\hat{\xi}^i(y_k)\Pi_{j=1}^p\hat{u}_{\alpha}(y_j)
\,dy,
$$
as well as by using \eqref{Drho} and the symmetry in $s$ and $k$,
$$
\sum_{k=1}^p f_k^0 = -\sum_{k=1}^p \sum_{s\neq k}
\int_{\bR^{pd}}
\partial_{y_k^i}\rho_\varepsilon(y)\hat{\xi}^i(y_s)\Pi_{j=1}^p\hat{u}_{\alpha}(y_j)
\,dy.
$$
Therefore also, with a constant $N = N(d,p,m,\lambda,K_\xi)$,
$$
\big|(p-1)\sum_{k=1}^p f_k^0 + p\sum_{k=1}^p f_k^0\big|
$$
$$
 = \Big|\sum_{k=1}^p \sum_{s\neq k}\int_{\bR^{pd}}
 \partial_{y_k^i}\rho_\varepsilon(y)
 \big(\hat{\xi}^i(y_k)-\hat{\xi}^i(y_s)\big)
 \Pi_{j=1}^p\hat{u}_{\alpha}(y_j)\,dy\Big|
 \leq N\bar{\eta}(\frz)|u|_{W_p^m}^p,
$$
where we used \eqref{19.2.22.2} together with \eqref{qrho}, 
as well as \eqref{19.2.22.3}. This proves the lemma.
\end{proof}

%%%%%%%%%%%% SECTION FOUR %%%%%%%%%%%%%%%%%
%%%%%%%%%%%%%%%%%%     SECTION FIVE     %%%%%%%%%%%%%%%%%%%%%

\mysection{Solvability of the filtering equations in Sobolev spaces}
\label{sec solvability}

The following two lemmas are essentially Lemma 5.2  in \cite{GG2}, 
where instead of $D^\alpha\ke$ the kernel $\ke$ is considered. 
However, keeping this difference in mind, 
the arguments in the proofs of Lemma 5.2  in \cite{GG2} can  easily be adapted. 
Hence we only provide an outline and refer the reader 
to the preceding article \cite{GG2} for full details.

\begin{lemma}                                                                                                                \label{lemma ediff}
Let the Assumption \ref{assumption SDE}  hold. 
Let $u$ be an $L_p$-solution of \eqref{equdZ}, $p\geq 2$, 
such that \eqref{equZ} holds for each $\vp\in C_0^\infty(\bR^d)$ 
almost surely for all $t\in [0,T]$ and assume moreover that 
$\esssup_{t\in [0,T]}|u_t|_{L_1}<\infty$. If $K_1\neq 0$ 
in Assumption \ref{assumption SDE} (ii), then assume additionally
\begin{equation}
\label{esssup condition}
 \esssup_{t\in [0,T]}\int_{\bR^d}|y|^2|u_t(y)|\,dy<\infty,\quad \text{almost surely.}
\end{equation} 
Then for each $\varepsilon>0$ and integer $m\geq 0$, 
for any multi-index $\alpha=(\alpha_1,\dots,\alpha_d)$, 
$|\alpha|\leq m$, for all $x \in\bR^d$ 
almost surely
\begin{equation}                                                                                                                  \label{ediff}
\begin{split}
    D^\alpha u_t^\ep(x)  &= D^\alpha u^\ep_0(x)
    + \int_0^t D^\alpha(\tilde{\mathcal{L}}_s^* u_s)^\ep(x)\,ds
    + \int_0^t D^\alpha(\mathcal{M}^{*k}_s u_s)^\ep(x)\,dV^k_s \\
    &+ \int_0^t\int_{\frZ_0}D^\alpha(J^{\eta *}_s u_s)^\ep(x)\,\nu_0(d\frz)ds
    + \int_0^t\int_{\frZ_1}D^\alpha(J^{\xi *}_s u_s)^\ep(x)\,\nu_1(d\frz)ds\\
    &+ \int_0^t\int_{\frZ_1}D^\alpha(I^{\xi *}_s u_{s-})^\ep(x)\,\Nte(ds,d\frz),
\end{split}
\end{equation}
for all $t\in[0,T]$.
\end{lemma}

\begin{proof}
The case of $\alpha = 0$ is Lemma 5.4 in \cite{GG2}. 
The case of $\alpha\neq 0$ such that $0<|\alpha|\leq m$ 
works exactly in the same way. We first define for 
a $\psi\in C_0^\infty(\bR)$ such that $\psi(0)=1$, $\psi(r)=0$ 
for $|r|\geq 2$,  for $n\geq 1$, $\psi_n(x):=\psi(|x|/n)\in C_0^\infty(\bR^d)$. 
Setting $\vp_x(y):=\kea(x-y)\psi_n(y)\in C_0^\infty$ in \eqref{equdZ}, 
where $\kea(x-y)=D^\alpha_x\ke(x-y)$, yields that for each $x\in\bR^d$ almost surely
$$
(u_t,\kea(x-\cdot)\psi_n)=(u_0,\kea(x-\cdot)\psi_n) 
+  \int_0^t\big(u_{s},\tilde\cL_s(\kea(x-\cdot)\psi_n)\big)\,ds
$$
\begin{equation}
\label{14.4.22}
+ \int_0^t\big(u_{s},\cM_s^k(\kea(x-\cdot)\psi_n)\big)\,dV^k_s
+ \int_0^t\int_{\frZ_0}\big(u_{s},J_s^{\eta}(\kea(x-\cdot)\psi_n)\big)\,\nu_0(d\frz)ds
\end{equation}
$$
+ \int_0^t\int_{\frZ_1}\big(u_{s},J_s^{\xi}(\kea(x-\cdot)\psi_n)\big)\,\nu_1(d\frz)ds
+\int_0^t\int_{\frZ_1}\big(u_{s-},I_s^{\xi}(\kea(x-\cdot)\psi_n)\big)\,\tilde N_1(d\frz,ds) 
$$
for all $t\in[0,T]$.
Then we notice that 
\begin{equation}
\label{Dke}
|\kea (x-y)|\leq\sum_{|\gamma|\leq m+2} |D^\gamma\ke (x-y) |\leq Nk_{2\varepsilon}(x-y),
\end{equation}
as well as that by Assumption for all $x,y\in\bR^d$, 
$s\in [0,T]$, $\frz_i\in\frZ_i$. $i=0,1$ and $n\geq 0$ we have 
$$
\sup_{x\in\bR^d}|D^k\psi_n| = n^{-k}\sup_{\bR^d}|D^k\psi|<\infty,
\quad
\text{for $k\in\bN$,}.
$$
$$
|\tilde{\cL}_s(\kea(x-y)\psi_n(y))| 
+ \sum_k|\cM^k_s(\kea(x-y)\psi_n(y))|^2
\leq N(K_0^2+K_1^2|y|^2+K_1^2|Y_s|^2),
$$
$$
|J_s^{\eta}(\kea(x-y)\psi_n(y))|
\leq\sup_{v\in\bR^d}|D^2_v(\kea(x-v)\psi_n(v))||\eta_{s}(y,\frz_0)|^2
\leq N|\eta_{s}(y,\frz_0)|^2
$$
$$
\leq N\bar{\eta}^2(\frz_0)(K_0^2+K_1^2|y|^2+K_1^2|Y_s|^2),
$$
and 
$$
|J_s^{\xi}(\kea(x-y)\psi_n(y))| + |I^{\xi}_{s}(\kea(x-y)\psi_n(y))|^2
$$
$$
\leq\sup_{v\in\bR^d}|D^2_v(\kea(x-v)\psi_n(v))||\xi_{s}(y,\frz_1)|^2
+ \sup_{v\in\bR^d}|D_v(\kea(x-v)\psi_n(v))|^2|\xi_{s}(y,\frz_1)|^2
$$
$$
\leq N|\xi_{s}(y,\frz_1)|^2 \leq N\bar{\xi}^2(\frz_1)(K_0^2+K_1^2|y|^2+K_1^2|Y_s|^2),
$$
for a constant $N=N(\varepsilon,m,d,K_0,K_1,K,K_\xi,K_\eta)$. 
Using
$$
\esssup_{t\in [0,T]} \int_{\bR^d}(1+|y|^2+|Y_t|^2)|u_t(y)|\,dy<\infty,\quad\text{(a.s.)}
$$
together with the estimates above, we can 
apply Lebesgue's theorem on 
Dominated Convergence to get that for all $x\in\bR^d$, 
$$
(u_t,\kea(x-\cdot)\psi_n)\to (u_t,\kea(x-\cdot)),
\quad (u_0,\kea(x-\cdot)\psi_n)\to (u_0,\kea(x-\cdot))\quad \text{and}
$$
$$
\int_0^t\big(u_s,\cA_s(\kea(x-\cdot)\psi_n)\big)\,ds
\rightarrow 
\int_0^t(u_s,\cA_s\kea(x-\cdot))\,ds
$$
as $n\to\infty$, almost surely uniformly in time, as well as that
$$
\lim_{n\to\infty}\int_0^t\big( u_s,\cM^k_s(\kea(x-\cdot)\psi_n(\cdot))\big)\,dV^k_s
 = \int_0^t\big( u_s,\cM^k_s\kea(x-\cdot)\big)\,dV^k_s,
$$
$$
\lim_{n\to\infty} \int_0^t\int_{\frZ_1}
(u_{s-},I_s^{\xi}(\kea(x-\cdot)\psi_n))\,\tilde N_1(d\frz,ds)
=
\int_0^t\int_{\frZ_1}(u_{s-},I_s^{\xi}\kea(x-\cdot))\,\tilde N_1(d\frz,ds)
$$
in probability, uniformly in time. Thus, letting $n\to\infty$ in \eqref{14.4.22} 
it remains to note that since $\cA$ acts in the $y$ variable,
$$
(u_s,\cA_s\kea(x-\cdot)) = \int_{\bR^d}u_s(y)\cA_s D^\alpha_x \ke(x-y)\,dy 
$$
$$
= D^\alpha_x\int_{\bR^d}u_s(y)\cA_s  \ke(x-y)\,dy =  D^\alpha (\cA^*_su_s)^\ep (x)
$$
for all $(\omega,s,x)\in \Omega\times [0,T]\times \bR^d$ 
if $\cA = \tilde{\cL},\cM^k$ or the identity, as well as 
for all $(\omega,s,x,\frz_i)\in \Omega\times [0,T]\times \bR^d\times\frZ_i$ 
if $\cA=J^\eta$ or $\cA=I^\xi,J^\xi$ with $i=0,1$ respectively.

\end{proof}
\begin{lemma}                                                                                                                           \label{lemma Ito Lp}
Let the Assumptions \ref{assumption SDE}  and \ref{assumption p} hold. 
Let $u$ be an $L_p$-solution of \eqref{equdZ}, 
$p\geq 2$ and assume moreover that $\esssup_{t\in [0,T]}|u_t|_{L_1}<\infty$. 
If $K_1\neq 0$ in Assumption \ref{assumption SDE} (ii), 
then assume additionally \eqref{esssup condition}.
Then for each $\varepsilon>0$ and integer $m\geq 0$, 
for any multi-index $\alpha=(\alpha_1,\dots,\alpha_d)$, $|\alpha|\leq m$, 
almost surely 
$$
|D^\alpha u^\ep_t|_{L_p}^p 
= |D^\alpha u^\ep_0|_{L_p}^p 
+ p\int_0^t 
\big(
|D^\alpha u^\ep_s|^{p-2}D^\alpha u^\ep_s,D^\alpha(\tilde{\mathcal{L}}^*_s  u_s)^\ep 
\big)\,ds
$$
$$
+ p\int_0^t 
\big(
|D^\alpha u^\ep_s|^{p-2}D^\alpha u^\ep_s,D^\alpha(\mathcal{M}_s^{k*} u_s)^\ep 
\big)\,dV^k_s 
$$
$$
+ \tfrac{p(p-1)}{2}\sum_k
\int_0^t
\big(|D^\alpha u_s^\ep|^{p-2}, |D^\alpha(\mathcal{M}_s^{k*} u_s)^\ep|^2 
\big)\,ds
$$
\begin{equation}
\label{Ito formula mu epsilon p norm}
+ p\int_0^t\int_{\frZ_0} 
\big(
|D^\alpha u^\ep_s|^{p-2}D^\alpha u^\ep_s,D^\alpha(J^{\eta *}_s\mu_s)^\ep
 \big)\,\nu_0(d\frz)ds  
\end{equation}
$$
+ p\int_0^t\int_{\frZ_1} 
\big(
|D^\alpha u^\ep_s|^{p-2}D^\alpha u^\ep_s,D^\alpha(J^{\xi *}_s\mu_s)^\ep 
\big)\,\nu_1(d\frz)ds 
$$
$$
+p\int_0^t\int_{\frZ_1} 
\big(|D^\alpha u^\ep_{s-}|^{p-2}D^\alpha u^\ep_{s-},D^\alpha(I^{\xi *}_s u_{s-})^\ep 
\big)\,\Nte(d\frz,ds)
 $$
 \begin{align*}                                                                     
 +\int_0^t\int_{\frZ_1}\int_{\bR^d}
 \Big\{\big|D^\alpha u_{s-}^\ep + D^\alpha(I_s^{\xi *}\mu_{s-})^\ep\big|^p 
&- |D^\alpha u_{s-}^\ep|^p \\
\nonumber
&- p|D^\alpha u_{s-}^\ep|^{p-2}D^\alpha u_{s-}^\ep D^\alpha(I_s^{\xi *} u_{s-})^\ep
\Big\}\,dxN_1(d\frz,ds)
 \end{align*}
holds for all $t\in[0,T]$.
\end{lemma}

\begin{proof}
We  apply the It\^o formula, Theorem 5.1 in \cite{GG2}, 
to $|D^\alpha u^\ep_t|_{L_p}^p$. In order to do that, 
we need to verify that
almost surely for each $x\in \bR^d$ and $\alpha$, 
such that $0\leq |\alpha|\leq m$,
$$
\int_0^T|D^\alpha (\tilde{\cL}_s^* u_s)^{(\varepsilon)}(x)|\,ds<\infty,\quad
\int_0^T\sum_k|D^\alpha(\cM^{k*}_s u_s)^{(\varepsilon)}(x)|^2ds<\infty,
$$
\begin{equation*}
\label{Fubini 1}
\int_0^T\int_{\frZ_0} 
|D^\alpha (J^{\eta *}_s u_s)^{(\varepsilon)}(x)|\,\nu_0(d\frz)\,ds<\infty,
\quad 
\int_0^T\int_{\frZ_1} 
|D^\alpha (J^{\xi *}_s u_s)^{(\varepsilon)}(x)|\,\nu_1(d\frz)\,ds<\infty,
\end{equation*}
$$
\int_0^T\int_{\frZ_1} 
|D^\alpha (I^{\xi *}_s u_s)^{(\varepsilon)}(x)|^2\,\nu_1(d\frz)ds<\infty,
$$
that  for every finite set $\Gamma\in\cB(\bR^d)$, almost surely
$$
\int_{\Gamma}\int_0^T
|D^\alpha (\tilde{\cL}_s^* u_s)^{(\varepsilon)}(x)|\,dxds<\infty,
\quad
\int_{\Gamma}
\Big(
\int_0^T\sum_k|D^\alpha(\cM^{k*}_s u_s)^{(\varepsilon)}(x)|^2ds
\Big)^{1/2}\,dx<\infty,
$$
\begin{equation*}
\label{Fubini 2}
\int_{\Gamma}\int_0^T\int_{\frZ_0} 
|D^\alpha (J^{\eta *}_s u_s)^{(\varepsilon)}(x)|\,\nu_0(d\frz)\,dxds<\infty,
\quad 
\int_{\Gamma}\int_0^T\int_{\frZ_1} 
|D^\alpha (J^{\xi *}_s u_s)^{(\varepsilon)}(x)|\,\nu_1(d\frz)\,dxds<\infty,
\end{equation*}
$$
\int_{\Gamma}
\Big(
\int_0^T\int_{\frZ_1} 
|D^\alpha (I^{\xi *}_s u_s)^{(\varepsilon)}(x)|^2\,\nu_1(d\frz)ds
\Big)^{1/2}\,dx<\infty,
$$
as well as that almost surely 
$$
A:=\int_0^T\int_{\bR^d}
|D^\alpha (\tilde{\cL}_s^* u_s)^{(\varepsilon)}(x)|^p\,dxds<\infty,
$$
$$
A_\eta:= \int_0^T\int_{\bR^d}
\Big|
\int_{\frZ_0} D^\alpha (J^{\eta *}_s u_s)^{(\varepsilon)}(x)\nu_0(d\frz)
\Big|^p\,dxds<\infty, 
$$
$$
A_\xi:= \int_0^T\int_{\bR^d}
\Big|
\int_{\frZ_1} D^\alpha (J^{\xi *}_s u_s)^{(\varepsilon)}(x)\nu_1(d\frz)
\Big|^p\,dxds<\infty, 
$$
$$
B:=\int_0^T\int_{\bR^d}
\big(
\sum_k|D^\alpha(\cM^{k*}_s u_s)^{(\varepsilon)}(x)|^2
\big)^{p/2}\,dxds<\infty, 
$$
$$
G:=\int_0^T\int_{\bR^d}\int_{\frZ_1}
|D^\alpha(I^{\xi *}_s u_s)^{(\varepsilon)}(x,\frz)|^p\,\nu_1(d\frz)dxds<\infty, 
$$
$$
H:=\int_0^T\int_{\bR^d}
\Big(
\int_{\frZ_1}|D^\alpha(I^{\xi *}_s u_s)^{(\varepsilon)}(x,\frz)|^2\,\nu_1(d\frz)
\Big)^{p/2}dxds<\infty.  
$$
For $\alpha=0$ the claim is Lemma 5.4 in \cite{GG2} 
and the estimates can be found in the proof of the 
preceding Lemma 5.2 therein. 
To prove the case  where  $0<|\alpha|\leq m$, 
we note that for $\cA=\tilde{\cL},\cM,I^\xi,J^\xi,J^\eta$ we have
$$
D^\alpha(\cA^*u)^\ep = \int_{\bR^d}D^\alpha_x(\cA_y\ke (x-y))u(y)\,dy 
= \int_{\bR^d}(\cA_y k_{\varepsilon,\alpha}(x-y))u(y)\,dy,
$$
for $k_{\varepsilon,\alpha}(x) = D^\alpha\ke(x)$. 
Hence, a word for word repetition of the proof of 
Lemma 5.2 \& 5.4 in \cite{GG2}, 
where we replace $\ke$ by $k_{\varepsilon,\alpha}$ 
and recall \eqref{Dke}, yields the desired result.
\end{proof}

\begin{lemma}                                                                 \label{lemma Ito Wmp}
Let Assumptions \ref{assumption SDE}, 
\ref{assumption p}, \ref{assumption SDE2} and 
\ref{assumption estimates} hold with an integer 
$m\geq 0$ and let $p\geq 2$ be even. 
Let $u$ be an $W^m_p$-solution to \eqref{equdZ}, 
such that $\E|u_0|^p_{W_p^m}<\infty$ 
and almost surely $\esssup_{t\in [0,T]}|u_t|_{L_1}<\infty$.
Then 
\begin{equation}
\label{supremum estimate wmp}
\E\sup_{t\in [0,T]}|u_t|_{W_p^m}^p\leq N\E|u_0|_{W^m_p}^p
\end{equation}
for a constant 
$N=N(m,d,p,K,K_\eta,K_\xi,L,T,\lambda, |\bar\xi|_{L_2(\frZ_1)},|\bar\eta_{L_2(\frZ_1)})$.
\end{lemma}
\begin{proof}
For $m=0$   the claim is Lemma 5.4 in \cite{GG2}. 
We proceed similarly here. For the present case, 
fix a multi-index $\alpha$ such that $0\neq|\alpha|\leq m$, and define
\begin{align}
Q_p(\alpha, b, \sigma, \rho, \beta,u, k_{\varepsilon})
&=p\big((D^\alpha u^\ep)^{p-1},D^\alpha(\tilde\cL^* u)^\ep \big)                                                
\\
&+\tfrac{p(p-1)}{2}\sum_k
\big(
(D^\alpha u^\ep)^{p-2}, (D^\alpha (\mathcal{M}^{k*}u)^\ep)^2 
\big),  
\end{align}
$$    
\cQ_p^{(0)}(\alpha,\eta(\frz_0),u, k_{\varepsilon})
=p\big(
(D^\alpha u^{\ep})^{p-1},D^\alpha(J^{\eta(\frz_0)*} u)^{\ep}
\big),                          
$$
\begin{equation}                                                                         \label{Q}
\cQ_p^{(1)}(\alpha,\xi(\frz_1), u, k_{\varepsilon})
=
p((D^\alpha u^{\ep})^{p-1}, D^\alpha (J^{\xi(\frz_1)*}u)^{\ep}),
\end{equation}
$$                                                    
\cR_p(\alpha,\xi(\frz_1),u,k_{\varepsilon})
=|D^\alpha u^{\ep} + D^\alpha(I^{\xi(\frz_1)*}u)^{\ep}|^p_{L_p} 
- |D^\alpha u^{\ep}|^p_{L_p} 
- p((D^\alpha u^{\ep})^{p-1},D^\alpha (I^{\xi(\frz_1)*}u)^{\ep}),                   
$$
for $u\in W^m_p$, $\beta\in\bR^{d'}$, 
functions $b$, $\sigma$ and $\rho$ on $\bR^d$, 
with values in $\bR^d$, $\bR^{d\times d_1}$ and $\bR^{d\times d'}$, 
respectively, 
and $\bR^d$-valued functions $\eta(\frz_0)$ 
and $\xi(\frz_1)$ for each $\frz_i\in\frZ_i$, $i=0,1$, 
where $\beta_t=B_t(X_t)$,
$$ 
\tilde\cL=\tfrac{1}{2}(\sigma^{il}\sigma^{jl}+\rho^{ik}\rho^{jk})D_{ij}
+
\beta^l\rho^{il}D_i+\beta^lB^l,   
\quad
\cM^k=\rho^{ik}D_i+B^k, \quad k=1,2,...,d'. 
$$

By Lemma \ref{lemma Ito Lp} almost surely
$$
d|D^\alpha u_t^{(\varepsilon)}|^p_{L_p}               
=\cQ_p(\alpha,b_t, \sigma_t, \rho_t, \beta_t,u_t, k_{\varepsilon})\,dt                                                                  
 +\int_{\frZ_0}
 \cQ_p^{(0)}(\alpha,\eta_t(\frz), u_t, k_{\varepsilon}) \,\nu_0(d\frz)\,dt 
 $$
 \begin{equation}                                                                                            \label{7.23.1}                                                                                       
 +\int_{\frZ_1}\cQ_p^{(1)}(\alpha,\xi_t(\frz), u_t, k_{\varepsilon})\,\nu_1(d\frz)\,dt
+\int_{\frZ_1}\cR_p(\alpha,\xi_{t}(\frz), u_{t-}, k_{\varepsilon})\,N_1(d\frz,dt) 
+d\zeta_1(\alpha,t)+d\zeta_2(\alpha,t),                        
\end{equation}
for all $t\in [0,T]$  and 
\begin{equation}
\label{30.9.21.7}
\zeta_1(\alpha,t)
=p\int_0^t 
\big(
(D^\alpha u^\ep_s)^{p-1},D^\alpha(\mathcal{M}_s^{k*} u_s)^\ep 
\big)\,dV^k_s, 
\end{equation}
$$
\zeta_2(\alpha,t)=p\int_0^t\int_{\frZ_1} 
\big(
(D^\alpha u^\ep_{s})^{p-1},D^\alpha (I_s^{\xi *} u_{s})^\ep 
\big)\,\Nte(d\frz,ds)\quad t\in[0,T]   
$$
are local martingales under $P$. 
We write 
\begin{equation}                                                                 \label{7.23.2}
\int_{\frZ_1}
\cR_p(\alpha, \xi_{t}(\frz_1), u_{t-}, k_{\varepsilon})\,N_1(d\frz,dt)
=\int_{\frZ_1}
\cR_p(\alpha,\xi_{t}(\frz_1), u_{t-}, k_{\varepsilon})\,\nu_1(d\frz)dt
+d\zeta_3(\alpha,t)
\end{equation}
with 
$$
\zeta_3(\alpha,t)
=\int_0^t\int_{\frZ_1}\cR_p(\alpha,\xi_s(\frz), u_{s-}, k_{\varepsilon})\,N_1(d\frz,ds)
-\int_0^t\int_{\frZ_1}\cR_p(\alpha,\xi_s(\frz), u_{s-}, k_{\varepsilon})\,\nu_1(d\frz)ds, 
$$
which we can justify if we show 
\begin{equation}
\label{6.10.21.3}
A:=\int_0^T
\int_{\frZ_1}
|\cR_p(\alpha,\xi_s(\frz), u_s, k_{\varepsilon})|\,
\nu_1(d\frz)\,ds<\infty \,\text{(a.s.)}.  
\end{equation}
To this end observe that by 
Taylor's formula 
\begin{equation}                                                                                    \label{7.14.4}                                                                                                  
0\leq\cR_p(\alpha,\xi_t(\frz), u_t, k_{\varepsilon}))
\leq N\int_{\bR^d}
(D^\alpha u^{(\varepsilon)}_t)^{p-2}(
D^\alpha(I^{\xi(\frz) *}u_t)^{(\varepsilon)})^{2}
+(D^\alpha (I^{\xi(\frz) *}u_t)^{(\varepsilon)})^{p}\,dx 
\end{equation}
with a constant $N=N(d,p)$. Hence 
$$
\int_{\frZ_1}\cR_p(\alpha,\xi_t(\frz), u_t, k_{\varepsilon}))\,\nu_1(d\frz)
$$
$$
\leq N\int_{\bR^d}
(D^\alpha u^{(\varepsilon)}_t)^{p-2}|
D^\alpha (I^{\xi(\frz) *}u_t)^{(\varepsilon)}|_{L_2(\frZ_1)}^{2}
+|D^\alpha (I^{\xi(\frz) *}u_t)^{(\varepsilon)}|^{p}_{L_p(\frZ_1)}\,dx
$$
$$
\leq N'\big(
|D^\alpha u_t^{(\varepsilon)}|^p_{L_p}+A_1(t)+A_2(t)\big)
$$
with
\begin{equation}
\label{6.10.21.2}
A_1(t)
=\int_{\bR^d}|D^\alpha (I^{\xi(\frz) *}u_t)^{(\varepsilon)}|^{p}_{L_2(\frZ_1)}\,dx,
\quad
A_2(t)=
\int_{\bR^d}|D^\alpha (I^{\xi(\frz) *}u_t)^{(\varepsilon)}|^{p}_{L_p(\frZ_1)}\,dx
\end{equation}
and constants $N$ and $N'$ depending only on $d$ and $p$. 
By Minkowski's inequality and using again 
that $D^\alpha_xI^\xi\ke(x-y) = I^\xi D^\alpha_x\ke(x-y)$,
$$
|D^\alpha u_t^{(\varepsilon)}|^p_{L_p}
=\int_{\bR^d}\Big|
\int_{\bR^d}D^\alpha_x k_{\varepsilon}(x-y)\,u_t(y)\,dy
\Big|^p\,dx,
$$
\begin{equation}                                                                              \label{7.14.1}
\leq
\Big|
\int_{\bR^d}|D^\alpha k_{\varepsilon}|_{L_p}\,|u_t(y)|\,dy
\Big|^p
\leq |D^\alpha k_{\varepsilon}|_{L_p}^p |u_t|_{L_1}^p, 
\end{equation}
$$
A_1(t)=\int_{\bR^d}
\Big|
\int_{\frZ_1}
\big|
\int_{\bR^d}
\big(
(D^\alpha k_{\varepsilon})(\cdot-y-\xi_t(y,\frz))-(D^\alpha k_{\varepsilon})(\cdot-y)
\big)\,u_t(y)\,dy 
\big|^2
\nu_1(d\frz)
\Big|^{p/2}dx
$$
$$
\leq \Big|\int_{\frZ_1}
\Big|
\int_{\bR^d}
|(D^\alpha k_{\varepsilon})(\cdot-y-\xi_t(y,\frz))-(D^\alpha k_{\varepsilon})(\cdot-y)|
|u_t(y)|\,dy
\Big|^{2}_{L_p}
\nu_1(d\frz)\Big|^{p/2}
$$
$$
\leq \Big|
\int_{\frZ_1}
\Big|
\int_{\bR^d}
|D^{\alpha+1} k_{\varepsilon}|_{L_p}\bar\xi(\frz_1) 
(K_0+K_1|y|+K_1|Y_t|) |u_t(y)|\,dy
\Big|^{2}
\nu_1(d\frz)\Big|^{p/2}
$$
\begin{equation}                                                                              \label{7.14.2}
\leq 
|D^{\alpha +1}k_{\varepsilon}|^p_{L_p}
|\bar\xi|_{L_2(\frZ_1)}^p\Big(\int_{\bR^d}
(K_0 + K_1|y|+K_1|Y_t|)|u_t(y)|\,dy \Big)^p,
\end{equation}
$$
$$
where $D^{\alpha+1} = D\,D^\alpha$ and similarly, 
using Assumption \ref{assumption p},
$$
A_2(t)= \int_{\bR^d}
\int_{\frZ_1}
\Big|
\int_{\bR^d}
\big(
(D^\alpha k_{\varepsilon})(x-y-\xi(t,y,\frz))-(D^\alpha k_{\varepsilon})(x-y)
\big)\,u_t(y)\,dy
\Big|^{p}
\nu_1(d\frz)dx
$$
$$
\leq\int_{\frZ_1}
\Big|\int_{\bR^d}
\big|
(D^\alpha k_{\varepsilon})(\cdot-y-\xi(t,y,\frz))-(D^\alpha k_{\varepsilon})(\cdot-y))
\big|_{L_p}|u_t(y)|\,dy
\Big|^{p}
\nu_1(d\frz)
$$
\begin{equation}                                                                              \label{7.14.3}
\leq K^{p-2}_\xi|D^{\alpha +1}k_{\varepsilon}|^p_{L_p}
|\bar\xi|^2_{L_2(\frZ_1)}\Big(\int_{\bR^d}
(K_0 + K_1|y|+K_1|Y_t|)|u_t(y)|\,dy \Big)^p.
\end{equation}                                                                            
By \eqref{7.14.4}--\eqref{7.14.3}  we have a constant 
$N=N(p,d,\varepsilon,|\bar\xi|_{L_2{(\frZ_1)}},K_\xi)$ such that 
$$
A\leq N\int_0^T |D^\alpha u_t^\ep|_{L_p}^p 
+ N\int_0^T
\Big(\int_{\bR^d}(K_0 + K_1|y|+K_1|Y_t|)|u_t(y)|\,dy 
\Big)^{p}dt<\infty \,\text{(a.s.)}.   
$$
Next we claim  that, with the operator $T^{\xi}$ 
defined in \eqref{equ operators TIJ}, we have
\begin{equation}                                                                    \label{7.19.1}
\zeta_2(\alpha,t)+\zeta_3(\alpha,t)
=\int_0^t\int_{\frZ_1}
|D^\alpha(T^{\xi *}_s  u_s)^{(\varepsilon)}|^p_{L_p}
-|D^\alpha u_s^{(\varepsilon)}|^p_{L_p}
\tilde N_1(d\frz,ds)=:\zeta(\alpha, t)\quad\text{for $t\in[0,T]$}. 
\end{equation}
For that purpose not first that $D^\alpha u^{\ep} 
+ D^\alpha(I^{\xi(\frz_1)*}u)^{\ep} = D^\alpha(T^{\xi *}  u_s)^{(\varepsilon)}$. 
To see that the stochastic integral $\zeta(\alpha,t)$  
is well-defined as an It\^o integral 
note that by Lemma \ref{lemma dpe4},
\begin{equation}                                                                               \label{7.27.1}                                                          
\int_0^T\int_{\frZ_1}||D^\alpha (T^{\xi *}u_s)^{(\varepsilon)}|^p_{L_p}
-|D^\alpha u_s^{(\varepsilon)}|^p_{L_p}|^2\,\nu_1(d\frz)ds
\leq N|\bar\xi|^2_{L_2(\frZ_1)}
\int_{0}^T|u_s|^{2p}_{W^m_p}\,ds<\infty\,\text{(a.s.)}
\end{equation}
with a constant $N=N(d,p,m,\lambda,K_\xi)$. 
Since $\frZ_1$ is $\sigma$-finite, 
there is an increasing sequence $(\frZ_{1n})_{n=1}^{\infty}$,  
$\frZ_{1n}\in\cZ_1$,  
such that $\nu_1(\frZ_{1n})<\infty$ for every $n$ 
and $\cup_{n=1}^{\infty}\frZ_{1n}=\frZ_1$. 
Then it is easy to see that 
$$
\bar\zeta_{2n}(\alpha,t)=p\int_0^t\int_{\frZ_1}{\bf1}_{\frZ_{1n}}(\frz)
\big((D^\alpha u^\ep_{s})^{p-1},D^\alpha (I_s^{\xi *} u_{s})^\ep \big)\,N(d\frz,ds), 
$$
$$
\hat\zeta_{2n}(\alpha,t)=p\int_0^t\int_{\frZ_1}{\bf1}_{\frZ_{1n}}(\frz)
\big((D^\alpha u^\ep_{s})^{p-1},D^\alpha (I_s^{\xi *} u_{s})^\ep \big)\,\nu_1(d\frz)ds, 
$$
$$
\bar\zeta_{3n}(\alpha,t)=\int_0^t\int_{\frZ_1}
{\bf1}_{\frZ_{1n}}(\frz)\cR_p(\alpha,\xi_s(\frz), u_{s-}, k_{\varepsilon})\,N_1(d\frz,ds),
$$
$$
\hat\zeta_{3n}(\alpha,t)=\int_0^t\int_{\frZ_1}{\bf1}_{\frZ_{1n}}(\frz)
\cR_p(\alpha,\xi_s(\frz), u_{s-}, k_{\varepsilon})\,\nu_1(d\frz)ds
$$
are well-defined, and 
$$
\zeta_2(\alpha,t)
=\lim_{n\to\infty}(\bar\zeta_{2n}(\alpha,t)-\hat\zeta_{2n}(\alpha,t)), 
\quad
\zeta_3(\alpha,t)
=\lim_{n\to\infty}\bar\zeta_{3n}(\alpha,t)-\lim_{n\to\infty}\hat\zeta_{3n}(\alpha,t), 
$$
where the limits are understood in probability. 
Hence 
$$
\zeta_2(\alpha,t)+\zeta_3(\alpha,t)=\lim_{n\to\infty}\Big(\bar\zeta_{2n}(\alpha,t)
+\bar\zeta_{3n}(\alpha,t)-\big(\hat\zeta_{2n}(t)+\hat\zeta_{3n}(\alpha,t)\big)\Big)
$$
$$
=\lim_{n\to\infty}\Big(\int_0^t\int_{\frZ_1}{\bf1}_{\frZ_{1n}}(\frz)
(|D^\alpha (T^{\xi *} u_s)^{(\varepsilon)}|^p_{L_p}-|D^\alpha u_s^{(\varepsilon)}|^p_{L_p})
\tilde N_1(d\frz,ds)\Big)=\zeta(\alpha,t), 
$$
which completes the proof of \eqref{7.19.1}. 
Consequently, from \eqref{7.23.1}-\eqref{7.23.2}
we have
$$
d|D^\alpha u_t^{(\varepsilon)}|^p_{L_p}               
=\cQ_p(\alpha,b_t, \sigma_t, \rho_t, \beta_t,u_t, k_{\varepsilon})\,dt                                                                  
 +\int_{\frZ_0}
 \cQ_p^{(0)}(\alpha,\eta_t(\frz_0),u_t, k_{\varepsilon}) \,\nu_0(d\frz)\,dt 
 $$
 \begin{equation}                                                                                  \label{7.23.3}                                                                                                                                                                              
 +\int_{\frZ_1}\cQ_p^{(1)}(\alpha,\xi_t(\frz_1), u_t, k_{\varepsilon})
 +\cR_p(\alpha,\xi_{t}(\frz_1), u_{t}, k_{\varepsilon})\,\nu_1(d\frz)\,dt
+d\zeta_1(\alpha,t)+d\zeta(\alpha,t).                       
\end{equation}
By  Lemma \ref{lemma dpe1}, Corollary \ref{p3 cor lemma dpe1} 
and Lemma \ref{lemma pe5} we have 
\begin{equation}                                                                    \label{7.27.3}
Q_p(\alpha,b_s, \sigma_s, \rho_s, \beta_s,u_s, k_{\varepsilon})
\leq N(L^2+K^2)|u_s|_{W^m_p}^p
\end{equation}
with a constant $N=N(d,p,m)$, 
and by Lemma \ref{lemma dpe3} and Corollary \ref{corollary dJ}, 
using that $\bar\xi\leq K_\xi$ and $\bar\eta\leq K_\eta$,
\begin{equation}                                                                    \label{7.27.4}
\cQ^{(0)}_p(\alpha,\eta_s(\frz), u_s, k_{\varepsilon})
\leq N{\bar\eta}^2(\frz)|u_s|_{W^m_p}^p, 
\quad
(\cQ^{(1)}_p+\cR_p)(\alpha,\xi_s(\frz), u_s, k_{\varepsilon})
\leq N{\bar\xi}^2(\frz)|u_s|_{W^m_p}^p  
\end{equation}
with a constant $N=N(K_\xi,K_\eta,d,p,\lambda,m)$. 
Thus from \eqref{7.23.3} we obtain that 
or all $\alpha$ with $|\alpha|\leq m$ almost surely 
$$
|D^\alpha u_t^\ep|_{L_p}^p
\leq |u^{(\varepsilon)}_0|^p_{W_p^m}
+N\int_0^t|u_s|_{W_p^m}^p\,ds+m^{(\varepsilon)}_t
\quad
\text{for all $t\in[0,T]$}
$$
with a constant 
$N=N(m,p,d,K,K_\xi,K_\eta,L,\lambda,|\bar\xi|_{L_2(\frZ_1)},|\bar\eta|_{L_2(\frZ_0)})$ 
and the local martingale 
$m^{(\varepsilon)}(\alpha,t)=\zeta_1(\alpha,t)+\zeta(\alpha,t)$. 
Summing over all $|\alpha|\leq m$ gives
\begin{equation}                                                                                            \label{Ctp}
|u_t^\ep|_{W_p^m}^p
\leq |u^{(\varepsilon)}_0|^p_{W_p^m}
+N\int_0^t|u_s|_{W_p^m}^p\,ds+m^{(\varepsilon)}_t
\quad
\text{for all $t\in[0,T]$}
\end{equation}
with (another) constant 
$N=N(m,p,d,K,K_\xi,K_\eta,L,\lambda,|\bar\xi|_{L_2(\frZ_1)},|\bar\eta|_{L_2(\frZ_0)})$ 
and a local martingale, denoted again by $m^\ep$.
For integers $n\geq1$ set 
$\tau_n=\bar\tau_n\wedge\tilde\tau_n$, where 
$(\tilde\tau_n)_{n=1}^{\infty}$ is a localising sequence 
of stopping times for $m^{(\varepsilon)}$ 
and 
$$
\bar\tau_n=\inf\Big\{t\in[0,T]:\int_0^t|u_s|_{W^m_p}^{p}\,ds\geq n\Big\}. 
$$
Then from  \eqref{Ctp}, using also $|D^\alpha u^\ep|_{L_p} 
= |(D^\alpha u)^\ep|_{L_p}\leq |D^\alpha u|_{L_p}$ 
for multi-indices $\alpha\leq m$ and $\varepsilon>0$ we get 
$$
\E |u_{t\wedge\tau_n}^\ep|_{W^m_p}^p
\leq \E|u_0|^p_{W^m_p}+N\int_0^t\E |u_{s\wedge\tau_n}|_{W^m_p}^p\,ds<\infty
\quad
\text{for $t\in[0,T]$ and integers $n\geq1$}.
$$
Applying Fatou's lemma for the limit $\varepsilon\to 0$ followed 
by Gr\"onwall's lemma gives
$$
\E |u_{t\wedge\tau_n}|^p_{W^m_p}\leq N\E|u_0|_{W^m_p}^p
\quad
\text{for $t\in[0,T]$ and integers $n\geq1$}
$$
with a constant 
$N=N(m,p,d,T,K,K_\xi,K_\eta,L,\lambda,|\bar\xi|_{L_2},|\bar\eta|_{L_2})$. 
Letting here $n\to\infty$, by Fatou's lemma we obtain 
\begin{equation}                                                                                           \label{7.27.2}                                                                                         
\sup_{t\in[0,T]}\E|u_t|^p_{W^m_p}\leq N\E|u_0|^p_{W^m_p}. 
\end{equation}

To prove \eqref{supremum estimate wmp} 
we define a localizing sequence of stopping times 
$(\tau^\varepsilon_k)_{k=1}^\infty$ for the local martingale $m^\varepsilon$, 
as well as
$$
\tilde\rho_n
=\inf\Big\{t\in[0,T]:\int_0^t|u_s|_{W^m_p}^{2p}\,ds\geq n\Big\},
\quad\text{and}\quad \rho_{n,k}^\varepsilon
=\bar{\rho}_n\wedge\tilde{\rho}_k^\varepsilon. 
$$
Using the Davis inequality and Lemma \ref{lemma pe5} by standard calculations 
for every $n\geq1$ we get for each $|\alpha|\leq m$ for the Doob-Meyer process of $\zeta_1$,
\begin{equation}
\label{equ dm1}
\E\sup_{t\leq T}|\zeta_1(\alpha,t\wedge\rho_{n,k}^\varepsilon)|
\leq 3\E\Big(\sum_k \int_0^{T\wedge\rho_{n,k}^\varepsilon} 
\big((D^\alpha u^\ep_s)^{p-1},D^\alpha (\cM_s^{k *} u_s)^\ep \big)^2\,ds \Big)^{1/2}
\end{equation}
$$
\leq N\E\Big(\int_0^{T\wedge\rho_{n,k}^\varepsilon} |u_s|_{W^m_p}^{2p}\,ds \Big)^{1/2}<\infty,
$$
and similarly, for each $|\alpha|\leq m$, 
the Doob-Meyer process of $\zeta(\alpha,\cdot)$ is
$$
\langle\zeta (\alpha,\cdot)\rangle(t)=\int_0^t\int_{\frZ_1}|
|D^\alpha (T^{\xi *} u_s)^{(\varepsilon)}|^p_{L_p}
-|D^\alpha u_s^{(\varepsilon)}|^p_{L_p}|^2\nu_1(\frz)ds, \quad t\in[0,T].
$$
Using the Davis inequality and Lemma \ref{lemma dpe4},
\begin{equation}
\label{equ dm2}
\E\sup_{s\leq T}|\zeta(\alpha,s\wedge\rho_{n,k}^\varepsilon)|
\leq 3\E\langle\zeta(\alpha,\cdot)\rangle^{1/2}(T\wedge\rho_{n,k}^\varepsilon)
\leq 
N\E\Big(\int_0^{T\wedge\rho_{n,k}^\varepsilon} |u_s|_{W^m_p}^{2p}\,ds \Big)^{1/2}<\infty,
\end{equation}
with a constant $N=N(m,d,p,K,K_\xi,L,\lambda, |\bar\xi|_{L_2(\frZ_1)})$. 
Thus, due to \eqref{7.27.2} together with \eqref{equ dm1} and \eqref{equ dm2}, 
we get from \eqref{Ctp}, with constant 
$N=N(m,p,d,T,K,K_\xi,K_\eta,L,\lambda,|\bar\xi|_{L_2},|\bar\eta|_{L_2})$,
$$
\E\sup _{t\in[0,T]}
|u^\ep_{t\wedge \rho_{n,k}^\varepsilon}|^p_{W^m_p}
\leq N\E|u_0|^p_{W^m_p} 
+ \sum_{|\alpha|\leq m}
\E\sup_{t\leq T}|\zeta_1(\alpha,t\wedge\rho_{n,k}^\varepsilon)| 
+ \sum_{|\alpha|\leq m}
\E\sup_{t\leq T}|\zeta(\alpha,t\wedge\rho_{n,k}^\varepsilon)|
$$
$$
\leq N\E|u_0|^p_{W^m_p} 
+ N\E\Big(
\int_0^T |u_{s\wedge\rho_{n}}|_{W^m_p}^{2p}\,ds 
\Big)^{1/2}.
$$
Letting here $k\to\infty$ and then $\varepsilon\to 0$, 
we obtain by Fatou's lemma with constants $N'$ 
and $N''$ only depending on $m$, $p$, $d$, $T$, 
$K$, $K_\xi$, $K_\eta$, $L$,  
$\lambda$, $|\bar\xi|_{L_2(\frZ_1)}$ and $|\bar\eta|_{L_2(\frZ_0)}$,
$$
\E\sup _{t\in[0,T]}|u_{t\wedge \rho_n}|^p_{W^m_p}
\leq N\E|u_0|^p_{W^m_p} 
+ N\E\Big(\int_0^T |u_{s\wedge\rho_{n}}|_{W^m_p}^{2p}\,ds \Big)^{1/2}
$$

$$
\leq N\E|u_0|^p_{W^m_p} 
+ N\E\Big(\sup_{t\in [0,T]}
|u_{t\wedge\rho_{n}}|_{W^m_p}^{p}\int_0^T |u_{s\wedge\rho_{n}}|_{W^m_p}^{p}\,ds \Big)^{1/2}
$$
$$
\leq N\E|u_0|^p_{W^m_p} 
+ \tfrac{1}{2}\E\sup_{t\in [0,T]}|u_{t\wedge\rho_{n}}|_{W^m_p}^{p} 
+ N'\E\int_0^T |u_{s\wedge\rho_{n}}|_{W^m_p}^{p}\,ds 
$$
$$
\leq N''\E|u_0|^p_{W^m_p} + \tfrac{1}{2}\E\sup_{t\in [0,T]}|u_{t\wedge\rho_n}|_{W^m_p}^{p},
$$
where we used Young's inequality.
Thus also, we get for all $n$,
$$
\E\sup _{t\in[0,T]}|u_{t\wedge \rho_n}|^p_{W^m_p}\leq 2N''\E|u_0|^p_{W^m_p}.
$$
Using Fatou's lemma we get the desired result.
\end{proof}

The following Lemma \ref{lemma compact} is Lemma 6.4 in \cite{GG2}. 
For integers $m\geq 0$ and real numbers $p\geq 1$ 
we define $\bW_p^m=\bW_p^m(\bR^d)$ to be the space of 
$\cF_T\otimes \cB(\bR^d)$-measurable real valued random variables $\psi$ such that
$$
|\psi|_{\bW^m_p}^p:=\E\sum_{k=0}^m\int_{\bR^d}|D^k\psi(x)|^p\,dx<\infty.
$$
For $p,q\geq 1$ and integers $m\geq 0$ we denote by $\bW_{p,q}^m$ 
the space of $\cO\otimes \cB(\bR^d)$-measurable real valued functions $v=v_t(\omega,x)$ such that 
$$
|v|_{\bW^m_{p,q}}^p:=\E\Big(\int_0^T|v_t|_{W_p^m}^q\,dt\Big)^{p/q} <\infty.
$$
If $m=0$ then we write $\bL_{p,q}:=\bW^0_{p,q}$.
Let $\bB_0^m$ denote the space of those functions 
$\psi\in \bigcap_{p\geq1}\bW^m_p$ such that 
$$
\sum_{k=0}^m\sup_{\omega\in\Omega}\sup_{x\in\bR^d}|D^k\psi(x)|<\infty
\quad
\text{and almost surely $\psi(x)=0$ for $|x|\geq R$,}
$$ 
for some constant $R$ depending on $\psi$.  
It is easy to see that $\bB^m_0$ is a dense 
subspace of $\bW^m_p$ for every $p\in[1,\infty)$. 
For $\varepsilon>0$ let in the following proposition $v^\ep$ denote the convolution
$$
v^\ep(x) = \int_{\bR^d}\chi_\varepsilon(x-y)v(y)\,dy
$$
of a Borel function $v$ on $\bR^d$, where $\chi$ is a smooth, 
symmetric function of unit integral on $\bR^d$, such that $\chi(x)=0$ for $|x|\geq 1$ and
 $\chi_\varepsilon(\cdot):=\varepsilon^{-d}\chi(\cdot/\varepsilon)$. Let
$$
\cM_t^{\varepsilon k} = \rho^{(\varepsilon) ik}_t D_i 
+ B_t^{(\varepsilon)k},\quad k=1,\dots,d',
$$
$$
\tilde\cL_t^{\varepsilon}=a_t^{\varepsilon, ij}D_{ij}
+b_t^{(\varepsilon)i}D_i+
\beta^k_t\cM^{\varepsilon k}_t, 
\quad
\beta_t=B(t,X_t,Y_t),  
$$
$$
a_t^{\varepsilon, ij}
:=\tfrac{1}{2}\sum_k (\sigma_t^{(\varepsilon)ik}\sigma_t^{(\varepsilon)jk}
+\rho_t^{(\varepsilon)ik}\rho_t^{(\varepsilon)jk}), 
\quad i,j=1,2,...,d
$$
and let $I^{\xi^\varepsilon}, J^{\xi^\varepsilon}$ 
and $J^{\eta^\varepsilon}$ be defined as $I^\xi,J^\xi$ and $J^\eta$, 
only with $\xi^\ep$ and $\eta^\ep$ instead of $\xi$ and $\eta$, respectively.

Consider for $\varepsilon\in(0,1)$ the equation 
\begin{align}                                                                       
du_t^{\varepsilon}=&\tilde\cL_t^{\varepsilon\ast}u_t^{\varepsilon}\,dt
+\cM_t^{\varepsilon k\ast}u_t^{\varepsilon}\,dV^k_t
+\int_{\frZ_0}J_t^{\eta^{\varepsilon}\ast}u_t^{\varepsilon}\,\nu_0(d\frz)dt                        \nonumber\\
&+\int_{\frZ_1}J_t^{\xi^{\varepsilon}*}u_t^{\varepsilon}\,\nu_1(d\frz)dt
+\int_{\frZ_1}I_t^{\xi^{\varepsilon}*}u_t^{\varepsilon}\,\tilde N_1(d\frz,dt),
\quad
\text{with $u_0^{\varepsilon}=\psi^{(\varepsilon)}$.}                                                           \label{equdZn}
\end{align}

\begin{proposition}                                                                             \label{proposition compact}
Let Assumptions \ref{assumption SDE},  
\ref{assumption p} and \ref{assumption estimates} 
hold with  $K_1=0$ and let $p\geq 2$ be even.
Assume that the following ``support condition" holds: There is 
some $R>0$ such that 
\begin{equation}                                                                                            \label{supp_condition}
\big(b_t(x),B_t(x),\sigma_t(x),\rho_t(x), \eta_t(x,\frz_0),\xi_t(x,\frz_1)\big)=0
\end{equation}
for $\omega\in\Omega$, $t\geq0$, $\frz_0\in\frZ_0$, $\frz_1\in\frZ_1$ and  
$x\in\bR^{d}$ such that 
$|x|\geq R$.
Let $\psi\in\bB^m_0$ such that almost surely 
$\psi(x)=0$ for $|x|\geq R$. Then there exist $\varepsilon_0>0$ 
and a $\bar{R}=\bar{R}(R,K,K_0,K_\xi,K_\eta)$ 
such that the following statements hold.\newline
(i) For each $\varepsilon\in (0,\varepsilon_0)$ there exists a 
$W^r_p$-solution $u^\varepsilon$ to \eqref{equdZn}, 
for every $r\geq 1$, with initial condition $u_0^\varepsilon=\psi^\ep$ 
and such that 
\begin{equation*}
%\label{E sup Wmp compact}
\E\sup_{t\in [0,T]}|u^\varepsilon_t|_{W^r_p}^p<\infty
\quad
\text{and}\quad u_t^\varepsilon(x)=0
\quad
\text{almost surely for $|x|\geq \bar{R}$ and $t\in [0,T]$.}
\end{equation*}
(ii) There exists a unique $L_p$-solution $u$ to \eqref{equdZ} 
(with non-smoothed coefficients) such that almost surely 
$u_t(x)=0$ for $dx$-almost every  
$x\in\{x\in\bR^d:|x|\geq \bar R\}$ for every $t\in[0,T]$  and 
\begin{equation*}                                                                                                   %\label{E sup Lp compact}
\E\sup_{t\in [0,T]}|u_t|_{L_p}^p\leq N\E|\psi|_{L_p}^p
\end{equation*}
with a constant 
$N=N(d,p,T,K, K_{\xi}, K_{\eta}, L,\lambda,|\bar{\xi}|_{L_2},|\bar{\eta}|_{L_2})$.
\newline
(iii) There exists a sequence $(\varepsilon_n)_{n=1}^\infty$, $\varepsilon\to 0$, such that 
$$
u^{\varepsilon_n}\to u
\quad 
\text{weakly in $\bL_{p,q}$ as $n\to\infty$, for every integer $q\geq 2$.}
$$
\end{proposition}
\begin{proof}
See Lemma 6.4 in \cite{GG2}.
\end{proof}
\begin{lemma}
\label{lemma compact}
Let Assumptions \ref{assumption SDE},  \ref{assumption p}, 
\ref{assumption SDE2}  and \ref{assumption estimates} 
hold with  $K_1=0$. Consider integers $m\geq 0$ and $p\geq 2$ even. 
Let moreover the support condition \eqref{supp_condition} 
of Proposition \ref{proposition compact} hold for some $R>0$.
Then there exists a unique $W^m_p$-solution $(u_t)_{t\in [0,T]}$ 
to equation \eqref{equdZ} with initial condition $u_0=\psi$. 
Moreover, almost surely $u_t(x)=0$ for $dx$-almost every  
$x\in\{x\in\bR^d:|x|\geq \bar R\}$ for every $t\in[0,T]$ 
for a constant $\bar{R}=\bar{R}(R,K,K_0, K_{\xi}, K_{\eta})$, and 
\begin{equation}                                                                                                   \label{E sup Wmp compact}
\E\sup_{t\in [0,T]}|u_t|_{W^m_p}^p\leq N\E|\psi|_{W^m_p}^p
\end{equation}
with a constant 
$N=N(m,d,p,T,K, K_{\xi}, K_{\eta}, L,\lambda,|\bar{\xi}|_{L_2},|\bar{\eta}|_{L_2})$.
\end{lemma}

\begin{proof} By Proposition \ref{proposition compact} 
(i) for $\varepsilon>0$ sufficiently small there exists a $W^m_p$-valued 
weakly cadlag $\cF_t$-adapted process $(u^\varepsilon_t)_{t\in [0,T]}$, such that 
for each $\varphi\in C_0^{\infty}$ almost surely 
\begin{align}                                                                       
(u_t^{\varepsilon},\varphi)=&(\psi^{(\varepsilon)},\varphi)+
\int_0^t(u_s^{\varepsilon},\tilde\cL_s^{\varepsilon}\varphi)\,ds
+\int_0^t(u_s^{\varepsilon},\cM_s^{\varepsilon k}\varphi)\,dV^k_s
+\int_0^t\int_{\frZ_0}
(u_s^{\varepsilon},J_s^{\eta^{\varepsilon}}\varphi)\,\nu_0(d\frz)\,ds                        \nonumber\\
&+\int_0^t\int_{\frZ_1}
(u_s^{\varepsilon},J_s^{\xi^{\varepsilon}}\varphi)\,\nu_1(d\frz)\,ds
+\int_0^t\int_{\frZ_1}
(u_s^{\varepsilon},I_s^{\xi^{\varepsilon}}\varphi)\,\tilde N_1(d\frz,ds),   \label{equdZ*}
\end{align}
holds for all $t\in[0,T]$. By Proposition \ref{proposition compact} (ii), 
since almost surely $u^\varepsilon_t=0$ for $|x|\geq \bar{R}$ 
for all $t\in [0,T]$ for a constant $\bar{R}
=\bar{R}(R,K,K_0, K_{\xi}, K_{\eta})$, we also have
$$
\E\sup_{t\in [0,T]}|u_t^\varepsilon|_{L_1}
\leq \bar{R}^{d/q} \E\sup_{t\in [0,T]}|u_t^\varepsilon|_{L_p}^p
$$
for $q=p/(p-1)$. Next, note that the smoothed coefficients 
$b^\ep,B^\ep,\sigma^\ep,\rho^\ep,\xi^\ep$ and $\eta^\ep$ satisfy 
Assumptions \ref{assumption SDE}, \ref{assumption p},  
\ref{assumption SDE2} and Assumption \ref{assumption estimates} 
(ii) \& (iii) with the same constants $K_0,L,K_\xi$ and $K_\eta$, 
independent of $\varepsilon$.
By Remark \ref{remark diffeom} (i) we have that 
for all $t\in [0,T],\theta\in [0,1],y\in\bR^{d'}$ and $\frz_i\in\frZ_i$, $i=0,1$, the mappings
$$
\tau^{\eta}_{t,\theta,\frz_0}(x) 
= x+\theta\eta^\ep_t(x,\frz_0),\quad\text{and}\quad \tau^{\xi}_{t,\theta,\frz_1}(x) 
= x+\theta\xi^\ep_t(x,\frz_1)
$$
are $C^1$-diffeomorphisms.
Moreover, by Lemma 6.2 in \cite{GG2}, 
we know that for $\varepsilon$ sufficiently small 
we have that for all $t\in [0,T],\theta\in [0,1]$ and $\frz_i\in\frZ_i$, $i=0,1$, 
the mappings
$$
(\tau_{t,\theta,\frz_0}^\eta)^\ep=\tau^{\eta^\ep}_{t,\theta,\frz_0}(x) 
= x+\theta\eta^\ep_t(x,\frz_0)\quad\text{and}\quad 
(\xi_{t,\theta,\frz_1}^\xi)^\ep=\tau^{\xi^\ep}_{t,\theta,\frz_1}(x) = x+\theta\xi^\ep_t(x,\frz_1)
$$
remain $C^1$-diffeomorphisms such that 
$$
|\det D\tau^{\eta^\ep}_{t,\theta,\frz_0}(x)| \geq \lambda' \quad\text{and}
\quad |\det D\tau^{\xi^\ep}_{t,\theta,\frz_1}(x)| \geq \lambda',
$$
with a $\lambda'=\lambda'(\lambda,K_\xi,K_\eta,K_0)$ 
independent of $\varepsilon$. By Remark \ref{remark diffeom} 
(ii) we then know that Assumption \ref{assumption estimates} (i) is satisfied with (another) $\lambda''=\lambda''(\lambda,K_\xi,K_\eta,K_0)$ independent of $\varepsilon$.
Hence by Lemma \ref{lemma Ito Wmp} for each $\varepsilon>0$ also
\begin{equation}
\label{equ E Wqp u eps}
\E|u^{\varepsilon}_T|_{W^m_p}^p
+
\E\Big(\int_0^T|u_t^{\varepsilon}|_{W^m_p}^{r}\,dt \Big)^{p/r}
\leq \E|u^{\varepsilon}_T|_{W^m_p}^p+T^{p/r}\E\sup_{t\in[0,T]}
|u_t^{\varepsilon}|_{W^m_p}^p\leq N\E|\psi|_{W^m_p}^p
\end{equation}
for a constant 
$N=N(m,d,p,K,K_\eta,K_\xi,L,T,\lambda, |\bar\xi|_{L_2(\frZ_1)},|\bar\eta_{L_2(\frZ_1)})$ 
independent of $\varepsilon$ 
for all integers $r\geq 1$. Letting $(\varepsilon_n)_{n=1}^\infty$ 
be the sequence from Proposition \ref{proposition compact} (iii), 
we know that 
$$
u_T^{\varepsilon_n}\to u_T\quad\text{weakly in 
$\bL_p(\cF_T)$ and}\quad u^{\varepsilon_n}\to u
\quad\text{weakly in $\bL_{p,r}$ for integers $r\geq 1$ as $n\to\infty$}
$$
where $u$ is the unique $L_p$-solution to \eqref{equdZ} and, 
if necessary by passing to a subsequence,
$$
u^{\varepsilon_n}_T\to u_T\quad\text{weakly in $\bW^m_p(\cF_T)$ and}
\quad
u^{\varepsilon_n}\to u\quad\text{weakly in $\bW^m_{p,r}$ for integers $r\geq 1$.}
$$
Letting $r\to\infty$ in \eqref{equ E Wqp u eps} yields
$$
\E|u_T|_{W^m_p}^p+\E\esssup_{t\in [0,T]}|u_t|_{W^m_p}^p<N\E|\psi|_{W^m_p}^p.
$$
By Lemma \ref{lemma weakly cadlag} $u$ 
is weakly cadlag as $W^m_p$-valued process. 
Thus we can replace the essential supremum above 
by the supremum to obtain \eqref{E sup Wmp compact}. 
By Proposition \ref{proposition compact} (ii) we also have 
that almost surely $u_t(x)=0$ for $dx$-almost every  
$x\in\{x\in\bR^d:|x|\geq \bar R\}$ for every $t\in[0,T]$ 
for a constant $\bar{R}=\bar{R}(R,K,K_0, K_{\xi}, K_{\eta})$.
This finishes the proof.
\end{proof}
\begin{corollary}                                                                               \label{corollary 1.3.4.22}
Let Assumptions  \ref{assumption SDE}, 
\ref{assumption p}, \ref{assumption estimates} and  \ref{assumption SDE2}
hold with an integer $m\geq 0$. 
Assume, moreover that  the support condition \eqref{supp_condition}
holds for some $R>0$. 
Then for every $p\geq2$ there is a linear operator $\bS$ defined 
on $\bW^m_p$ 
such that $\bS\psi$ admits a $P\otimes dt$-modification 
$u=(u_t)_{t\in[0,T]}$ which is a $W^m_p$-solution 
to equation \eqref{equdZ} for every $\psi\in\bW^m_p$, with initial condition $u_0=\psi$,  
and 
\begin{equation}                                                                                                   \label{30.9.21.5}
\E\sup_{t\in [0,T]}|u_t|_{W^m_p}^p\leq N\E|\psi|_{W^m_p}^p
\end{equation}
with a constant 
$N=
N(m,d,p,T,K, K_{\xi}, K_{\eta}, L,\lambda,|\bar{\xi}|_{L_2},|\bar{\eta}|_{L_2})$. 
Moreover, if $\psi\in\bW^m_p$ such that almost surely $\psi(x)=0$ for $|x|\geq R$,  
then almost surely $u_t(x)=0$ for $|x|\geq \bar R$ for $t\in[0,T]$ for a constant 
$\bar{R}=\bar{R}(R,K,K_0, K_{\xi}, K_{\eta})$.
\end{corollary}
\begin{proof}
By Corollary 6.5 in \cite{GG2} we know that there exist 
linear operators $\bS$ and $\bS_T$ on $\bL_p$ 
such that $\bS\psi$ admits a $P\otimes dt$-modification 
$u=(u_t)_{t\in [0,T]}$ that is an $L_p$-solution to \eqref{equdZ} 
such that $u_T=\bS_T\psi$ satisfies equation \eqref{equdZ} 
for each $\vp\in C_0^\infty$ almost surely with $u_T$ 
in place of $u_t$ and $t:=T$. By an abuse of notation 
we refer to this stochastic modification $u$ whenever 
we write $\bS\psi$ in the following. It remains 
to show that if $\psi\in\bW^m_p$, then $u$  is 
in particular a $W^m_p$-solution to \eqref{equdZ}, 
i.e. it is weakly cadlag as $W^m_p$-valued process.\newline
If $p$ is an even integer, then this follows from 
Lemma \ref{lemma compact}. 
Assume $p$ is not an even integer. Then let $p_0$ be 
the greatest even integer such that $p_0\leq p$ and let $p_1$ 
be the smallest even integer such that $p\leq p_1$. 
By Lemma \ref{lemma compact}, in particular \eqref{E sup Wmp compact}, 
we get that 
\begin{equation}
\label{equ norm Wmp_i}
|\bS_T\psi|_{\bW^m_{p_i}}+|\bS\psi|_{\bW^m_{p_i,r}}
\leq N_i|\psi|_{\bW^m_{p_i}}\quad\text{for $i=0,1$}
\end{equation}
for every $r\in[1,\infty)$ and constants $N_i=
N_i(m,d,p_i,T,K, K_{\xi}, K_{\eta}, L,\lambda,|\bar{\xi}|_{L_2},|\bar{\eta}|_{L_2})$, 
$i=0,1$, independent of $r$.
Hence, by a well-known generalization of the Riesz-Thorin interpolation theorem
we also get for all $r\geq 1$,
\begin{equation}
\label{equ norm Wmp p}
|\bS_T\psi|_{\bW^m_{p}}+|\bS\psi|_{\bW^m_{p,r}}
\leq N|\psi|_{\bW^m_{p}}\quad\text{for $i=0,1$},
\end{equation}
for (another) constant $N=
N(m,d,p,T,K, K_{\xi}, K_{\eta}, L,\lambda,|\bar{\xi}|_{L_2},|\bar{\eta}|_{L_2})$. 
Consider a sequence $(\psi^n)_{n=1}^\infty\subset \bB^m_0$ 
such that $\psi^n\to\psi$ in $\bW^m_p$. 
For each $n$, $u^n=\bS\psi^n$  is the unique 
$W^m_{p_i}$-solution to \eqref{equdZ}, $i=0,1$, 
with initial condition $\psi^n$.
By virtue of \eqref{equ norm Wmp p}, 
using that $|\psi^n-\psi|_{\bW^m_p}\to 0$, as $n\to\infty$  
we know that also
$$
u^n\to u\quad\text{weakly in $\bW^m_{p,r}$ for every integer $r\geq 1$ and}
\quad 
u_T^n\to  u_T
\quad
\text{weakly in $\bW^m_p(\cF_T)$,}
$$
where $u = \bS\psi$ is the unique $L_p$-solution introduced 
in the beginning of the proof, satisfying \eqref{equ norm Wmp p}. 
To see that $u$ is  weakly cadlag as $W^m_p$-valued process, 
note that by letting $r\to\infty$ in \eqref{equ norm Wmp p} 
or $\bS\psi=u$ and $\bS_T\psi = u_T$ yields
$$
\E|u_T|_{W_p^m}^p +   \E\esssup_{t\in [0,T]}|u_t|_{W^m_p}^p\leq N\E|\psi|_{W^m_p}^p,
$$
for (another) constant $N=
N(m,d,p,T,K, K_{\xi}, K_{\eta}, L,\lambda,|\bar{\xi}|_{L_2},|\bar{\eta}|_{L_2})$.
By Lemma \ref{lemma weakly cadlag} we then know that $u$ 
is weakly cadlag as $W^m_p$-valued process. 
Thus we can replace the essential supremum above with the supremum, 
to obtain \eqref{30.9.21.5}. To prove the claim 
about the support of $u$, note that if $\psi(x)=0$ for $|x|\geq R$, 
for a constant $R$, and $\psi^n\to\psi$ in $\bW^m_p$, 
then for sufficiently large $n$ we have $\psi^n(x)=0$ for $|x|\geq 2R$. 
By Proposition \ref{proposition compact} (ii) thus also $u_t^n(x)=0$ for $dx$-almost every  
$x\in\{x\in\bR^d:|x|\geq \bar R\}$ for every $t\in[0,T]$ and $n$ sufficiently large, 
for a constant $\bar{R}=\bar{R}(R,K,K_0, K_{\xi}, K_{\eta})$. 
This is clearly preserved in the limit as $n\to\infty$.
This finishes the proof.
\end{proof}

\mysection{Proof of Theorem \ref{theorem regularity}}
\label{sec proof}

Let $\chi$ be a smooth function on $\bR$ 
such that $\chi(r)=1$ for $r\in[-1,1]$, $\chi(r)=0$ for $|r|\geq2$, 
$\chi(r)\in[0,1]$ and $\sum_{k=1}^{m+2}|d^k/(dr^k)\chi(r)|\leq C$ 
for all $r\in\bR$ and a real nonnegative constant $C$. For integers $n\geq1$ 
we define the function $\chi_n$ by $\chi_n(x)=\chi(|x|/n)$, $x\in\bR^d$.

\begin{lemma}                                                                          
\label{lemma truncated coeffs}
(i) Let $b=(b^i)$ be an $\bR^d$-valued function on $\bR^m$ 
such that for a constant $L$ 
\begin{equation}                                                                        \label{8.5.2}
|b(v)-b(z)|\leq L|v-z| \quad\text{for all $v,z\in\bR^m$}. 
\end{equation}                                                                           
Then for $b_n(z)=\chi(|z|/n)b(z)$, $z\in\bR^m$,  
for integers $n\geq1$ we have 
\begin{equation}                                                                          \label{8.5.3}
|b_n(z)|\leq 2nL+|b(0)|,
\quad 
|b_n(v)-b_n(z)|\leq (5L+2|b(0)|)|v-z| \quad\text{for all $v,z\in\bR^m$}. 
\end{equation}
(ii) Let additionally to (i) the function $b$ satisfy
\begin{equation}
\label{Db}
\sum_{k=1}^m |D^k b|\leq M,
\end{equation} 
for a constant $M>0$. Then $b_n$ satisfies \eqref{Db} 
in place of $b$ with $M' = M'(M,C,m,|b(0)|)$ in place of $M$.
\end{lemma}
\begin{proof}
The proof of (i) is Lemma 7.2 in \cite{GG2}. The proof of (ii) is an easy exercise.
\end{proof}

To preserve the diffeomorphic property of the mappings 
\begin{equation}
\label{tau}
\tau^\eta_{t,\frz_0,\theta}(x) 
= x+\theta\eta_t(x,\frz_0)\quad\text{and}\quad \tau^\xi_{t,\frz_1,\theta}(x)
 = x+\theta\xi_t(x,\frz_1) 
\end{equation}
(for all $\omega\in\Omega$, $t\in [0,T]$, $\theta\in [0,1]$ and $\frz_i\in\frZ_i$, $i=0,1$) 
as a function of $x\in\bR^d$, when the functions $\xi$ and $\eta$ are truncated, 
we introduce, for each 
fixed $R>0$ and $\epsilon>0$, the function 
$\kappa^R_\epsilon$ defined on $\bR^d$ by 
\begin{equation}                                                                             \label{Sandy function}
\kappa^R_\epsilon(x)=\int_{\bR^d}\phi^R_{\varepsilon}(x-y)k(y)\,dy, 
\quad
\phi^R_{\varepsilon}(x)\begin{cases}
1,& |x|\leq R+1,\\
1+\epsilon\log\big(\tfrac{R+1}{|x|}\big),& R+1<|x|< (R+1)e^{1/\epsilon},\\
0,& |x|\geq (R+1)e^{1/\epsilon}, 
\end{cases}
\end{equation}
where $k$ is a nonnegative $C^{\infty}$ 
mapping on $\bR^d$ with support in $\{x\in\bR^d:|x|\leq 1\}$. 
\begin{lemma}                                                                               \label{lemma Sandy}
Let $\xi:\bR^d\mapsto\bR^d$ be such that for a constant $L\geq 1$ 
and for every $\theta\in [0,1]$ the function 
$\tau_\theta(x)=x+\theta\xi(x)$ is $L$-biLipschitz, i.e. 
\begin{equation}
\label{tth}L^{-1}|x-y|\leq|\tau_\theta(x)-\tau_\theta(y)|\leq L|x-y|
\end{equation}
for all $x,y\in\R^d$. Then for any $M>L$ and any $R>0$ 
there is an $\epsilon=\epsilon(L,M,R,|\xi(0)|)>0$ such that 
with $\kappa^R:=\kappa^R_\epsilon$ the function $\xi^R:=\kappa^R\xi$ 
vanishes for $|x|\geq \bar R$ for a constant 
$\bar{R}=\bar R(L,M,R,|\xi(0)|)>R$, $|\xi^R|$ is bounded by a constant 
$N=N(L,M,R, |\xi(0)|)$,  and for every $\theta\in [0,1]$ the mapping
$$
\tau^R_\theta (x) = x + \theta\xi^R(x) , \quad x\in\bR^d
$$
is $M$-biLipschitz.
\end{lemma}
\begin{proof}
This is Lemma 7.3 in \cite{GG2}. 
\end{proof}

We summarize the results of Lemmas 7.1, 7.2 and Remark 7.1 in \cite{GG2} 
in the following lemma.
For that purpose, define the functions $b^n=(b^{ni}(t,z))$, 
$B^n = (B^{nj}(t,z))$, 
$\sigma^n=(\sigma^{nij}(t,z))$, 
$\eta^n=(\eta^{ni}(t,z,\frz_0))$ and $\xi^n=(\xi^{ni}(t,z,\frz_1))$ by 
\begin{equation}
\label{truncated coeff}
(b^n, B^n,\sigma^n, \rho^n) = (b, B,\sigma, \rho)\chi_n,\quad 
(\eta^n,\xi^n)=(\eta,\xi) \bar\chi_n 
\end{equation}
for every integer $n\geq1$, 
where $\chi_n$ and $\bar\chi_n$ are functions on $\bR^{d+d'}$ 
defined by $\chi_n(z)=\chi(|z|/n)$ and $\bar\chi_n(x,y)=\kappa^R(|x|/n)\chi(|y|/n)$ for 
$z=(x,y)\in\bR^{d+d'}$, with $\chi$ used in Lemma \ref{lemma truncated coeffs}
and with $\kappa^R=\kappa^R_\varepsilon$ 
from Lemma \ref{lemma Sandy}, 
such that, by the $L$-biLipschitzness of the mappings in \eqref{tau}, the mappings 
\begin{equation*}
\tau^{\eta^n}_{t,\frz_0,\theta}(x) = x+\theta\eta^n_t(x,\frz_0)
\quad
\text{and}\quad \tau^{\xi^n}_{t,\frz_1,\theta}(x) = x+\theta\xi^n_t(x,\frz_1) 
\end{equation*}
are biLipschitz (for all $\omega\in\Omega$, $t\in [0,T]$, 
$\theta\in [0,1]$ and $\frz_i\in\frZ_i$, $i=0,1$).

\begin{lemma}                                                         \label{lemma convergence 2}
Let Assumptions \ref{assumption SDE}, \ref{assumption p} 
and \ref{assumption SDE2} hold. 
If $K_1\neq 0$ in Assumption \ref{assumption SDE} (ii), 
then let additionally Assumption \ref{assumption nu}
for some $r>2$ hold. 
Assume the initial conditional density 
$\pi_0 =P(X_0\in dx|\cF^Y_0)/dx$ exists (a.s.) and satisfies 
$\E|\pi_0|_{W^m_p}^p<\infty$ for some $p\geq 2$ and integer $m\geq 0$.
Then there exist sequences
$$
(X_0^n)_{n=1}^\infty, ((X_t^n,Y_t^n)_{t\in [0,T]})_{n=1}^\infty,
\quad
\text{as well as}\quad(\pi_0^n)_{n=1}^\infty
\quad \text{and}\quad ((\pi_t^n)_{t\in [0,T]})_{n=1}^\infty
$$
such that the following are satisfied:\newline
(i) For each $n\geq 1$ the coefficients 
$b^n,B^n,\sigma^n,\rho^n,\xi^n$ and $\eta^n$, 
defined in \eqref{truncated coeff}, satisfy Assumptions \ref{assumption SDE} 
and \ref{assumption p} with $K_1=K_2=0$ and constants 
$K_0'=K'_0(n, K, K_0, K_1,K_{\xi},K_{\eta})$ 
 and $L'=L'(K,K_0, K_1, L,K_{\xi},K_{\eta})$ in place of 
 $K_0$ and $L$, Assumption \ref{assumption SDE2} 
 with a constant $K'=K'(K_0,K_1)$ in place of $L$, as well as
 Assumption \ref{assumption estimates} with 
 $\lambda'=\lambda'(K_0,K_1,K_\xi,K_\eta,\lambda)$ 
 in place of $\lambda$. Moreover, for each $n\geq 1$ they 
 satisfy the support condition \eqref{supp_condition} of 
 Lemma \ref{lemma compact} for some $R=R(n)$.\newline
(ii) For each $n\geq 1$ the random variable $X_0^n$ is $\cF_0$-measurable and satisfies 
$$
\lim_{n\to\infty} X_0^n = X_0\,\, ,\omega \in \Omega,\quad\text{and}
\quad \E|X_0^n|^r \leq N(1+\E|X_0|^r)
$$
for $r\geq 1$ with a constant $N$ independent of $n$.\newline
(iii) $Z_t^n=(X_t^n,Y_t^n)$ is the solution to \eqref{system_1} 
with the coefficients $b^n,B^n,\sigma^n,\rho^n,\xi^n$ and $\eta^n$ 
in place of $b,B,\sigma,\rho,\xi$ and $\eta$, respectively, 
and with initial condition $Z_0^n = (X_0^n,Y_0)$.\newline
(iv) For each $n\geq 1$ we have $\pi_0^n 
= P(X_0^n\in dx|\cF^{Y^n}_0)/dx$, $\pi_0^n(x)=0$ for $|x|\geq n+1$ and 
$$
\lim_{n\to\infty}|\pi^n_0-\pi_0|_{\bW^m_p}=0,
$$
where $\pi_0 = P(X_0\in dx|\cF^{Y^n}_0)/dx$. \newline
(v) For each $n\geq 1$ there exists an $L_r$-solution $u^n$ to \eqref{equdZ}, $r=2,p$, such that $u^n$ is the unnormalised conditional density of $X^n$ given $Y^n$, almost surely
$$
u_t^n(x)=0 \quad\text{for $dx$-a.e. $x\in\{x\in\bR^d:|x|\geq\bar R\}$ for all $t\in[0,T]$} 
$$
with a constant $\bar R=\bar R(n,K,K_0,K_\xi,K_\eta)$ and 
\begin{equation}
\label{28.5.22}
\E\sup_{t\in [0,T]}|u^n_t|_{L_p}^p\leq N\E|\pi_0^n|^p_{L_p}
\end{equation} 
with a constant 
$N=N(d, d',K, L, K_{\xi}, K_{\eta}, T, p,\lambda,|\bar\xi|_{L_2}, |\bar\eta|_{L_2})$. Moreover,
$$
u^n\rightarrow  u
\quad
\text{weakly in $\bL_{r,q}$ for $r=p,2$ and all integers $q>1$},
$$
where $u$ is the unnormalised conditional density 
of $X$ given $Y$, satisfying \eqref{28.5.22} 
with the same constant $N$ and $(u,\pi_0)$ 
in place of $(u^n,\pi_0^n)$. \newline
(vi) Consequently, for each $n\geq 1$ and $t\in [0,T]$ we have 
$$
\pi^n_t(x) = P(X_t^n\in dx|\cF^{Y^n}_t)/dx 
= u^n_t(x){^o\!\gamma_t^n},\quad\text{almost surely},
$$
as well as
$$
\pi_t(x) = P(X_t\in dx|\cF^{Y}_t)/dx 
= u_t(x){^o\!\gamma_t},\quad\text{almost surely},
$$
where ${^o\!\gamma_t^n}$ and ${^o\!\gamma_t}$ are cadlag 
 positive normalising process, adapted to $\cF^{Y^n}_t$ 
 and $\cF^Y_t$, respectively.
\end{lemma}
\begin{proof}
This is Corollary 7.4 in \cite{GG2}.
\end{proof}

Now we are in the position to prove our main result.

\begin{proof}[Proof of Theorem \ref{theorem regularity}]

\textbf{Step I.}
Assume first that the support condition \eqref{supp_condition} 
holds with some $R>0$ and that the initial conditional 
density $\pi_0$ is such that $\pi_0(x)=0$ for $|x|\geq R$.
By Corollary \ref{corollary 1.3.4.22} we know that 
there exists a $W^m_p$-solution $(u_t)_{t\in [0,T]}$ 
to \eqref{equdZ} with initial condition $\pi_0$, satisfying
\begin{equation}                                                                                                  
\label{13.6.22.1}
\E\sup_{t\in [0,T]}|u_t|_{W^m_p}^p\leq N\E|\pi_0|_{W^m_p}^p
\end{equation}
with a constant 
$N=
N(m,d,p,T,K, K_{\xi}, K_{\eta}, L,\lambda,|\bar{\xi}|_{L_2},|\bar{\eta}|_{L_2})$. 
Moreover, we have $u_t=0$ for $|x|\geq \bar{R}$, 
for a constant $\bar{R}=\bar{R}(R,K,K_0,K_1,K_\xi,K_\eta)$, and hence clearly 
$$
\sup_{t\in [0,T]}|u_t|_{L_1}
\leq \bar{R}^{d/q}\sup_{t\in [0,T]}|u_t|_{L_p}
\quad
\text{and}
\quad
\sup_{t\in [0,T]}\int_{\bR^d}|y|^2|u_t(y)|\,dy<\infty\text{ (a.s.),}
$$
with $q=p/(p-1)$.
Since also $\pi_0 = P(X_0\in dx| \cF^Y_0)/dx\in \bL_1$, 
then in particular $\pi_0\in \bL_2$ and hence  
\begin{equation}
\label{13.6.22.2}
\E\sup_{t\in [0,T]}|u_t|_{L_2}^2\leq N\E|\pi_0|_{L_2}^2,
\end{equation}
with a constant 
$N=
N(d,p,T,K, K_{\xi}, K_{\eta}, L,\lambda,|\bar{\xi}|_{L_2},|\bar{\eta}|_{L_2})$
By Lemma \ref{lemma compact} $u$ is the unique $L_2$-solution 
and therefore by Theorem \ref{theorem Lp}, $u$ is in particular 
the unnormalised conditional density, i.e., $u_t = d\mu_t/dx$
for all $t\in [0,T]$,  almost surely, with $\mu$ the unnormalised 
conditional distribution from Theorem \ref{theorem Z1}. Thus also for each $t\in [0,T]$,
$$
\pi_t = P(X_t\in dx|\cF^Y_t)/dx = u_t {^o\!\gamma}_t,\quad\text{almost surely,}
$$
where $ {^o\!\gamma}_t$ is the $\cF^Y_t$-optional projection of 
the normalizing process $\gamma$ under $P$ introduced in \eqref{o gamma}.\newline
\textbf{Step II.}
Finally, we dispense with the assumption that the coefficients 
and the initial condition are compactly supported.
Define the functions $b_n,B_n,\sigma_n,\rho_n,\xi_n$ and $\eta_n$ 
as in \eqref{truncated coeff}. Note that by Lemma \ref{lemma convergence 2} 
the truncated coefficients satisfy Assumptions \ref{assumption SDE} 
and \ref{assumption p} with $K_1=K_2=0$ and constants 
$K_0'=K'_0(n, K, K_0, K_1,K_{\xi},K_{\eta})$ 
 and $L'=L'(K,K_0, K_1, L,K_{\xi},K_{\eta})$ in place of $K_0$ and $L$, 
 the coefficients $b_n,B_n,\sigma_n,\rho_n$ satisfy 
 Assumption \ref{assumption SDE2} with a constant $K'=K'(m,K_0,K_1)$ 
 in place of $L$, and moreover that the coefficients $\eta_n$ and $\xi_n$ 
 satisfy Assumption \ref{assumption SDE2} with $K'\bar\eta$ and $K'\bar\xi$ 
 instead of $\bar{\eta}$ and $\bar{\xi}$ respectively. 
 Furthermore, by Lemma \ref{lemma Sandy}, 
 for each $n\geq 1$ the coefficients $\eta_n$ and $\xi_n$ satisfy
Assumption \ref{assumption estimates} with a constant 
$\lambda'=\lambda'(\lambda,K_0,K_1,K_\eta,K_\xi)$ in place of $\lambda$. 
Note that $K'$, $L'$ and $\lambda'$ do not depend on $n$.  
Moreover, for each $n\geq 1$ they satisfy 
the support condition \eqref{supp_condition} 
of Lemma \ref{lemma compact} for some $R=R(n)>0$.
By assumption, $\pi_0=P(X_0\in dx|\cF_0^Y)/dx$ exists almost surely  
and $\E|\pi_0|^p_{W^m_p}<\infty$. 
Then let $(X^n_0)_{n=1}^{\infty}$ and  
$(\pi^n_0)_{n=1}^\infty\subset \bW^m_p$ be the sequences  
from Lemma \ref{lemma convergence 2} such that
\begin{equation}
\label{13.6.22.3}
\lim_{n\to\infty}|\pi^n_0-\pi_0|_{\bW^m_p}=0,
\end{equation}
$\pi_0^n(x)=0$ for  $|x|\geq R(n)$ 
and $\pi^n_0 = P(X_0^n\in dx|\cF^Y_0)/dx$ (a.s.), 
where $(X_0^n,Y_0)$ is the initial condition 
to the system \eqref{system_1}, and $(R(n))_{n=1}^{\infty}$ 
is the sequence of positive numbers from the support condition for 
the coefficients $(\sigma^n,...,\xi^n)$.
By Step I we know that there exists a $W^m_p$-solution 
$(u_t)_{t\in [0,T]}$ to \eqref{equdZ} with initial condition $\pi^n_0$, 
which is the unnormalized conditional density of $X^n=(X^n_t)_{t\in [0,T]}$ 
given $Y^n = (Y^n_t)_{t\in [0,T]}$, where $Z^n=(X^n,Y^n)$ 
is the solution to \eqref{system_1} with initial condition $(X_0^n, Y_0)$.
By Lemma \ref{lemma convergence 2} $(v)$ we know moreover that 
$$
u^n\rightarrow  u
\quad
\text{weakly in $\bL_{r,q}$ for $r=p,2$ and all integers $q>1$},
$$
where $u$ is the unnormalised conditional density of $X$ given $Y$ 
from Theorem \ref{theorem Lp}, satisfying 
$$
\E\sup_{t\in [0,T]}|u_t|_{L_2}^2\leq N\E|\pi_0|_{L_2}^2,
$$
 with a constant 
$N=
N(d,p,T,K, K_{\xi}, K_{\eta}, L,\lambda,|\bar{\xi}|_{L_2},|\bar{\eta}|_{L_2})$ 
independent of $n$.  Moreover, $u$ is an $L_p$-solution to \eqref{equdZ} 
and by Theorem \ref{theorem Lp} (ii), it is the unique $L_2$-solution 
to \eqref{equdZ}. It remains to show that $u$ is also a $W^m_p$-solution 
to \eqref{equdZ}, as well as that it is strongly cadlag as $W^s_p$-valued process, 
for $s\in [0,m)$. To prove the former,
by \eqref{13.6.22.1} together with \eqref{13.6.22.3} 
we get that for $n$ sufficiently large,
\begin{equation}                                                                                                   \label{13.6.22.5}
\E|u_T^n|_{W^m_p}^p + \E\left(\int_0^T|u_t^n|^r_{W^m_p}\,dt \right)^{p/r}
\leq \E|u_T^n|_{W^m_p}^p +T^{p/r} 
\E\sup_{t\in [0,T]}|u_t^n|_{W^m_p}^p\leq 2N\E|\pi_0|_{W^m_p}^p.
\end{equation}
Hence we know that 
$$
u^n_T\to u_T,\quad\text{weakly in $\bW^m_p$ and}
\quad 
u^n\to u
\quad
\text{weakly in $\bW^m_{p,r}$ for any $r> 1$,}
$$
where $u$ satisfies for all $r\geq 1$,
$$
\E|u_T|_{W^m_p}^p 
+ \E\left(\int_0^T|u_t|^r_{W^m_p}\,dt \right)^{p/r}
\leq 2N\E|\pi_0|_{W^m_p}.
$$
Letting $r\to\infty$ above yields
$$
\E|u_T|_{W^m_p}^p 
+ \E\esssup_{t\in [0,T]}|u_t|^p_{W^m_p}
\leq 2N\E|\pi_0|_{W^m_p}.
$$
By Lemma \ref{lemma weakly cadlag} 
we then know that $u$ is weakly cadlag 
as an $W^m_p$-valued process, i.e. it is 
a $W^m_p$-solution to \eqref{equdZ}. 
Clearly, by Lemma \ref{lemma convergence 2}, 
also for each $t\in [0,T]$
$$
\pi_t(x) = P(X_t\in dx|\cF^Y_t)/dx = u_t(x){^o\!\gamma_t},
\quad
\text{almost surely},
$$
with ${^o\!\gamma}$ from Theorem \ref{theorem Lp}. 
We now show that if $m\geq 1$ and $K_1=0$, 
then $u$ is strongly cadlag as $W^s_p$-valued process 
for $s\in [0,m)$. 
To this and first we 
state $u$ is strongly cadlag as an $L_p$-valued 
process.

\begin{proposition}                                                    \label{proposition strongly cadlag}
Let Assumptions \ref{assumption SDE} 
through \ref{assumption SDE2} 
hold with $K_1=0$ and with $m=1$. 
Let $p\geq 2$ and 
let $u=(u_t)_{t\in [0,T]}$ be a $W^1_p$-solution to \eqref{equdZ}.  
Then $u$ is strongly cadlag as an $L_p$-valued process. 
\end{proposition}

\begin{proof}
We apply Theorem 2.2 in \cite{GW}. 
In order to do so, we rewrite equation \eqref{equZ} 
into the form used therein. 
Clearly, for $v\in W^1_p$ and $\varphi\in C^{\infty}_0$ 
we have 
$$
(v,\cM^k\varphi)=(\cM^{k\ast}v,\varphi)
\quad 
\text{with $\cM^{*k}v = -D_i(\rho^{ik}_sv) + B_s^kv$, \quad $k=1,2,...,d_1$} 
$$
and 
$$
(v,\tilde{\cL}_s\varphi) 
= -(D_j(a^{ij}_sv),D_i\vp) - (D_i(b^i_sv),\vp) + \beta_s^k(\cM^{*k}_sv,\vp),  
$$
Using Corollary \ref{corollary 1.2.10.22} with 
$\eta_t(\cdot,\frz)$ and $\xi_t(\cdot,\frz)$ in place of $\zeta$ 
we can see that 
$$
(v, J_t^{\eta}\varphi)=(K^{\eta_t}_iv,D_i\varphi), 
\quad
(v, J_t^{\xi}\varphi)=(K^{\xi_t}_iv,D_i\varphi)
\quad
\text{and}
\quad
(v,I_t^{\xi}\varphi)=(I_t^{\xi\ast}v,\varphi), 
$$
for $v\in W^1_p$ and $\varphi\in C_0^{\infty}$, 
where $K_{i}^{\eta_t}v$ and $K^{\xi_t}_iv$ are defined as $K_i^{\zeta}v$ 
in Corollary \ref{corollary 1.2.10.22} with $\zeta$ 
replaced with $\eta_t(\frz_0)$ and $\xi_t(\frz_1)$, respectively, (for $t\in [0,T]$, $\omega\in\Omega$, $\frz_0\in\frZ_0$, 
$\frz_1\in\frZ_1$, $i=1,\dots,d$),  and $I_t^{\xi\ast}$ is defined as $I^{\zeta\ast}$,  
with $\zeta$ replaced by $\xi_t(\frz_1)$. 
Thus for every $\vp\in C_0^\infty$ almost surely
$$
(u_t,\vp) = (\psi,\vp) - \int_0^t (D_j(a^{ij}_su_s),D_i\vp)\,ds 
- \int_0^t (D_i(b^i_su_s)+\beta^k_s\cM^{*k}_su_s,\vp)\,ds 
$$
$$
+ \int_0^t(\cM^{*k}_su_s,\vp)\,dV^k_s
+\int_0^t \int_{\frZ_0}(K_i^{\eta_s}u_s,D_i\vp)\,\nu_0(d\frz)ds 
$$
\begin{equation}                                                                      \label{9.2.10.22}
+\int_0^t\int_{\frZ_1}(K_i^{\xi_s}u_s,D_i\vp)\,\nu_1(d\frz)ds 
+\int_0^t\int_{\frZ_1}(I^{\xi*}_su_s,\vp)\,\Nte(d\frz,ds)
\end{equation}
for all $t\in [0,T]$.
It is easy to see that almost surely
$$
\int_0^T|D_i(a^{ij}_su_s)|_{L_p}^p\,ds<\infty,  
\quad
\int_0^T |D_i(b^i_su_s)+\beta^k_s\cM^{*k}_su_s|^p_{L_p}\,ds<\infty,   
$$
\begin{equation}                                                                \label{5.2.10.22}
\int_0^T\int_{\bR^d}\big(\sum_k|(\cM^{k*}_s u_s)(x)|^2\big)^{p/2}\,dxds<\infty. 
\end{equation}
By estimates \eqref{3.2.10.22} and \eqref{4.2.10.22}, for all $x\in\bR^d$ we have  
\begin{equation*}                                                        
|I^{\xi\ast}u(x)|
\leq N\bar\xi
\int_0^1|u(\tau^{-1}_{\theta\xi}(x))|+|(Du)(\tau^{-1}_{\theta\xi}(x))|\,d\theta, 
\end{equation*}
\begin{equation*}                                                         
|K_i^{\xi}u(x)|
\leq N\bar\xi^2
\int_0^1|u(\tau^{-1}_{\theta\xi}(x))|+|(Du)(\tau^{-1}_{\theta\xi}(x))|\,d\theta, 
\end{equation*}
\begin{equation*}                                                         
|K_i^{\eta}u(x)|
\leq N\bar\eta^2
\int_0^1|u(\tau^{-1}_{\theta\eta}(x))|+|(Du)(\tau^{-1}_{\theta\eta}(x))|\,d\theta
\end{equation*}
for every $\omega\in\Omega$, $s\in[0,T]$, $\frz_i\in\frZ_i$ (i=0,1), 
suppressed in these estimates, with a constant $N=N(d,\lambda,L,K_\eta,K_\xi)$
and with the $C^2$-diffeomorphisms
$$
\tau_{\theta\eta}(x)=x + \theta\eta(x)\quad\text{and}\quad
\tau_{\theta\xi}(x) = x + \theta\xi(x).
$$
Hence by Jensen's inequality, Fubini's theorem and Minkovski's inequality 
we get 
\begin{equation}                                                                    \label{6.2.10.22}
\int_0^T\int_{\bR^d}
\Big(\int_{\frZ_1}|I_s^{\xi\ast}u_s(x)|^2\nu_1(d\frz)\Big)^{p/2}\,dx\,ds
\leq 
N|\bar\xi|_{L_2(\frZ_1)}^p\int_0^T|u_s|^p_{W^1_p}ds<\infty \,{\rm(a.s.)}
\end{equation}
with a constant $N=N(p,d,\lambda, L,K_\eta,K_\xi)$. By Jensen's inequality and Fubini's 
theorem we obtain 
\begin{equation}
\int_0^T\int_{\bR^d}\int_{\frZ_1}|I_s^{\xi\ast}u_s(x)|^p\,\nu_1(d\frz)dx\,ds
\leq N|\bar\xi|^2_{L_2(\frZ_1)}\int_0^T|u_s|^p_{W^1_p}\,ds
<\infty \,{\rm(a.s.)},
\end{equation}
and for every $i=1,2,...,d$
\begin{equation}                                                          \label{7.2.10.22}
\int_0^T\int_{\bR^d}\int_{\frZ_1}|K_i^{\xi_s}u_s(x)|^p\,\nu_1(d\frz)dx\,ds
\leq N|\bar\xi|^2_{L_2(\frZ_1)}\int_0^T|u_s|^p_{W^1_p}\,ds
<\infty \,{\rm(a.s.)},
\end{equation}
\begin{equation}                                                             \label{8.2.10.22}
\int_0^T\int_{\bR^d}\int_{\frZ_1}|K_i^{\eta_s}u_s(x)|^p\,\nu_1(d\frz)dx\,ds
\leq N|\bar\xi|^2_{L_2(\frZ_1)}\int_0^T|u_s|^p_{W^1_p}\,ds
<\infty \,{\rm(a.s.)}
\end{equation}
with a constant $N=N(p,d,\lambda, L,K_\eta,K_\xi)$. 
Hence, by virtue of Theorem 2.2 in \cite{GW} we get from 
equation \eqref{9.2.10.22}, taking into account \eqref{5.2.10.22} 
through \eqref{8.2.10.22}, that $(u_t)_{t\in[0,T]}$ is 
strongly cadlag as an $L_p$-valued process.  
\end{proof}

By the above proposition $u$ is a strongly cadlag 
$L_p$-valued process, as well as weakly cadlag as an 
$W^m_p$-valued process. By interpolation we then 
have a constant $N=N(d,m,s,p)$ such that 
$$
|u_t-u_{t_n}|_{W^s_p}
\leq N |u_t-u_{t_n}|_{W^m_p}|u_t-u_{t_n}|_{L_p}
\leq 2N\zeta |u_t-u_{t_n}|_{L_p}, 
$$
$$
|u_{r_n}-u_{r-}|_{W^s_p}
\leq N |u_{r_n}-u_{r-}|_{W^m_p}|u_{r_n}-u_{r-}|_{L_p}
\leq 2N \zeta|u_{r_n}-u_{r-}|_{L_p}
$$
for any $t\in[0,T)$, $r\in(0,T]$, any strictly decreasing sequences $t_n\to t$  
and strictly increasing sequences $r_n\to r$ with $r_n,t_n\in(0,T)$, 
where $u_{r-}$ 
denotes the weak limit in $W^m_p$ of $u$ at $r$ from the left, and 
$\zeta:=\sup_{t\in [0,T]}|u_t|_{W^m_p}<\infty$ (a.s.). 
Letting here $n\to\infty$ we finish the proof.
\end{proof}

{\bf Acknowledgements.} The authors are very grateful to Nicolai Krylov, 
whose comments and suggestions greatly improved the presentation of the present article.


\begin{thebibliography}{mm}

\bibitem{B2014} S. Blackwood, L\'evy processes and filtering theory, 
Dissertation, University of Sheffield, 2014.

\bibitem{CF2022} A. Calvia and G. Ferrari, 
{Nonlinear Filtering of Partially Observed Systems Arising in Singular Stochastic Optimal Control}, 
Applied Mathematics \& Optimization 85.2 (2022), 1-43.

\bibitem{DKS} K. A. Dareiotis, C. Kumar and S. Sabanis, 
{On tamed Euler approximations of SDEs driven by L\'{e}vy noise
with applications to delay equations}, SIAM Journal on Numerical Analysis (2016)

\bibitem{DGW} M. De-L\'eon Contreras, 
I. Gy{\"o}ngy and S. Wu, {On solvability of integro-differential equations}, 
Potential Anal. 55 (2021), no. 3, 443-475.

\bibitem{GG1} F. Germ and I. Gy{\"o}ngy, {On partially observed jump diffusions I. The filtering equations},  arXiv:2205.08286, 2022

\bibitem{GG2} F. Germ and I. Gy{\"o}ngy, {On partially observed jump diffusions II. The filtering density}, arXiv:2205.14534, 2022

\bibitem{GW} I. Gy{\"o}ngy and S. Wu, {It\^{o}'s formula for jump processes in ${L}^p$-spaces}, Stochastic processes and their applications, 2021.

\bibitem{K1999} N.V. Krylov, An analytic approach to SPDEs, 
Stochastic Partial Differential Equations: 
Six Perspectives, Mathematical Surveys and Monographs 64 (1999), 185-242.

\bibitem{K2009} N. V. Krylov, {\em On divergence form SPDEs with VMO coefficients}, 
SIAM J. Math. Anal. 40 (2009), no. 6, 2262-2285.

\bibitem{K2010a} N. V. Krylov, {On divergence form SPDEs with growing coefficients 
in $W^1_2$ spaces without weights}, SIAM J. Math. Anal. 42 (2010), 609-633.

\bibitem{K2010} {N. V. Krylov, Kalman-Bucy filter and SPDEs with growing 
lower-order coefficients in $W^1_p$ spaces 
without weights}, {Illinois Journal of Mathematics} 54.3 (2010), 1069-1114. 

\bibitem{K2011} N. V. Krylov, Filtering equations for partially observable diffusion processes 
with Lipschitz continuous coefficients. The Oxford handbook of nonlinear filtering, 169-194, 
Oxford Univ. Press, Oxford, 2011. 

\bibitem{K1978} N.V. Krylov and B.L. Rozovskii, 
{On conditional distributions of diffusion processes}, 
Math. USSR Izv. 12 (1978), 336-356. 

\bibitem{KX1999} T.G. Kurtz and J. Xiong, 
{Particle representations for a class of nonlinear SPDEs}, 
Stochastic Processes and their Applications 83 (1999).

\bibitem{KO} T.G. Kurtz and D.L. Ocone, 
{Unique characterization of conditional distributions in nonlinear filtering}, Annals of Probability (1988).

\bibitem{MPX2019} V. Maroulas, X. Pan and J. Xiong, 
{Large deviations for the optimal filter of nonlinear
dynamical systems driven by L\'evy noise}, 
Stochastic Processes and their Applications 130 (2020), 203–231.

\bibitem{R1980} B.L. Rozovskii, On conditional distributions of degenerate diffusion processes, 
Theory of Probability \& its Applications 25.1 (1980), 147-151.

\bibitem{QD2015} H. Qiao and J. Duan, 
{Nonlinear filtering of stochastic dynamical systems with {L}\'evy noises}, 
Advances in Applied Probability 47-3 (2015). 

\bibitem{Q2021} H. Qiao, 
{Nonlinear filtering of stochastic differential equations driven by correlated L\'{e}vy noises}, 
Stochastics (2021).

\end{thebibliography}
\end{document}